\documentclass[sn-mathphys-num]{sn-jnl}
\usepackage[T1]{fontenc}
\usepackage{amsmath,amssymb,amsfonts,amsthm,mathtools,mathrsfs}
\usepackage[title]{appendix}
\usepackage{booktabs,enumitem}

\usepackage{manyfoot}
\usepackage{xcolor}
\makeatletter
\newtheoremstyle{thmstyleone}{.5\baselineskip}{.5\baselineskip}%
{\itshape}{}{\bfseries}{}{ }%
{\thmname{#1}\thmnumber{ #2}\thmnote{ (#3)}}%
\newtheoremstyle{thmstyletwo}{.5\baselineskip}{.5\baselineskip}%
{}{}{\bfseries}{}{ }%
{\thmname{#1}\thmnumber{ #2}\thmnote{ (#3)}}%
\newtheoremstyle{thmstylethree}{.5\baselineskip}{.5\baselineskip}%
{}{}{\bfseries}{}{ }%
{\thmname{#1}\thmnumber{ #2}\thmnote{ (#3)}}%
\@ifundefined{color@artcatboxgray}%
{\definecolor{artcatboxgray}{cmyk}{0,0,0,0.30}}{}
\makeatother
\theoremstyle{thmstyleone}%
\newtheorem{proposition}{Proposition}%
\newtheorem{lemma}{Lemma}%
\newtheorem{corollary}{Corollary}%
\makeatletter
\let\c@proposition\c@theorem
\let\c@lemma\c@theorem
\let\c@corollary\c@theorem
\makeatother
\theoremstyle{thmstyletwo}%
\newtheorem{example}{Example}%
\newtheorem{remark}{Remark}%
\theoremstyle{thmstylethree}%
\newtheorem{definition}{Definition}%
\theoremstyle{thmstyleone}%
\usepackage{thmtools,thm-restate}
\providecommand{\R}{\mathbb{R}}
\providecommand{\C}{\mathbb{C}}

\providecommand{\Z}{\mathbb{Z}}
\providecommand{\E}{\mathbb{E}}
\providecommand{\Prob}{\mathbb{P}}
\providecommand{\Var}{\operatorname{Var}}
\providecommand{\Cov}{\operatorname{Cov}}
\providecommand{\KL}{\operatorname{KL}}

\providecommand{\osc}{\operatorname{osc}}
\providecommand{\Rpos}{\mathbb{R}_{>0}}
\providecommand{\Leh}{\mathscr{L}}
\providecommand{\Mellin}{\mathcal{M}}
\providecommand{\dd}{\,\mathrm{d}}
\providecommand{\supp}{\operatorname{supp}}
\providecommand{\esssup}{\operatorname{ess\,sup}}
\providecommand{\essinf}{\operatorname{ess\,inf}}
\providecommand{\eqdef}{\mathrel{\vcentcolon=}}
\providecommand{\xv}{\mathbf{x}}

\providecommand{\pp}{\boldsymbol p}

\begin{document}
	
	\title[Lehmer Transform]{Lehmer Transform}
	
	\author*[1]{\fnm{Masoud} \sur{Ataei}}\email{masoud.ataei@utoronto.ca}
	
	\affil[1]{\normalsize\orgdiv{Department of Mathematical and
			Computational Sciences},\\ \orgname{University of Toronto},
		\country{Canada}}
	
	\abstract{The Lehmer transform encodes a positive dataset as a single analytic
		curve, the ratio of its power sums at consecutive orders read as a
		function of the order. As the order sweeps the real line the curve
		generates every location statistic between the sample extremes, and
		its logarithm is the unit increment of the free energy of an
		exponentially tilted ensemble; the statistical theory of the curve
		flows from that identity. The slope of the curve is a Jeffreys
		divergence, the arithmetic--geometric gap is a relative entropy, and
		multiplicativity over independent factors makes deconvolution of
		multiplicative noise a pointwise division. A quantitative rigidity
		theorem locates the frontier of identifiability at geometric growth
		of the curve, with resonant log-periodic families of distinct laws
		sharing one curve beyond it. The empirical curve is an exact
		re-coordinatization of the order statistic, hence minimal
		sufficient, and its influence function, exact leading bias and limit
		theory are developed in full. At heavy tails the curve undergoes a
		phase transition whose two knees, one unit apart, sit at the tail
		index and its unit shift, a whole-sample diagnostic complementing
		threshold-based estimators.}
	
	
	\maketitle
	
	\clearpage
	
	%=====================================================================
	\section{Introduction}\label{sec:intro}
	%=====================================================================
	
	The classical theory of means treats each summary of positive data
	as a separate functional: the harmonic, geometric, arithmetic and
	contraharmonic means are compared through inequalities, and the
	families interpolating them are studied one order at a time,
	yielding a list rather than a function. When the ratio of power sums
	at consecutive orders is read instead as a single analytic curve in
	its order, the separate statistics become values of one object, the
	inequalities become increments along it, and the tools of analysis
	apply to the summary itself. Such a curve is a statistic-generating
	function. It generates every location statistic between the smallest
	and the largest observation, and it carries more than any fixed
	finite list of its values. What it carries, and how well it can be
	estimated, are questions of statistics.
	
	This paper develops the statistical theory of that curve, the Lehmer
	transform of a dataset. The reading of the Lehmer family of
	means~\cite{lehmer1971,gini1938,hlp1934,bullen2003} as one function
	of its order, the \emph{suddency moment}, was introduced by Ataei
	and Wang~\cite{ataei2022}, where the order served as a dial
	controlling how sharply the statistic responds to sudden excursions
	of a signal; Lehmer aggregation has since served as a trainable
	activation in neural networks~\cite{ataei2025}. Beyond
	monotonicity, a formula for the derivative and a study of the
	inflection points~\cite{sluciak2015}, and a maximum weighted
	likelihood derivation of the family~\cite{ziou2024}, its theory as
	a statistical object is missing. The present work supplies it, from
	one identity: the logarithm of the transform is the unit increment
	of the free energy of an exponentially tilted ensemble.
	
	The objects behind this identity are classical. The tilted measures
	are Esscher transforms~\cite{esscher1932} and the escort
	distributions of nonextensive statistical
	mechanics~\cite{tsallis2009,beck1993,halsey1986}; the derivative of
	the curve, recognized here as a Jeffreys divergence, is displayed
	in~\cite{sluciak2015}; the power sum is a Mellin transform, the
	classical tool for products of random
	variables~\cite{springer1979,epstein1948}. Whether the curve
	determines the law is a moment problem in the sense of Heyde and
	Stoyanov~\cite{heyde1963,stoyanov2013,linstoyanov2017}, and the
	rigidity theory descends from the Bohr--Mollerup theory of
	Gamma-type functional
	equations~\cite{bohr1922,krull1948,webster1997}. On the statistical
	side the transform meets heavy-tail
	inference~\cite{hill1975,resnick2007}, superquantile
	risk~\cite{rockafellar2002} and the R\'enyi--Hill diversity
	spectrum~\cite{renyi1961,hillnumbers1973}.
	
	What does the Lehmer reading add to the Mellin transform? Three
	things. The ratio of consecutive orders is scale-free, so the curve
	depends on the data only through their internal proportions and
	determinacy is posed, correctly, up to scale. The reading replaces
	derivatives by unit differences, and the kernel of that difference
	structure, a log-periodic gauge, is what organizes the inversion
	theory, its rigidity theorems and its counterexamples alike. And the
	family read as one curve repairs a defect of every one of its
	members: no single fixed mean is multiplicative over independent
	factors, while the curve as a whole is.
	
	What this paper contributes is the statistical theory built on these
	objects. The free-energy identity converts the classical
	inequalities between means into exact information identities: the
	slope of the curve is a Jeffreys divergence, equivalently
	unit-averaged Fisher information, and the arithmetic--geometric gap
	is a relative entropy. Multiplicativity makes deconvolution of
	multiplicative noise a pointwise division, stable in relative error.
	A quantitative rigidity theorem locates the frontier of
	identifiability at geometric growth of the curve and bounds the
	residual ambiguity explicitly. A period-group dichotomy separates
	the resonant indeterminacy of Heyde type from that forced by a
	lattice of scales, and resonance and subgeometric growth are proved
	mutually exclusive. A lattice-supported companion of the lognormal
	shares the entire curve, strengthening the classical integer-moment
	indeterminacy of the lognormal to indeterminacy of the whole curve.
	
	A dictionary of closed forms identifies the classical families from
	the shape of the curve alone. An exact-recovery theorem inverts the
	curve on finite samples, and a sufficiency theorem shows the
	empirical curve to be an exact re-coordinatization of the order
	statistic, hence minimal sufficient. The sampling theory is
	developed in full: the influence function, the exact leading bias,
	central and functional limit theorems in independent and mixing
	regimes, finite-sample bounds, and large deviations. It closes with
	a heavy-tail phase diagram whose two knees, at known unit
	separation, sit at the tail index and its unit shift; the knees
	sharpen only at logarithmic resolution, so the diagram is developed
	as a whole-sample diagnostic, with an internal consistency check
	that threshold-based estimators such as Hill's~\cite{hill1975} do
	not possess.
	
	Section~\ref{sec:foundations} sets up the partition function, its
	escort family and the master identity;
	Sections~\ref{sec:divergence} and~\ref{sec:multiplicative} develop
	the divergence calculus and the multiplicative structure;
	Section~\ref{sec:rigidity} settles inversion;
	Section~\ref{sec:characterization} builds the dictionary and proves
	that its entries characterize; Section~\ref{sec:statistics} develops
	the sampling theory; Section~\ref{sec:conclusion} closes with a
	synthesis and open problems. Proofs not given in the body are
	collected in the Appendix.
	
	%=====================================================================
	\section{Foundations}\label{sec:foundations}
	%=====================================================================
	
	For a data vector $\xv=(x_1,\dots,x_n)\in\Rpos^n$ the Lehmer
	transform is the ratio of consecutive power sums,
	\begin{equation}\label{eq:defdata}
		\Leh(s)=\frac{\sum_{i=1}^{n}x_i^{\,s}}
		{\sum_{i=1}^{n}x_i^{\,s-1}} .
	\end{equation}
	As the order $s$ runs over the real line, the statistic sweeps
	monotonically from the smallest observation through the harmonic,
	arithmetic and contraharmonic means to the largest. The transform
	is not primarily a mean, however, but a ratio of a partition
	function at consecutive orders: reading $x^{s}=e^{s\log x}$ as a
	Boltzmann weight at inverse temperature $s$ on the energy $\log x$
	makes $\Leh$ the exponential of a unit increment of the associated
	free energy. Every result below flows from this reframing, so we
	set it up for general measures, which costs nothing and covers the
	empirical, absolutely continuous, spectral and arithmetic cases at
	once.
	
	The parameter $s$ is called the \emph{suddency moment}. It acts as
	a dial controlling how sharply the statistic responds to sudden
	large excursions of the data: large positive values of $s$
	concentrate the ratio on the largest observations, large negative
	values on the smallest. The name also marks a contrast. Where the
	Fourier transform reads a signal against oscillations, and so
	emphasizes its periodicities, the Lehmer transform reads positive
	data against powers, and so emphasizes their sudden movements.
	
	Throughout the paper $\mu$ denotes a positive Borel measure on
	$\Rpos=(0,\infty)$, the two guiding instances being the empirical
	measure $\sum_{i=1}^n\delta_{x_i}$ of a data vector and an absolutely
	continuous law with density $f$; $X$ denotes the identity variable, so
	that $\E_\mu[X]=\int x\dd\mu$; and $\log$ is the natural
	logarithm.
	
	\begin{definition}[Partition function]\label{def:Z}
		Let $\mu$ be a nonzero positive Borel measure on $\Rpos$. The
		\emph{suddency partition function} of $\mu$ is
		\begin{equation}\label{eq:defZ}
			Z_\mu(s)\eqdef\int_{\Rpos}x^{s}\dd\mu(x)\in(0,\infty],
			\qquad s\in\R,
		\end{equation}
		its \emph{moment domain} is
		\begin{equation}\label{eq:defI}
			I_\mu\eqdef\{s\in\R:Z_\mu(s)<\infty\},
		\end{equation}
		its \emph{suddency domain} is
		\begin{equation}\label{eq:defS}
			S_\mu\eqdef I_\mu\cap(I_\mu+1)
			=\{s\in\R:\ s\in I_\mu\ \text{and}\ s-1\in I_\mu\},
		\end{equation}
		and its \emph{free energy} is $K_\mu\eqdef\log Z_\mu$. We call
		$\mu$ \emph{admissible} if $I_\mu\neq\emptyset$.
	\end{definition}
	
	One convention will be used silently: a statement about
	$\Leh_\mu$ that names no domain is asserted on the suddency
	domain, and it carries content only when
	$\mathring S_\mu\neq\emptyset$.
	
	Admissibility is the only standing hypothesis. It is weaker than
	finiteness of $\mu$, and deliberately so: the geometric ensemble
	$\sum_{n\ge0}\delta_{q^{\,n}}$ has
	infinite total mass, as do the arithmetic ensembles to which the
	theory applies equally. Admissibility already forces the
	regularity one needs.
	
	\begin{remark}[$\sigma$-finiteness]\label{rem:sigmafinite}
		For $s\in I_\mu$ and compact $[a,b]\subset\Rpos$ one has the
		bound $\mu([a,b])\le Z_\mu(s)/\min(a^{s},b^{s})<\infty$, so
		an admissible measure is finite on compacts, hence
		$\sigma$-finite.
		It is finite in the ordinary sense exactly when $0\in I_\mu$,
		which holds for empirical measures and probability laws but not
		for the infinite-mass ensembles just mentioned.
	\end{remark}
	
	For the empirical measure $\mu=\sum_{i=1}^n\delta_{x_i}$ of a data
	vector one has $Z_\mu(s)=\sum_ix_i^{s}$, a finite sum, and
	$I_\mu=S_\mu=\R$. For measures of unbounded support, or of support
	accumulating at the origin, the domain may be a proper subinterval of
	the line, and when it is, its endpoints turn out to carry the tail
	indices.
	
	\begin{restatable}[Regularity]{lemma}{lemlogconvex}\label{lem:logconvex}
		Let $\mu$ be admissible. Then:
		\begin{enumerate}[label=\textup{(\roman*)},leftmargin=2.2em]
			\item $I_\mu$ is an interval and $K_\mu$ is convex on $I_\mu$;
			\item $K_\mu$ is strictly convex on $\mathring I_\mu$ unless
			$\mu$ is concentrated at a single point;
			\item for every $[a,b]\subset\mathring I_\mu$ there are
			$\varepsilon>0$ and constants $C_{k,\varepsilon}$ such that,
			for $\Re s\in(a,b)$ and every integer $k\ge0$,
			\begin{equation}\label{eq:domination}
				\bigl|x^{s}\log^kx\bigr|
				\le(x^{a}+x^{b})\,|\log x|^{k}
				\le C_{k,\varepsilon}\bigl(x^{a-\varepsilon}
				+x^{b+\varepsilon}\bigr)\in L^1(\mu);
			\end{equation}
			\item consequently $Z_\mu$ extends holomorphically to the
			vertical strip $\{\Re s\in\mathring I_\mu\}$, differentiation
			under the integral sign is valid there to all orders, and
			$K_\mu\in C^\infty(\mathring I_\mu)$;
			\item $S_\mu$ is an interval and $\Leh_\mu\eqdef
			Z_\mu/Z_\mu(\cdot-1)$ is real-analytic on $\mathring S_\mu$.
		\end{enumerate}
	\end{restatable}
	
	Bound \eqref{eq:domination} justifies every differentiation under an
	integral sign in what follows.
	
	\begin{definition}[Transform and escorts]\label{def:L}
		For $s\in S_\mu$ the \emph{Lehmer transform} of $\mu$ is
		\begin{equation}\label{eq:defL}
			\Leh_\mu(s)\eqdef\frac{Z_\mu(s)}{Z_\mu(s-1)} ,
		\end{equation}
		and we write $\Lambda_\mu\eqdef\log\Leh_\mu$. For $u\in I_\mu$ the
		\emph{escort measure} of order $u$ is the probability measure
		\begin{equation}\label{eq:defescort}
			\dd\mu_u\eqdef\frac{x^{u}\,\dd\mu}{Z_\mu(u)} ;
		\end{equation}
		for empirical data the escort weights are
		\begin{equation}\label{eq:escortweights}
			p_i(u)=\frac{x_i^{\,u}}{\sum_jx_j^{\,u}} .
		\end{equation}
		The family $\{\mu_u\}_{u\in\mathring I_\mu}$ is the exponential
		family with natural parameter $u$, sufficient statistic $\log X$,
		base measure $\mu$ and log-partition function $K_\mu$. The
		normalization of $\mu$ is immaterial:
		$\Leh_{c\mu}=\Leh_\mu$ for every $c>0$.
	\end{definition}
	
	Definition~\ref{def:L} imports a standard object: the measures
	\eqref{eq:defescort} are the escort distributions of the
	thermodynamic and multifractal
	formalisms~\cite{beck1993,halsey1986}, and in that language the
	transform is the exponential of a unit increment of the free energy
	along the temperature axis. Four objects have now appeared, the
	Mellin transform, the log-pushforward, the escort family and the
	free energy; a single identity ties them together.
	
	\begin{restatable}[Master identity]{proposition}{propmaster}\label{prop:master}
		Let $\mu$ be admissible, let its Mellin transform be
		\begin{equation}\label{eq:defmellin}
			(\Mellin\mu)(s)\eqdef\int_{\Rpos} x^{s-1}\dd\mu ,
		\end{equation}
		and let $M$ be the moment generating function of the
		log-pushforward $(\log)_*\mu$, namely
		\begin{equation}\label{eq:defmgf}
			M(s)\eqdef\int_{\R}e^{st}\dd\bigl((\log)_*\mu\bigr)(t).
		\end{equation}
		For all $s\in\mathring S_\mu$,
		\begin{equation}\label{eq:master}
			\Leh_\mu(s)
			=\frac{(\Mellin\mu)(s+1)}{(\Mellin\mu)(s)}
			=\frac{M(s)}{M(s-1)}
			=\E_{\mu_{s-1}}[X]
			=\exp\bigl(K_\mu(s)-K_\mu(s-1)\bigr).
		\end{equation}
		Moreover
		\begin{equation}\label{eq:logL}
			\Lambda_\mu(s)=\int_{s-1}^{s}K_\mu'(u)\dd u ,
		\end{equation}
		where
		\begin{equation}\label{eq:Kderivs}
			\begin{gathered}
				K_\mu'(u)=\E_{\mu_u}[\log X],\\
				K_\mu''(u)=\Var_{\mu_u}(\log X).
			\end{gathered}
		\end{equation}
	\end{restatable}

	The partition function on an open interval determines the measure.
	This uniqueness principle is elementary and is used repeatedly.
	Determination of the measure by the \emph{transform}, which only
	prescribes increments of the free energy, is a different and
	subtler matter.
	
	\begin{restatable}[Strip uniqueness]{lemma}{lemstrip}\label{lem:strip}
		Let $\mu,\nu$ be admissible and suppose $Z_\mu=Z_\nu$ on a
		nonempty open interval
		$U\subset\mathring I_\mu\cap\mathring I_\nu$. Then $\mu=\nu$.
	\end{restatable}

	The complex strip carries one more piece of elementary structure:
	the partition function is strictly contractive off the real axis
	unless the data sit on a geometric progression.
	
	\begin{restatable}[Modulus bound]{proposition}{propmodulus}\label{prop:modulus}
		Let $\mu$ be admissible and $a\in\mathring I_\mu$. Then:
		\begin{enumerate}[label=\textup{(\roman*)},leftmargin=2.2em]
			\item $Z_\mu(\bar s)=\overline{Z_\mu(s)}$ on the strip, and
			\begin{equation}\label{eq:modbound}
				|Z_\mu(a+ib)|\le Z_\mu(a)\qquad\text{for all }b\in\R;
			\end{equation}
			\item for a fixed $b\neq0$, equality holds in
			bound \eqref{eq:modbound} if and only if $\mu$ is carried by a
			geometric progression of ratio $e^{2\pi/|b|}$, that is,
			\begin{equation}\label{eq:lattice}
				\supp\mu\subset\bigl\{c\,e^{2\pi k/|b|}:k\in\Z\bigr\}
				\quad\text{for some }c>0 ;
			\end{equation}
			\item if equality holds in bound \eqref{eq:modbound} for two
			frequencies $b_1,b_2>0$ with $b_1/b_2$ irrational, then $\mu$
			is concentrated at a single point.
		\end{enumerate}
	\end{restatable}
	
	We next record the exact discrete identity
	on which inversion, the power law and the diversity theory all rest.
	The logarithm of the transform is the unit backward difference of the
	free energy, and differences telescope.
	
	\begin{lemma}[Cocycle]\label{lem:cocycle}
		Let $\mu$ be admissible, let $k\ge1$ be an integer, and let
		$s\in I_\mu$ with $s-k\in I_\mu$, so that the whole ladder
		$s,s-1,\dots,s-k$ lies in $I_\mu$. Then
		\begin{equation}\label{eq:cocycle}
			\begin{gathered}
				\prod_{j=0}^{k-1}\Leh_\mu(s-j)=\frac{Z_\mu(s)}{Z_\mu(s-k)},\\
				\sum_{j=0}^{k-1}\Lambda_\mu(s-j)=K_\mu(s)-K_\mu(s-k).
			\end{gathered}
		\end{equation}
		If in addition $\{0,k\}\subset I_\mu$, so that the whole ladder
		$0,1,\dots,k$ lies in $I_\mu$ by convexity, the power moments
		telescope,
		\begin{equation}\label{eq:momenttelescope}
			Z_\mu(k)=Z_\mu(0)\prod_{j=1}^{k}\Leh_\mu(j),
		\end{equation}
		so the transform on the positive integers is a complete moment
		coordinate.
	\end{lemma}
	
	\begin{proof}
		Each factor is $\Leh_\mu(s-j)=Z_\mu(s-j)/Z_\mu(s-j-1)$, and in the
		product each intermediate value
		$Z_\mu(s-1),\dots,Z_\mu(s-k+1)$ appears once in a numerator and
		once in a denominator, leaving $Z_\mu(s)/Z_\mu(s-k)$. Taking
		logarithms gives the additive form. Setting $s=k$, which is
		legitimate because $I_\mu$ is an interval containing $0$ and $k$,
		gives identity \eqref{eq:momenttelescope}.
	\end{proof}
	
	The behaviour of the transform at the endpoints of its suddency
	domain is already determined by the domain itself: a divergence at
	the right endpoint records an upper tail, while a zero at the left
	endpoint
	records an accumulation of mass at the origin, displaced upward by
	exactly one suddency unit.
	
	\begin{restatable}[Endpoints]{proposition}{propendpoints}\label{prop:endpoints}
		Let $\mu$ be admissible with $\mathring S_\mu\neq\emptyset$.
		\begin{enumerate}[label=\textup{(\roman*)},leftmargin=2.2em]
			\item If $\tau\eqdef\sup I_\mu<\infty$, if
			$\tau-1\in\mathring I_\mu$ and if $Z_\mu(u)\to\infty$ as
			$u\uparrow\tau$, then $\sup S_\mu=\tau$ and
			\begin{equation}\label{eq:rightendpoint}
				\Leh_\mu(s)\longrightarrow\infty,\qquad s\uparrow\tau .
			\end{equation}
			\item If $\iota\eqdef\inf I_\mu>-\infty$, if
			$\iota+1\in\mathring I_\mu$ and if $Z_\mu(u)\to\infty$ as
			$u\downarrow\iota$, then $\inf S_\mu=\iota+1$ and
			\begin{equation}\label{eq:leftendpoint}
				\Leh_\mu(s)\longrightarrow0,\qquad s\downarrow\iota+1 .
			\end{equation}
		\end{enumerate}
	\end{restatable}

	The divergence hypothesis in the proposition is not automatic. A
	law with tail $\Prob(X>x)=x^{-\alpha}\log^{-2}x$ has
	$Z_P(\alpha)<\infty$, so its transform stays finite at the
	endpoint of the moment domain. Under a regularly varying tail,
	however, divergence at the tail index becomes a theorem, and the
	Abelian tail law proved in this paper makes it quantitative.
	
	Monotonicity of the Lehmer family in its order is classical, going
	back to the Gini means~\cite{gini1938,bullen2003}; what the escort
	structure adds is the exact reading of both rates of increase.
	
	\begin{restatable}[Monotonicity]{theorem}{thmmono}\label{thm:mono}
		Let $\mu$ be admissible with $\mathring S_\mu\neq\emptyset$ and let
		$s\in\mathring S_\mu$. Then:
		\begin{enumerate}[label=\textup{(\roman*)},leftmargin=2.2em]
			\item the derivative of the transform is an escort covariance,
			\begin{equation}\label{eq:covid}
				\Leh_\mu'(s)=\Cov_{\mu_{s-1}}\bigl(X,\log X\bigr)\;\ge\;0,
			\end{equation}
			with equality if and only if $\mu$ is concentrated at a single
			point;
			\item the derivative of its logarithm is an integrated escort
			variance,
			\begin{equation}\label{eq:logderiv}
				\Lambda_\mu'(s)
				=\int_{s-1}^{s}\Var_{\mu_u}(\log X)\dd u\;\ge\;0 ;
			\end{equation}
			\item consequently $\Leh_\mu$ is nondecreasing on $S_\mu$, and
			strictly increasing unless the data are constant.
		\end{enumerate}
	\end{restatable}

	Identity \eqref{eq:covid} says more than monotonicity. The rate at
	which the statistic rises with the suddency moment is a genuine
	dispersion measure on the tilted sample. At the arithmetic node the
	tilt is trivial and the escort is the normalized measure
	$\mu_0=\mu/Z_\mu(0)$, so that
	\begin{equation}\label{eq:slopeatone}
		\Leh_\mu'(1)=\Cov_{\mu_0}\bigl(X,\log X\bigr) ,
	\end{equation}
	the ordinary sample covariance of the $x_i$ and $\log x_i$ for a
	sample of size $n$. Identity \eqref{eq:covid} also yields sharp
	quantitative control, and organizes the higher derivatives into an
	exact cumulant hierarchy.
	
	\begin{restatable}[Slope bounds and cumulants]{proposition}{propgruss}\label{prop:gruss}
		Let $\mu$ be admissible.
		\begin{enumerate}[label=\textup{(\roman*)},leftmargin=2.2em]
			\item If $\supp\mu\subset[m,M]\subset\Rpos$, then for all
			$s\in\R$,
			\begin{equation}\label{eq:gruss}
				\begin{gathered}
					0\le\Leh_\mu'(s)\le\tfrac14\,(M-m)\log\tfrac Mm,\\
					0\le\Lambda_\mu'(s)\le\tfrac14\log^2\tfrac Mm .
				\end{gathered}
			\end{equation}
			\item With $\kappa_k(u)$ the $k$th cumulant of $\log X$ under
			$\mu_u$, one has for all $k\ge1$ and $s\in\mathring S_\mu$
			\begin{equation}\label{eq:Lamk}
				\Lambda_\mu^{(k)}(s)=\kappa_k(s)-\kappa_k(s-1),
			\end{equation}
			and for all $n\ge1$
			\begin{equation}\label{eq:Belln}
				\Leh_\mu^{(n)}(s)=\Leh_\mu(s)\,
				B_n\bigl(\Lambda_\mu'(s),\dots,\Lambda_\mu^{(n)}(s)\bigr),
			\end{equation}
			with $B_n$ the complete Bell polynomial.
		\end{enumerate}
	\end{restatable}
	
	The bounds \eqref{eq:gruss} certify that the transform of bounded
	data is Lipschitz in the suddency moment with constants read off the
	dynamic range, so discretizing the suddency axis of a signal incurs a
	controlled error. The hierarchy \eqref{eq:Lamk} shows that
	$\Lambda_\mu$ is convex exactly where the escort log-variance at $s$
	exceeds that at $s-1$, so the inflection points of $\Lambda_\mu$
	solve
	\begin{equation}\label{eq:inflection}
		\Var_{\mu_s}(\log X)=\Var_{\mu_{s-1}}(\log X).
	\end{equation}
	The transform itself inflects on a different locus. By the case
	$n=2$ of formula \eqref{eq:Belln}, the second derivative
	$\Leh_\mu''$ vanishes where
	$\Lambda_\mu''+(\Lambda_\mu')^{2}=0$, that is, where
	\begin{equation}\label{eq:inflectionL}
		\kappa_2(s)-\kappa_2(s-1)
		=-\bigl(\kappa_1(s)-\kappa_1(s-1)\bigr)^{2}.
	\end{equation}
	Both loci are thereby settled in exact escort terms.
	
	Two elementary symmetries govern how the transform responds to
	transformations of the data.
	
	\begin{restatable}[Homogeneity and reflection]{proposition}{prophomogrefl}\label{prop:homogrefl}
		Let $\mu$ be admissible.
		\begin{enumerate}[label=\textup{(\roman*)},leftmargin=2.2em]
			\item For $\lambda>0$, the dilation $x\mapsto\lambda x$ sends
			$Z_\mu(s)$ to $\lambda^{s}Z_\mu(s)$ and $\Leh_\mu$ to
			$\lambda\Leh_\mu$: the transform is homogeneous of degree one,
			so that after division by any one of its values it depends on
			the data only through their ratios.
			\item If $\check\mu$ is the pushforward of $\mu$ under
			$x\mapsto1/x$, then $I_{\check\mu}=-I_\mu$ and, wherever both
			sides are defined,
			\begin{equation}\label{eq:reflection}
				\Leh_{\check\mu}(s)\,\Leh_\mu(1-s)=1 .
			\end{equation}
			If $\check\mu=\mu$, then $\Lambda_\mu$ is odd about
			$s=\tfrac12$; in particular $\Leh_\mu(\tfrac12)=1$ and every
			even derivative of $\Lambda_\mu$ vanishes at $\tfrac12$.
		\end{enumerate}
	\end{restatable}

	Finally, the global shape: on compactly supported data the transform
	is a strictly increasing real-analytic bijection onto the open range
	of the data, meeting the classical means at data-independent nodes,
	harmonic at $s=0$, arithmetic at $s=1$, contraharmonic at $s=2$.
	This is the precise sense in which it is a
	\emph{statistic-generating function}.
	
	\begin{restatable}[Generating bijection]{theorem}{thmstatgen}
		\label{thm:statgen}
		Let $\mu$ be admissible and compactly supported in $\Rpos$, with
		$m=\min\supp\mu<M=\max\supp\mu$; such a measure is automatically
		finite, by Remark~\ref{rem:sigmafinite}, and $I_\mu=S_\mu=\R$.
		Then:
		\begin{enumerate}[label=\textup{(\roman*)},leftmargin=2.2em]
			\item $\Leh_\mu:\R\to(m,M)$ is a strictly increasing
			real-analytic bijection onto the \emph{open} interval: the
			extremes are approached as $s\to\pm\infty$ but never
			attained;
			\item if the top of the support is an atom separated by a gap
			below $M'<M$, the approach is geometric,
			\begin{equation}\label{eq:georate}
				M-\Leh_\mu(s)=O\bigl((M'/M)^{s}\bigr),
				\qquad s\to+\infty;
			\end{equation}
			\item for a two-point measure $\mu=\delta_u+\delta_v$ with
			$M=\max(u,v)$ and $s>1$,
			\begin{equation}\label{eq:twopoint}
				0\;\le\;M-\Leh_\mu(s)\;\le\;
				\frac{M}{s}\Bigl(1-\frac1s\Bigr)^{s-1}
				\;=\;\frac{M}{es}\bigl(1+O(s^{-1})\bigr);
			\end{equation}
			\item the real-analytic inverse of $\Leh_\mu$, the
			\emph{suddency map}, has first-order form at the arithmetic
			node
			\begin{equation}\label{eq:suddencymap}
				\Leh_\mu^{-1}(t)=1+\frac{t-\Leh_\mu(1)}
				{\Cov_{\mu_0}(X,\log X)}
				+O\bigl((t-\Leh_\mu(1))^2\bigr).
			\end{equation}
		\end{enumerate}
	\end{restatable}
	
	The limit $\tfrac1s\log Z_\mu(s)\to\log M$ behind the endpoints
	is Varadhan's lemma: the transform is a smooth self-normalizing soft
	maximum, and the bound \eqref{eq:twopoint} is the rate at which it
	resolves the largest observation. The suddency map converts a
	statistic back into the moment that generates it, with sensitivity,
	by expansion \eqref{eq:suddencymap}, inversely proportional to the
	log-covariance dispersion: the flatter the curve, the worse
	conditioned the reading.
	
	%=====================================================================
	\section{Divergence Calculus}\label{sec:divergence}
	%=====================================================================
	
	The classical mean value theorem writes the unit increment of the
	free energy as $K_\mu'(\xi)$ for an unknowable intermediate point
	$\xi$. The escort structure does strictly better: it replaces the
	unknowable point by \emph{exact} remainders, and the remainders are
	relative entropies. Throughout this section we write
	\begin{equation}\label{eq:kldef}
		\begin{gathered}
			\KL(P\|Q)=\int\log\frac{\dd P}{\dd Q}\dd P,\\
			J(P,Q)=\KL(P\|Q)+\KL(Q\|P)
		\end{gathered}
	\end{equation}
	for the relative entropy and the Jeffreys divergence. The content of
	this section lies in the identification of the right-hand sides: it
	is this identification that turns the classical inequalities between
	means into exact information-theoretic identities. The basic
	identity of the tilt family is the following.
	
	\begin{restatable}[Bregman identity]{lemma}{lembregman}\label{lem:bregman}
		Let $\mu$ be admissible. For all $a,b\in\mathring I_\mu$,
		\begin{equation}\label{eq:bregman}
			\KL(\mu_a\|\mu_b)=K_\mu(b)-K_\mu(a)-(b-a)K_\mu'(a).
		\end{equation}
	\end{restatable}

	\begin{restatable}[Two-sided remainder]{theorem}{thmsandwich}\label{thm:sandwich}
		Let $\mu$ be admissible. For all $s\in\mathring S_\mu$,
		\begin{align}
			\Lambda_\mu(s)
			&=\E_{\mu_{s-1}}[\log X]+\KL(\mu_{s-1}\|\mu_s),
			\label{eq:sandwichlower}\\[2pt]
			\Lambda_\mu(s)
			&=\E_{\mu_{s}}[\log X]-\KL(\mu_{s}\|\mu_{s-1}).
			\label{eq:sandwichupper}
		\end{align}
		In particular
		\begin{equation}\label{eq:sandwich}
			\E_{\mu_{s-1}}[\log X]\;\le\;\Lambda_\mu(s)
			\;\le\;\E_{\mu_s}[\log X],
		\end{equation}
		with equality at one point if and only if equality holds at every
		point, if and only if $\mu$ is concentrated at a single point.
	\end{restatable}

	Specializing to the arithmetic node produces the sharpest elementary
	consequence of the calculus.
	
	\begin{restatable}[AM--GM gap]{corollary}{coramgm}\label{cor:amgm}
		Let $\mu$ be a probability measure with $1\in I_\mu$ and
		$\E_\mu|\log X|<\infty$, and set
		\begin{equation}\label{eq:AGdef}
			\begin{gathered}
				A=\E_\mu[X],\\
				G=\exp\E_\mu[\log X].
			\end{gathered}
		\end{equation}
		Then
		\begin{equation}\label{eq:amgm}
			\log\frac AG=\KL(\mu\,\|\,\mu_1),
		\end{equation}
		where $\mu_1$ is the size-biased law $\dd\mu_1=x\dd\mu/A$. For a
		data vector, the arithmetic--geometric gap is exactly the relative
		entropy from uniform weights to size-biased weights. More
		generally, every Lehmer statistic exceeds its escort geometric
		mean by an exact relative entropy,
		\begin{equation}\label{eq:amgmgeneral}
			\Lambda_\mu(s)-\E_{\mu_{s-1}}[\log X]=\KL(\mu_{s-1}\|\mu_s).
		\end{equation}
	\end{restatable}

	Identity \eqref{eq:amgm} upgrades the arithmetic--geometric mean
	inequality to an identity and explains its equality case
	information-theoretically: uniform weights coincide with size-biased
	weights exactly for constant data. It also produces refinements on
	demand. Writing
	\begin{equation}\label{eq:tvdef}
		d_{\mathrm{TV}}(P,Q)\eqdef\sup_{B}\bigl|P(B)-Q(B)\bigr|\in[0,1]
	\end{equation}
	for the total variation distance in the supremum normalization,
	Pinsker's inequality gives
	\begin{equation}\label{eq:pinsker}
		\log\frac AG\;\ge\;2\,d_{\mathrm{TV}}(\mu,\mu_1)^{2},
	\end{equation}
	a total-variation strengthening of the classical inequality; the
	constant~$2$ in inequality \eqref{eq:pinsker} belongs to the
	normalization \eqref{eq:tvdef} and becomes $\tfrac12$ if the total
	variation is instead measured by the $L^1$ distance of densities.
	
	The same calculus settles the relation of the transform to the
	geometric mean.
	
	\begin{restatable}[Midpoint rule]{lemma}{lemmidpoint}\label{lem:midpoint}
		Let $\mu$ be admissible and suppose
		$[s-\tfrac12,s+\tfrac12]\subset\mathring I_\mu$. Then there
		exists $\xi\in(s-\tfrac12,s+\tfrac12)$ with
		\begin{equation}\label{eq:midpoint}
			\Lambda_\mu\bigl(s+\tfrac12\bigr)
			=K_\mu'(s)+\tfrac1{24}K_\mu'''(\xi),
		\end{equation}
		and
		\begin{equation}\label{eq:thirdcentral}
			K_\mu'''(u)=\E_{\mu_u}\bigl[(\log X-K_\mu'(u))^3\bigr].
		\end{equation}
	\end{restatable}

	At $s=0$ the value $\Leh_\mu(\tfrac12)$ is therefore a
	second-order accurate geometric mean, the defect a pure skewness
	functional of the log-data. For a two-point measure
	$\mu=\delta_u+\delta_v$ the defect vanishes at this node,
	$\Leh_\mu(\tfrac12)=\sqrt{uv}$, because the escort $\mu_0$
	weighs the two points equally and $K_\mu'''(0)=0$; for every
	lognormal law the free energy is quadratic, so the midpoint rule is
	exact at every node. The following consequence is worth isolating.
	
	\begin{example}[Digamma approximation]\label{ex:digamma}
		The Gamma law with shape $k$ and unit scale has
		\begin{equation}\label{eq:gammaKL}
			\begin{gathered}
				Z_\mu(s)=\frac{\Gamma(k+s)}{\Gamma(k)},\\
				K_\mu(s)=\log\Gamma(k+s)-\log\Gamma(k),\\
				\Leh_\mu(s)=s+k-1 .
			\end{gathered}
		\end{equation}
		Identity \eqref{eq:midpoint} at this measure reads
		\begin{equation}\label{eq:digamma}
			\begin{gathered}
				\log\bigl(z-\tfrac12\bigr)=\psi(z)+\tfrac1{24}\psi''(\xi),\\
				z=k+s,\qquad\xi\in\bigl(z-\tfrac12,\,z+\tfrac12\bigr),
			\end{gathered}
		\end{equation}
		so the classical approximation $\psi(z)\approx\log(z-\tfrac12)$ is
		nothing but the midpoint rule for the Lehmer transform of the
		Gamma distribution, with error
		$\tfrac1{24}\psi''(\xi)=O(z^{-2})$.
	\end{example}
	
	The slope of the transform carries the deeper identity. Relation
	\eqref{eq:logderiv} expressed the slope of the log-transform as an
	integrated escort variance. What follows identifies that quantity,
	at once, as the Fisher information of the escort family integrated
	over one unit of suddency and as a symmetrized divergence between
	consecutive escorts.
	
	\begin{restatable}[Jeffreys identity]{theorem}{thmjeffreys}\label{thm:jeffreys}
		Let $\mu$ be admissible. For all $s\in\mathring S_\mu$,
		\begin{equation}\label{eq:jeffreys}
			\Lambda_\mu'(s)
			=J\bigl(\mu_{s-1},\mu_s\bigr)
			=\int_{s-1}^{s}\Var_{\mu_u}(\log X)\dd u:
		\end{equation}
		the slope of the log-transform is the Jeffreys divergence between
		consecutive escort measures, equivalently the escort Fisher
		information integrated over one unit of suddency.
	\end{restatable}

	The slope of the log-transform at $s$ therefore measures, exactly,
	the statistical distinguishability of the tilted ensembles at
	inverse temperatures $s-1$ and $s$: the curve is flat where
	neighboring tilts are indistinguishable and steep across a phase
	boundary. In thermodynamic language $K_\mu''$ is the
	susceptibility of the ensemble with energy $\log X$, so the
	suddency slope is the integrated susceptibility, and a sharp peak
	of the susceptibility appears as a steep rise and inflection of the
	curve; for orbits of a dynamical system with an equilibrium state
	the same identity reads through the thermodynamic
	formalism~\cite{ruelle2004}, with the slope estimating the
	Green--Kubo diffusion coefficient of Birkhoff averages.
	
	Two consequences complete the calculus. The first is that the entire
	R\'enyi divergence geometry of the tilt family is encoded in
	increments of the free energy, hence, by the cocycle identity
	\eqref{eq:cocycle}, in the transform. The second is a variational
	principle for the transform itself.
	
	\begin{restatable}[R\'enyi and Donsker--Varadhan]{proposition}{proprenyidiv}\label{prop:renyidiv}
		Let $\mu$ be admissible.
		\begin{enumerate}[label=\textup{(\roman*)},leftmargin=2.2em]
			\item Let $a\neq b$ lie in $\mathring I_\mu$ and let
			$\lambda\in(0,1)\cup(1,\infty)$ be such that
			$\lambda a+(1-\lambda)b\in\mathring I_\mu$, which holds
			for every $\lambda\in(0,1)$ by convexity. Then
			\begin{equation}\label{eq:renyidiv}
				D_\lambda(\mu_a\|\mu_b)
				=\frac{K_\mu\bigl(\lambda a+(1-\lambda)b\bigr)
					-\lambda K_\mu(a)-(1-\lambda)K_\mu(b)}{\lambda-1} .
			\end{equation}
			\item For $s\in\mathring S_\mu$ the Chernoff information
			between consecutive escorts is
			\begin{equation}\label{eq:chernoff}
				C(\mu_s,\mu_{s-1})=\max_{\lambda\in[0,1]}
				\bigl[\lambda K_\mu(s)+(1-\lambda)K_\mu(s-1)
				-K_\mu(s-1+\lambda)\bigr].
			\end{equation}
			\item For $s\in\mathring S_\mu$,
			\begin{equation}\label{eq:dv}
				\Lambda_\mu(s)
				=\sup_{q}\Bigl\{\E_q[\log X]
				-\KL\bigl(q\,\|\,\mu_{s-1}\bigr)\Bigr\},
			\end{equation}
			the supremum running over the probability measures $q$ with
			$\KL(q\,\|\,\mu_{s-1})<\infty$ and $\E_q[\log X]$
			finite, and being attained uniquely at $q=\mu_s$.
		\end{enumerate}
	\end{restatable}
	
	At $\lambda=1$ the R\'enyi divergence is the Kullback--Leibler
	divergence by continuity, and the restriction on $q$ in part (iii)
	removes the indeterminate form $\infty-\infty$ while discarding no
	candidate, the objective being $-\infty$ on the excluded measures.
	All R\'enyi and Chernoff quantities along the tilt path are
	therefore computable from the curve by telescoping, with no density
	estimation; in the language of information geometry the escort
	curve is the e-geodesic of the exponential family generated by
	$\log X$~\cite{amari2000}, and the transform is the mean of $X$
	transported along it.
	
	The cocycle turns all of this into a design principle. Free-energy
	differences across integer spans factor exactly into unit Lehmer
	steps,
	\begin{equation}\label{eq:thermint}
		K_\mu(s+N)-K_\mu(s)=\sum_{k=1}^N\Lambda_\mu(s+k),
	\end{equation}
	each step an escort expectation estimable by reweighted sampling
	from a single ensemble. This is the discrete skeleton of
	thermodynamic integration~\cite{kirkwood1935}, free-energy
	perturbation~\cite{zwanzig1954}, annealed importance
	sampling~\cite{neal2001} and Bennett's acceptance-ratio
	method~\cite{bennett1976,shirts2008}. The telescoping is exact on
	any grid, there as here; what is specific here is that the grid is
	canonical, fixed at one unit of the natural parameter by the Lehmer
	normalization, and the accuracy of each step is governed by the
	Jeffreys divergences of Theorem~\ref{thm:jeffreys}.
	
	%=====================================================================
	\section{Multiplicative Structure}\label{sec:multiplicative}
	%=====================================================================
	
	The multiplicative group $(\Rpos,\times)$ has continuous characters
	$\chi_s(x)=x^{s}$, and
	\begin{equation}\label{eq:pairing}
		Z_\mu(s)=\int_{\Rpos}\chi_s\dd\mu
	\end{equation}
	is the pairing of $\mu$ against them: the Fourier--Mellin transform
	of $\mu$ on its natural group. The Lehmer transform is the pairing at
	adjacent characters, and it therefore inherits the defining behaviour
	of the Mellin transform under multiplicative convolution: it converts
	multiplication of independent factors into the pointwise product of
	transforms. What is worth noticing is that no \emph{fixed} mean has
	this property: reading the family as one curve is exactly what makes
	it available. The
	transform is for products what the characteristic function is for
	sums. We develop the
	consequences here: a transformation calculus for powers, a
	L\'evy--Khintchine representation, a renormalization fixed point,
	and exact deconvolution of multiplicative noise.
	
	Throughout this section, $\mu\boxtimes\nu$ denotes multiplicative
	convolution, defined for admissible $\mu$ and $\nu$ as the
	pushforward of the product measure under the group operation,
	\begin{equation}\label{eq:boxtimes}
		\begin{gathered}
			\mu\boxtimes\nu\eqdef\pi_*(\mu\otimes\nu),\\
			\pi(x,y)=xy .
		\end{gathered}
	\end{equation}
	Definition \eqref{eq:boxtimes} makes sense for the $\sigma$-finite
	measures of Remark~\ref{rem:sigmafinite}, and not only for
	probability laws. This matters because several ensembles to which
	the theory applies, such as the geometric ensemble
	$\sum_{n\ge0}\delta_{q^{\,n}}$, have infinite total mass. When $\mu$ and $\nu$ are
	probability laws, definition \eqref{eq:boxtimes} is the law of $XY$
	for independent $X\sim\mu$ and $Y\sim\nu$.
	
	\begin{restatable}[Character property]{theorem}{thmcharacter}\label{thm:homo}
		For admissible $\mu,\nu$ one has $Z_{\mu\boxtimes\nu}=Z_\mu
		Z_\nu$ on $I_\mu\cap I_\nu$, and consequently
		\begin{equation}\label{eq:homo}
			\Leh_{\mu\boxtimes\nu}(s)=\Leh_\mu(s)\,\Leh_\nu(s),
			\qquad s\in S_\mu\cap S_\nu .
		\end{equation}
		Thus $\mu\mapsto\Leh_\mu$ is a homomorphism from positive measures
		under multiplicative convolution to positive functions under
		pointwise multiplication, and $\Lambda_\mu$ is additive over
		independent factors.
	\end{restatable}

	Whereas the homomorphism concerns products of \emph{distinct}
	independent variables, raising a \emph{single} variable to a power
	interacts with the transform through the cocycle identity and
	produces a dilation law: the suddency axis stretches, and the
	transform of the power is a product over consecutive unit windows.
	
	\begin{restatable}[Powers]{theorem}{thmpowers}\label{thm:power}
		Let $X\sim\mu$ be admissible, let $r>0$ be real, and let $s$ be
		such that both $rs$ and $rs-r$ lie in $\mathring I_\mu$;
		equivalently, $s\in\mathring S_{X^r}$, since
		$I_{X^r}=r^{-1}I_\mu$. Then
		\begin{equation}\label{eq:powerlog}
			\Lambda_{X^r}(s)=K_\mu(rs)-K_\mu(rs-r)
			=r\int_{s-1}^{s}K_\mu'(rv)\dd v ,
		\end{equation}
		and for integer $r\ge1$, the whole ladder
		$rs,rs-1,\dots,rs-r$ then lying in $\mathring I_\mu$ by convexity,
		\begin{equation}\label{eq:power}
			\Leh_{X^r}(s)=\prod_{j=0}^{r-1}\Leh_X(rs-j).
		\end{equation}
	\end{restatable}

	Together with homogeneity, which is the case $Y\equiv\lambda$ of
	Theorem~\ref{thm:homo}, and the reflection identity
	\eqref{eq:reflection}, which is the case $r=-1$ up to
	$s\mapsto1-s$, this is the complete transformation calculus of the
	multiplicative group. On multiplicatively infinitely divisible laws
	the additive structure linearizes completely.
	
	\begin{restatable}[L\'evy--Khintchine]{theorem}{thmlevy}
		\label{thm:levy}
		Let $\log X$ be infinitely divisible with triplet
		$(b,\sigma^2,\Pi)$ and exponential moments on an open interval
		containing $[s-1,s]$. Then:
		\begin{enumerate}[label=\textup{(\roman*)},leftmargin=2.2em]
			\item the log-transform has the L\'evy--Khintchine form
			\begin{equation}\label{eq:levyL}
				\Lambda_X(s)=b+\sigma^2\Bigl(s-\tfrac12\Bigr)
				+\int_\R\Bigl[e^{(s-1)y}\bigl(e^{y}-1\bigr)
				-y\,\mathbf 1_{|y|\le1}\Bigr]\Pi(\dd y),
			\end{equation}
			\item its slope is bounded below by the Gaussian coefficient,
			\begin{equation}\label{eq:levyderiv}
				\Lambda_X'(s)
				=\sigma^2+\int_\R y\,e^{(s-1)y}\bigl(e^{y}-1\bigr)
				\Pi(\dd y)\;\ge\;\sigma^2 ;
			\end{equation}
			\item in this class the $k$th multiplicative root of $X$, by
			which we mean the positive variable $X^{[k]}$ whose logarithm
			is the $k$th convolution root of $\log X$, so that a product of
			$k$ independent copies of $X^{[k]}$ has the law of $X$, has
			transform $\Leh_X^{1/k}$.
		\end{enumerate}
	\end{restatable}
	
	The root in part (iii) is the infinitely divisible root, not the
	pointwise power $X^{1/k}$, which by Theorem~\ref{thm:power}
	rescales the suddency axis instead. The lognormal transform
	\begin{equation}\label{eq:lognormalcurve}
		\Leh(s)=e^{\,m+\sigma^2(s-1/2)}
	\end{equation}
	is not invariant under $\Leh\mapsto\Leh^{k}$, even up to
	dilation; what is invariant is the log-affine \emph{class}, and it
	is exactly the fixed-point set of the correctly normalized
	renormalization.
	
	\begin{restatable}[Renormalization fixed points]{theorem}{thmrenorm}\label{thm:renorm}
		Let $X$ be a positive random variable with $\E[\log X]=0$,
		$\sigma^2=\Var(\log X)\in(0,\infty)$, and with the moment
		generating function of $\log X$ finite on a neighbourhood
		$(-\varepsilon_0,\varepsilon_0)$ of the origin. For an integer $k\ge1$ let
		$X_1,\dots,X_k$ be independent copies of $X$ and define the
		renormalization
		\begin{equation}\label{eq:renormdef}
			\mathcal R_kX\eqdef\Bigl(\prod_{i=1}^{k}X_i\Bigr)^{1/\sqrt k} .
		\end{equation}
		Then:
		\begin{enumerate}[label=\textup{(\roman*)},leftmargin=2.2em]
			\item the free energy transforms by
			\begin{equation}\label{eq:renormK}
				K_{\mathcal R_kX}(s)=k\,K_X\bigl(s/\sqrt k\bigr) ;
			\end{equation}
			\item $X$ is a fixed point of $\mathcal R_k$ for some, and then
			for every, integer $k\ge2$ if and only if $X$ is lognormal with
			parameters $(0,\sigma^2)$, in which case its transform is
			log-affine,
			\begin{equation}\label{eq:logaffine}
				\Leh_X(s)=e^{\sigma^2(s-1/2)} ;
			\end{equation}
			\item for every $X$ as above,
			\begin{equation}\label{eq:multclt}
				\Lambda_{\mathcal R_nX}(s)\longrightarrow
				\sigma^2\Bigl(s-\tfrac12\Bigr),
				\qquad n\to\infty,
			\end{equation}
			locally uniformly in $s\in\R$.
		\end{enumerate}
	\end{restatable}
	
	Part (ii) explains the lognormal's place in the theory, and part
	(iii) is the multiplicative central limit theorem read on the curve.
	The implication of part (ii) does not reverse, and the obstruction
	deserves explicit exhibition, because the same mechanism governs
	the whole rigidity theory: a positive $1$-periodic factor on the
	partition function leaves the transform unchanged, and it can be
	chosen to preserve both normalizations above.
	
	\begin{example}[A log-affine companion]\label{ex:logaffinecompanion}
		Fix $\sigma^2>0$, put $a_k\eqdef e^{-2\pi^2k^2/\sigma^2}$, and let
		$\varepsilon_1,\varepsilon_2$ satisfy
		\begin{equation}\label{eq:companionconstraint}
			|\varepsilon_1|+|\varepsilon_2|\le1,
			\qquad
			\varepsilon_1a_1+2\varepsilon_2a_2=0 .
		\end{equation}
		The density
		\begin{equation}\label{eq:companiondensity}
			f(x)=\frac{1}{x\sigma\sqrt{2\pi}}\,
			e^{-(\log x)^2/(2\sigma^2)}
			\Bigl[1+\varepsilon_1\sin\frac{2\pi\log x}{\sigma^2}
			+\varepsilon_2\sin\frac{4\pi\log x}{\sigma^2}\Bigr]
		\end{equation}
		is then nonnegative and defines a probability law $\mu$ with
		\begin{equation}\label{eq:companionZ}
			Z_\mu(s)=e^{\sigma^2s^2/2}\,G(s),
			\qquad G(s)=1+\varepsilon_1a_1\sin(2\pi s)
			+\varepsilon_2a_2\sin(4\pi s),
		\end{equation}
		and $G$ has period one, so the transform of $\mu$ is the
		log-affine curve \eqref{eq:logaffine} exactly. The constraint
		\eqref{eq:companionconstraint} gives $G(0)=1$ and
		$G'(0)=G''(0)=0$, so $\mu$ is a probability law with
		$\E_\mu[\log X]=0$ and $\Var_\mu(\log X)=\sigma^2$, and
		unless $\varepsilon_1=\varepsilon_2=0$ it is not lognormal.
		Taking $\sigma^2=20$ and $\varepsilon_2=0.9$, so that
		$\varepsilon_1=-0.093$ and the bracket in density
		\eqref{eq:companiondensity} stays above $0.033$, gives a law
		visibly different from the lognormal whose Lehmer curve coincides
		with it identically.
	\end{example}
	
	Log-affinity of the curve therefore characterizes the lognormal
	only under an additional hypothesis, and multiplicative infinite
	divisibility proves to be one such.
	
	The character property also gives an exact treatment of the generic
	contamination model for positive data: for a product $Y=XN$ of a
	signal and an independent positive noise of known law,
	deconvolution is pointwise \emph{division} of Lehmer curves, with
	no inversion integral, and it is stable in relative error.
	
	\begin{restatable}[Deconvolution]{proposition}{propdeconv}\label{prop:deconv}
		Let $Y=XN$ with $X$ and $N$ independent and positive. Then
		\begin{equation}\label{eq:deconv}
			\Leh_X(s)=\frac{\Leh_Y(s)}{\Leh_N(s)},
			\qquad s\in S_X\cap S_N .
		\end{equation}
		Moreover, if
		\begin{equation}\label{eq:relerr}
			\begin{gathered}
				\widehat\Leh_Y=\Leh_Y(1+\varepsilon_Y),\\
				\widehat\Leh_N=\Leh_N(1+\varepsilon_N)
			\end{gathered}
		\end{equation}
		pointwise with
		$|\varepsilon_Y|,|\varepsilon_N|\le\varepsilon<\tfrac12$, then
		\begin{equation}\label{eq:deconvstab}
			\begin{gathered}
				\frac{\widehat\Leh_Y}{\widehat\Leh_N}=\Leh_X(1+\varepsilon_X),\\
				|\varepsilon_X|\le\frac{2\varepsilon}{1-\varepsilon} ,
			\end{gathered}
		\end{equation}
		and the constant in bound \eqref{eq:deconvstab} is attained, so it
		cannot be improved.
	\end{restatable}

	A developed statistical literature on multiplicative deconvolution
	complements the identity. Density estimation under multiplicative
	measurement error proceeds through the Mellin transform, with
	minimax rates governed by the decay of the noise Mellin
	transform~\cite{belomestny2020,brennermiguel2021}. The present
	statement is the population identity at the level of the curve, and
	the division is stable in \emph{relative} error precisely because
	the curve is a ratio of two Mellin evaluations; the ill-posedness
	that drives the minimax rates reappears only when the curve is
	inverted into the measure itself.
	
	Every functional of the signal expressible through its Lehmer curve
	is thereby recovered directly. Full recovery of the law itself is
	the inversion problem, to which we now turn.
	%=====================================================================
	\section{Rigidity and Inversion}\label{sec:rigidity}
	%=====================================================================
	
	Does the Lehmer transform determine the measure? At best up to a
	positive scalar, since $\Leh_{c\mu}=\Leh_\mu$; and since
	$\Lambda_\mu$ prescribes only unit increments of the free energy,
	the a priori ambiguity is the kernel of the unit difference
	operator. This section settles the question in both directions. A
	growth condition forces determinacy, by a convexity squeeze in the
	spirit of Bohr--Mollerup made quantitative, and the condition holds
	for every measure of bounded support, hence all empirical data.
	Against this, the lognormal carries an explicit one-parameter fiber
	of distinct laws sharing its transform at every real suddency
	moment, by resonant log-periodic modulation; the mechanism is
	Heyde's classical moment indeterminacy~\cite{heyde1963}, and the
	exact obstruction is a nontrivial group of imaginary periods. The
	two mechanisms are then proved mutually exclusive.
	
	\begin{restatable}[Gauge fiber]{lemma}{lemfiber}\label{lem:fiber}
		Let $\mathcal G$ denote the multiplicative group of functions
		$e^{\varpi}$ with $\varpi$ real-analytic and $1$-periodic. Let
		$\mu,\mu'$ be admissible with a common \emph{open} suddency
		interval $S$ of length greater than one. Then
		$\Leh_\mu\equiv\Leh_{\mu'}$ on $S$ if and only if
		\begin{equation}\label{eq:gaugeeq}
			Z_{\mu'}=G\,Z_\mu\quad\text{on }S\cup(S-1)
			\qquad\text{for some }G\in\mathcal G .
		\end{equation}
	\end{restatable}

	\begin{remark}[Reading the lemma]\label{rem:fiberread}
		The two orders $s$ and $s-1$ both enter the transform, so the
		gauge relation \eqref{eq:gaugeeq} must be imposed on the union
		$S\cup(S-1)$ and not on $S$ alone; openness of $S$ is what
		guarantees real-analyticity of the quotient and hence of its
		periodic extension. The lemma says that the transform determines
		the measure exactly modulo the log-periodic gauge, the
		exponentiated kernel of the unit difference operator.
	\end{remark}
	
	Determinacy is therefore the statement that no nonconstant gauge is
	\emph{realizable}. The next theorem kills the gauge under a weak
	hypothesis: convexity traps the free energy's derivative between
	consecutive increments, so two log-convex solutions of one
	difference equation differ by a periodic function whose oscillation
	is controlled by the growth ratio of the transform.
	
	Part (i) of the next theorem is a specialization of the
	$\Gamma$-type uniqueness theory of Krull~\cite{krull1948} and
	Webster~\cite{webster1997}, whose framework is subsumed by the
	definitive modern treatment of Marichal and
	Zena\"idi~\cite{marichal2022,marichal2024}: applying Webster's
	theorem to $Z_1$ and $Z_2$ separately gives part (i). Part (ii)
	lies outside that theory and appears to be new: it bounds the
	residual gauge when the growth ratio does \emph{not} tend to one,
	so that no limit formula is available.
	
	\begin{restatable}[Rigidity]{theorem}{thmrigidity}\label{thm:rigidity}
		Let $g>0$ and let $Z_1,Z_2$ be positive log-convex solutions of
		\begin{equation}\label{eq:funceq}
			Z(s)=g(s)\,Z(s-1)\qquad\text{on }(s_0,\infty),
		\end{equation}
		with the same $g$. Put
		\begin{equation}\label{eq:etadef}
			\eta\eqdef\limsup_{s\to\infty}
			\log\frac{g(s+1)}{g(s)}\in[0,\infty] .
		\end{equation}
		Then:
		\begin{enumerate}[label=\textup{(\roman*)},leftmargin=2.2em]
			\item if $\eta=0$, there is $c>0$ with $Z_1=cZ_2$;
			\item if $\eta<\infty$, there is $c>0$ with
			\begin{equation}\label{eq:quantgauge}
				e^{-\eta/4}\;\le\;\frac{Z_1(s)}{c\,Z_2(s)}\;\le\;e^{\eta/4}
				\qquad\text{for all }s>s_0+1 .
			\end{equation}
		\end{enumerate}
		Consequently, if $\Leh_\mu=\Leh_\nu$ on a common suddency interval
		$(s_0,\infty)$ and
		\begin{equation}\label{eq:ratiocond}
			\lim_{s\to\infty}\frac{\Leh_\mu(s+1)}{\Leh_\mu(s)}=1,
		\end{equation}
		then $\nu=c\mu$. If instead the limit
		superior in definition \eqref{eq:etadef} for $g=\Leh_\mu$ is a
		finite number $\eta$, then $Z_\nu/Z_\mu$ lies within a factor
		$e^{\pm\eta/4}$ of a constant.
	\end{restatable}
	
	\begin{corollary}[Rigid measures]\label{cor:rigid}
		Let $\mu$ be admissible with $(s_0,\infty)\subset S_\mu$ for some
		$s_0$, so that Theorem~\ref{thm:rigidity} applies to it.
		\begin{enumerate}[label=\textup{(\roman*)},leftmargin=2.2em]
			\item Condition \eqref{eq:ratiocond} holds automatically
			whenever $\esssup_\mu X<\infty$, since in that case
			$\Leh_\mu(s)\to\esssup_\mu X\in(0,\infty)$ as
			$s\to\infty$.
			\item Hence every measure of support bounded above, in
			particular every empirical measure, every probability vector,
			and every scalar spectral measure
			$\langle E(\cdot)v,v\rangle$ of a positive bounded
			operator, is determined by its Lehmer transform up to a
			positive scalar.
			\item Among unbounded laws the condition covers the Gamma and
			Weibull families, and every law whose transform grows
			subgeometrically.
			\item When ${\essinf_\mu X>0}$, the mirror statement as
			$s\to-\infty$ holds through the reflection identity
			\eqref{eq:reflection}.
		\end{enumerate}
	\end{corollary}
	
	Two boundary cases delimit part (ii). Mass accumulating at the
	origin is permitted, since a finite measure supported in $(0,M]$
	still satisfies the standing hypothesis; an atom at the origin
	itself, as for a positive semidefinite operator with nontrivial
	kernel, falls outside a theory that lives on $\Rpos$. For the
	Gamma family in part (iii) the ratio is explicit,
	$\Leh(s+1)/\Leh(s)=(s+k)/(s+k-1)\to1$.
	
	\begin{remark}[The bounded case]\label{rem:hausdorff}
		For compactly supported $\mu$ the conclusion can be reached
		without Theorem~\ref{thm:rigidity}: the cocycle identity
		\eqref{eq:momenttelescope} determines every moment up to one
		constant, and a compactly supported measure is determined by its
		moments. The force of the theorem is in the unbounded case, and
		its quantitative part in the measures failing condition
		\eqref{eq:ratiocond}, where it is the only available control on
		the fiber.
	\end{remark}
	
	The growth condition has a purely information-theoretic reading,
	which shows it to be a statement about ensembles rather than an
	analytic hypothesis imposed from outside.
	
	\begin{restatable}[Merging of escorts]{proposition}{propinfoform}\label{prop:infoform}
		Let $\mu$ be admissible with $(s_0,\infty)\subset S_\mu$. The
		following are equivalent:
		\begin{enumerate}[label=\textup{(\roman*)},leftmargin=2.2em]
			\item condition \eqref{eq:ratiocond} holds;
			\item the free energy is asymptotically affine at unit scale,
			\begin{equation}\label{eq:seconddiff}
				K_\mu(s+1)-2K_\mu(s)+K_\mu(s-1)\longrightarrow0,
				\qquad s\to\infty;
			\end{equation}
			\item consecutive escort ensembles merge in Ces\`aro mean,
			\begin{equation}\label{eq:jeffreysmerge}
				\int_{s}^{s+1}J\bigl(\mu_{u-1},\mu_u\bigr)\dd u
				\longrightarrow0,\qquad s\to\infty .
			\end{equation}
		\end{enumerate}
		A measure is therefore determined by its Lehmer transform as soon
		as its high-suddency escort ensembles become statistically
		indistinguishable.
	\end{restatable}

	The indeterminacy on the far side of the frontier is not a defect of
	method but a fact, and it can be written in closed form.
	
	\begin{restatable}[Lognormal fiber]{theorem}{thmheyde}\label{thm:heyde}
		Let $\mu$ be lognormal with parameters $(m,\sigma^2)$ and, for
		$\varepsilon\in[-1,1]$, define
		\begin{equation}\label{eq:heydedef}
			\dd\mu^{(\varepsilon)}(x)
			=\Bigl[1+\varepsilon\sin\Bigl(
			\frac{2\pi(\log x-m)}{\sigma^{2}}\Bigr)\Bigr]\dd\mu(x).
		\end{equation}
		Then each $\mu^{(\varepsilon)}$ is a probability measure with all
		moments finite,
		\begin{equation}\label{eq:heydeZ}
			Z_{\mu^{(\varepsilon)}}(s)
			=e^{\,ms+\sigma^2s^2/2}
			\Bigl(1+\varepsilon\,e^{-2\pi^{2}/\sigma^{2}}\sin(2\pi s)\Bigr),
		\end{equation}
		and $\Leh_{\mu^{(\varepsilon)}}=\Leh_\mu$ on all of $\R$,
		while the measures are pairwise distinct.
	\end{restatable}
	
	The transform is therefore not injective, even up to scale, even on
	smooth densities with all moments finite, strictly positive
	whenever $|\varepsilon|<1$.
	
	\begin{remark}[Size of the gauge]\label{rem:heydequant}
		For the lognormal, $\eta=\sigma^2$ and the two-sided bound
		\eqref{eq:quantgauge} permits a gauge of logarithmic oscillation
		up to $\sigma^2/2$, while the realized gauge in identity
		\eqref{eq:heydeZ} has oscillation
		$\log\frac{1+\delta}{1-\delta}$ with
		$\delta=|\varepsilon|e^{-2\pi^2/\sigma^2}$: for $\sigma^2=1$
		the bound permits $0.5$ against a realized $5.4\times10^{-9}$.
		Closing the gap is an open problem; the computation suggests the
		realizable gauge is exponentially small in $1/\eta$, of order
		$e^{-c/\eta}$, and we know no obstruction to this and record it
		as a conjecture.
	\end{remark}
	
	The construction of Theorem~\ref{thm:heyde} is not special to the
	lognormal. Its correct level of generality is the group of imaginary
	periods of the transform, which generates a resonant fiber whenever
	it is nontrivial.
	
	\begin{definition}[Period group]\label{def:periodgroup}
		Let $\mu$ be admissible with $\mathring S_\mu\neq\emptyset$. The
		\emph{imaginary period group} of $\mu$, so named because its
		elements $\omega$ are the real frequencies whose purely imaginary
		multiples $i\omega$ are periods of the transform, is
		\begin{equation}\label{eq:Pidef}
			\Theta_\mu\eqdef\bigl\{\omega\in\R:\
			\Leh_\mu(\,\cdot+i\omega)\equiv\Leh_\mu\ \text{on }
			\{\Re s\in\mathring S_\mu\}
			\bigr\} .
		\end{equation}
		On that strip $Z_\mu$ is holomorphic and zero-free on the real
		axis, while $Z_\mu(\cdot-1)$ may vanish off it, so
		$\Leh_\mu$ is meromorphic there and the identity in
		definition \eqref{eq:Pidef} is an identity of meromorphic
		functions.
	\end{definition}
	
	The period group carries more structure than its definition
	suggests; the following proposition records the classification.
	
	\begin{restatable}[Period group structure]{proposition}{propPistructure}\label{prop:Pistructure}
		Let $\mu$ be admissible with $\mathring S_\mu\neq\emptyset$ and
		not concentrated at a single point. Then:
		\begin{enumerate}[label=\textup{(\roman*)},leftmargin=2.2em]
			\item $\Theta_\mu$ is a closed subgroup of $(\R,+)$;
			\item $\Theta_\mu\neq\R$;
			\item consequently either $\Theta_\mu=\{0\}$, or
			$\Theta_\mu=\omega_0\Z$ for a unique $\omega_0>0$.
		\end{enumerate}
	\end{restatable}
	
	A nontrivial period group is thus generated by a single frequency,
	and it is that frequency that drives the modulation.
	
	\begin{restatable}[Resonant fiber]{theorem}{thmresonant}\label{thm:resonant}
		Let $\mu$ be admissible with $\mathring S_\mu\neq\emptyset$. Then:
		\begin{enumerate}[label=\textup{(\roman*)},leftmargin=2.2em]
			\item let $\omega\in\Theta_\mu\setminus\{0\}$, suppose
			$\mu$ is \emph{not} carried by a geometric progression of
			ratio $e^{2\pi/|\omega|}$, and let $\phi\in\R$ make
			$\cos(\omega\log X+\phi)$ non-constant $\mu$-almost
			surely; such $\phi$ exist. Then the measures
			\begin{equation}\label{eq:modulation}
				\dd\mu_\varepsilon\eqdef
				\bigl(1+\varepsilon\cos(\omega\log x+\phi)\bigr)\dd\mu,
				\qquad\varepsilon\in[-1,1],
			\end{equation}
			are admissible, pairwise distinct, and satisfy
			$\Leh_{\mu_\varepsilon}=\Leh_\mu$ on $S_\mu$;
			\item the lognormal with parameters $(m,\sigma^2)$ has
			\begin{equation}\label{eq:Pilognormal}
				\Theta_\mu=\frac{2\pi}{\sigma^2}\,\Z,
			\end{equation}
			and part (i) at the fundamental period
			$\omega=2\pi/\sigma^2$, with the phase
			$\phi=-2\pi m/\sigma^2-\pi/2$, recovers
			Theorem~\ref{thm:heyde};
			\item suppose the meromorphic continuation of $\Leh_\mu$
			is a nonconstant product of a rational function and finitely
			many Gamma factors $\Gamma(a_i+s/p_i)^{\pm1}$ with real
			$a_i$ and real $p_i\neq0$, a class that contains every
			nonconstant rational transform and, through Euler's
			reflection formula, the trigonometric ones. Then either
			$\Leh_\mu(s)=c\,b^{\,s}$ for some $c>0$ and $b>1$, the
			log-affine shape of the lognormal row, in which case
			$\Theta_\mu=\tfrac{2\pi}{\log b}\Z$, or
			$\Theta_\mu=\{0\}$.
		\end{enumerate}
	\end{restatable}
	
	In part (i) already one of $\phi=0$ and $\phi=\pi/2$ serves, the
	fiber of $\Leh_\mu$ therefore contains a nondegenerate
	one-parameter family, and superposing non-constant harmonics
	$k\omega$ enlarges it further. The mechanism is that
	$\omega\in\Theta_\mu$ makes $Z_\mu(\cdot+i\omega)/Z_\mu$ a
	$1$-periodic function, so the modulation shifts the partition
	function by a realizable gauge, positive by the strict modulus
	bound that the exclusion of geometric progressions buys. That
	exclusion separates the two structurally different ways a
	nontrivial period can arise, and only one of them produces a
	fiber.
	
	\begin{restatable}[Two sources of periods]{proposition}{proptwosources}\label{prop:twosources}
		Let $\mu$ be admissible, not concentrated at a single point, with
		$\Theta_\mu=\omega_0\Z\neq\{0\}$ and $R\eqdef e^{2\pi/\omega_0}$.
		Exactly one of the following holds.
		\begin{enumerate}[label=\textup{(\roman*)},leftmargin=2.2em]
			\item \emph{Lattice type.} The measure $\mu$ is carried by a
			geometric progression $\{cR^{k}:k\in\Z\}$. In this case
			$|Z_\mu(a+i\omega_0)|=Z_\mu(a)$ for every real
			$a\in\mathring I_\mu$, and every modulation
			\eqref{eq:modulation} at $\omega=\omega_0$ is $\mu$-almost
			surely constant. Conversely, every measure carried by such a
			progression has $\omega_0\in\Theta_\mu$.
			\item \emph{Resonant type.} The measure $\mu$ is not so
			carried. In this case
			$|Z_\mu(a+i\omega_0)|<Z_\mu(a)$ for every real
			$a\in\mathring I_\mu$, and Theorem~\ref{thm:resonant}(i)
			produces a nondegenerate one-parameter fiber of distinct
			measures sharing the transform of $\mu$.
		\end{enumerate}
	\end{restatable}
	
	The geometric measures $\sum_{k\ge0}\delta_{q^{k}}$, $0<q<1$,
	are of lattice type; the lognormal laws are of resonant type. On a
	lattice-type measure the modulation is constant, so the
	construction of Theorem~\ref{thm:resonant}(i) produces nothing at
	that frequency.

	The dichotomy concerns the \emph{mechanism}, not the outcome: a
	lattice-type measure is immune to the modulation construction, but
	it may still fail to be determined by its transform for an
	unrelated reason, and the following proposition shows that it can.
	
	\begin{restatable}[A lattice companion of the lognormal]{proposition}{proplatticecompanion}
		\label{prop:latticecompanion}
		Fix $\beta>0$ and let
		\begin{equation}\label{eq:latticecompanion}
			\mu=\sum_{k\in\Z}e^{-\beta k^{2}/2}\,\delta_{e^{\beta k}} .
		\end{equation}
		Then $\mu$ is of lattice type with $\omega_0=2\pi/\beta$, and
		\begin{equation}\label{eq:latticeLeh}
			\Leh_\mu(s)=e^{\beta(s-1/2)}
		\end{equation}
		for every real $s$: the lattice measure shares its entire
		Lehmer transform on $\R$ with the lognormal law of parameters
		$(0,\beta)$, which is atomless.
	\end{restatable}
	
	The fiber of the curve \eqref{eq:latticeLeh} is therefore
	nontrivial for a reason the modulation construction does not see.

	\begin{remark}[Classical companions]\label{rem:latticelit}
		Discrete companions of the lognormal carried by a geometric
		progression are classical in the moment-problem literature,
		going back at least to Leipnik~\cite{leipnik1982}, with the
		strong non-uniqueness of the lognormal moment problem treated
		definitively by Berg~\cite{berg1988}; see
		also~\cite{stoyanov2013,linstoyanov2017}. Those results share
		the integer moments. Proposition~\ref{prop:latticecompanion}
		strengthens the statement to the full curve on the real line:
		every real order, and not merely the integer ones, is shared,
		through the $1$-periodicity in $s$ of the theta series
		$\sum_{k\in\Z}e^{-\beta(k-s)^{2}/2}$.
	\end{remark}
	
	The two sides of the frontier are now seen to be mutually exclusive
	in the resonant case, which is the precise version of the consistency
	one expects.
	
	\begin{restatable}[Resonance and growth]{corollary}{corfrontier}\label{cor:frontier}
		Let $\mu$ be a probability law with
		$(s_0,\infty)\subset S_\mu$. If $\Theta_\mu\neq\{0\}$ and
		$\mu$ is of resonant type in the sense of
		Proposition~\ref{prop:twosources}(ii), then the growth
		condition \eqref{eq:ratiocond} fails: the Lehmer curve grows
		at least geometrically along a sequence.
	\end{restatable}

	Three points confirm the coherence of the picture. The lognormal has
	$\Leh(s+1)/\Leh(s)\equiv e^{\sigma^2}>1$, so condition
	\eqref{eq:ratiocond} fails there, exactly as
	Corollary~\ref{cor:frontier} requires. The geometric ensemble
	$\gamma_q\eqdef\sum_{k\ge0}\delta_{q^{\,k}}$ has
	$\Theta_{\gamma_q}=\tfrac{2\pi}{\log(1/q)}\Z\neq\{0\}$ and
	yet satisfies condition \eqref{eq:ratiocond}, its transform tending
	to $1$: not a contradiction but case (i) of
	Proposition~\ref{prop:twosources}, showing that the lattice
	exclusion in Theorem~\ref{thm:resonant}(i) cannot be dropped. And
	the transform's values on any set with an accumulation point in the
	strip determine it throughout, by the identity theorem, so the
	determinacy statements apply verbatim to complex sampling.
	
	What remains open is the converse of Corollary~\ref{cor:frontier}:
	whether every measure whose curve grows geometrically along some
	sequence carries a nontrivial fiber. The two results bracket the
	injectivity frontier at geometric growth; they do not yet meet there.
	Proposition~\ref{prop:latticecompanion} shows that any such
	converse must reach beyond the resonance mechanism, since a fiber can exist where
	resonance is unavailable. This is the first of the open problems
	recorded at the end of the paper.
	
	%=====================================================================
	\section{Characterization}\label{sec:characterization}
	%=====================================================================
	
	For the standard families the transform is elementary, and its
	functional form is a fingerprint: affine growth encodes Gamma-type
	tails, log-affine growth the lognormal, and rational forms the
	compactly supported and power-tailed families, with the pole at the
	tail index. This section records the dictionary and proves that the
	fingerprints characterize.
	
	\begin{restatable}[Dictionary]{proposition}{propdictionary}
		\label{prop:dictionary}
		Table~\ref{tab:dictionary} lists the Lehmer transforms and
		imaginary period groups of the standard families of positive laws,
		each valid on the stated suddency domain.
	\end{restatable}
	
	\begin{table}[t]
		\centering
		\caption{A dictionary of Lehmer transforms: the transform, the
			suddency domain $S_\mu$ and the imaginary period group
			$\Theta_\mu$ of the standard families of positive laws.
			Scale parameters act by homogeneity and are set to one where
			omitted; parameters are constrained so that $S_\mu$ is a
			nonempty open interval.}
		\label{tab:dictionary}
		\small
		\begin{tabular}{@{}llll@{}}
			\toprule
			Law & $\Leh_\mu(s)$ & $S_\mu$ & $\Theta_\mu$\\
			\midrule
			Dirac $\delta_c$ & $c$ & $\R$ & $\R$\\
			Exponential (mean $\theta$) & $\theta\,s$ & $(0,\infty)$
			& $\{0\}$\\
			Gamma$(k,\theta)$ & $\theta\,(s+k-1)$ & $(1-k,\infty)$
			& $\{0\}$\\
			Chi-square$(k)$ & $2s+k-2$ & $(1-\tfrac k2,\infty)$
			& $\{0\}$\\
			Chi$(k)$ & $\sqrt2\,\Gamma\bigl(\tfrac{k+s}2\bigr)
			/\Gamma\bigl(\tfrac{k+s-1}2\bigr)$ & $(1-k,\infty)$
			& $\{0\}$\\
			Weibull$(k,\lambda)$ & $\lambda\,\Gamma(1+\tfrac sk)
			/\Gamma(1+\tfrac{s-1}k)$ & $(1-k,\infty)$ & $\{0\}$\\
			Generalized Gamma$(a,d,p)$ & $a\,\Gamma\bigl(\tfrac{d+s}p\bigr)
			/\Gamma\bigl(\tfrac{d+s-1}p\bigr)$ & $(1-d,\infty)$
			& $\{0\}$\\
			Lognormal$(m,\sigma^2)$ & $e^{\,m+\sigma^2(s-1/2)}$ & $\R$
			& $\tfrac{2\pi}{\sigma^2}\Z$\\
			\midrule
			Uniform$(0,b)$ & $b\,s/(s+1)$ & $(0,\infty)$ & $\{0\}$\\
			Beta$(a,b)$ & $(s+a-1)/(s+a+b-1)$ & $(1-a,\infty)$
			& $\{0\}$\\
			Kumaraswamy$(a,b)$ & $B(1+\tfrac sa,\,b)
			/B(1+\tfrac{s-1}a,\,b)$ & $(1-a,\infty)$ & $\{0\}$\\
			\midrule
			Pareto$(\alpha)$, $x\ge x_m$ & $x_m(\alpha-s+1)/(\alpha-s)$
			& $(-\infty,\alpha)$ & $\{0\}$\\
			Beta-prime$(a,b)$ & $(s+a-1)/(b-s)$ & $(1-a,\,b)$
			& $\{0\}$\\
			F$(d_1,d_2)$ & $\tfrac{d_2}{d_1}\,
			\bigl(s+\tfrac{d_1}2-1\bigr)/\bigl(\tfrac{d_2}2-s\bigr)$
			& $(1-\tfrac{d_1}2,\,\tfrac{d_2}2)$ & $\{0\}$\\
			Inverse-Gamma$(k,\theta)$ & $\theta/(k-s)$ & $(-\infty,k)$
			& $\{0\}$\\
			Fr\'echet$(\alpha)$ & $\Gamma(1-\tfrac s\alpha)
			/\Gamma(1-\tfrac{s-1}\alpha)$ & $(-\infty,\alpha)$
			& $\{0\}$\\
			Log-logistic$(\beta)$ & $\dfrac{s\,\sin(\pi(s-1)/\beta)}
			{(s-1)\sin(\pi s/\beta)}$ & $(1-\beta,\,\beta)$ & $\{0\}$\\
			Half-Cauchy & $\tan(\pi s/2)$ & $(0,1)$ & $\{0\}$\\
			One-sided stable$(\alpha)$ & $(1-s)\,\Gamma(1-\tfrac s\alpha)
			/\Gamma(1+\tfrac{1-s}\alpha)$ & $(-\infty,\alpha)$
			& $\{0\}$\\
			\midrule
			Geometric $\sum_{n\ge0}\delta_{q^{\,n}}$ &
			$(1-q^{\,s-1})/(1-q^{\,s})$ & $(1,\infty)$
			& $\tfrac{2\pi}{\log(1/q)}\Z$\\
			Zeta/Zipf$(\varrho)$ & $\zeta(\varrho-s)/\zeta(\varrho-s+1)$
			& $(-\infty,\varrho-1)$ & $\{0\}$\\
			\bottomrule
		\end{tabular}
	\end{table}
	
	The horizontal blocks collect the exponential-type laws, the
	compactly supported laws, the heavy-tailed laws and the discrete
	ensembles. The chi row contains the half-normal, Rayleigh and
	Maxwell laws at $k=1,2,3$; the stated parameter constraints keep
	the suddency domains nonempty, with $\varrho>2$ placing the
	arithmetic node inside the Zipf domain; and the geometric row is
	the counting measure $\sum_{n\ge0}\delta_{q^{\,n}}$, of
	infinite total mass rather than a probability law.
	
	Four structural features of the table deserve emphasis.
	
	The first is that the position of the singularity encodes the tail:
	the compact block has its poles left of the suddency domain, while
	in the heavy block the divergence sits exactly at the tail index,
	for the half-Cauchy the pole of $\tan(\pi s/2)$ at the Cauchy
	index $s=1$. Proposition~\ref{prop:endpoints} shows this is a
	theorem rather than a coincidence, and the Abelian law below makes
	the heavy-tailed case quantitative.
	
	The second is that the table \emph{multiplies}, by the character
	property: products of independent Gammas and Betas realize rational
	transforms of every degree with real zeros and poles, and
	inversion-symmetric laws produce trigonometric transforms through
	Euler's reflection formula.

	The third is that the table is \emph{escort-closed}: tilting maps it
	into itself, each family into itself or into a neighbouring row, with
	shifted parameters, as in
	\begin{equation}\label{eq:escortclosed}
		\begin{gathered}
			\mathrm{Gamma}(k,\theta)_u=\mathrm{Gamma}(k+u,\theta),\\
			\mathrm{Pareto}(\alpha)_u=\mathrm{Pareto}(\alpha-u),
		\end{gathered}
	\end{equation}
	while the uniform law on $(0,b)$ tilts to $\mathrm{Beta}(1+u,1)$
	rescaled, and the half-Cauchy tilts, for $|u|<1$, to the law of the
	\emph{square root} of a
	$\mathrm{Beta\text{-}prime}\bigl(\tfrac{u+1}{2},
	\tfrac{1-u}{2}\bigr)$ variable. This is \emph{why} closed forms
	exist: the escort mean stays inside a finitely parametrized family
	closed under tilting, and the transform inherits its functional
	form.
	
	The fourth is that the period column is trivial in all but three rows,
	and each of the three is nontrivial for a different reason: the Dirac
	row degenerately, because a constant curve has every frequency as a
	period; the lognormal row by resonance; and the geometric row by a
	lattice of scales. Proposition~\ref{prop:twosources} explains the
	distinction between the last two, and
	Theorem~\ref{thm:resonant} shows why only one of them generates a
	fiber by modulation.
	
	The behaviour at the endpoints of the suddency domain was settled in
	general in Proposition~\ref{prop:endpoints}. The following Abelian
	law makes its first case quantitative for laws with a genuine power
	tail; it is the population fact on which tail-index diagnostics from
	the curve rest. Its interiority hypothesis is automatic when
	$\alpha>1$, since $0\in I_P$ and $I_P$ is an interval with right
	endpoint $\alpha$; when $\alpha\le1$ it is a genuine additional
	requirement on the mass near the origin, of the same nature as a
	negative-moment condition.
	
	\begin{restatable}[Abelian tail law]{proposition}{propabelian}
		\label{prop:abelian}
		Let $P$ be a probability law on $\Rpos$ whose tail is regularly
		varying in the sense of~\cite{bgt1987}, with
		\begin{equation}\label{eq:regvar}
			x^{\alpha}\,\Prob(X>x)\longrightarrow c\in(0,\infty),
			\qquad x\to\infty,
		\end{equation}
		for some $\alpha>0$, and assume $\alpha-1\in\mathring I_P$,
		that is, $\E[X^{\alpha-1-\delta}]<\infty$ for some
		$\delta>0$. Then
		$\sup I_P=\alpha$ and, as $s\uparrow\alpha$,
		\begin{equation}\label{eq:abelianZ}
			\begin{gathered}
				(\alpha-s)\,Z_P(s)\;\longrightarrow\;\alpha c,\\
				(\alpha-s)\,\Leh_P(s)\;\longrightarrow\;
				\frac{\alpha c}{Z_P(\alpha-1)} .
			\end{gathered}
		\end{equation}
		If in addition $Z_P$ continues meromorphically across the line
		$\Re s=\alpha$, the continuation of $\Leh_P$ has a simple pole at
		$s=\alpha$ with residue $-\alpha c/Z_P(\alpha-1)$; without such a
		continuation only the boundary asymptotics limit \eqref{eq:abelianZ} are
		asserted.
	\end{restatable}
	
	\begin{remark}[Interiority at the left order]\label{rem:abelianinterior}
		The interiority hypothesis plays the same role at the left
		order that the divergence hypothesis of
		Proposition~\ref{prop:endpoints} plays at the right one, and
		neither is removable. A tail
		$\Prob(X>x)=x^{-\alpha}\log^{-2}x$ keeps $Z_P(\alpha)$
		finite, so the transform stays bounded at the right endpoint.
		In the same way, a law whose density of $\log X$ decays like
		$t^{-2}e^{(1-\alpha)t}$ as $t\to-\infty$ has
		$Z_P(\alpha-1)<\infty$ with $\alpha-1$ an endpoint of $I_P$
		rather than an interior point; then $Z_P(s-1)=\infty$ for
		every $s<\alpha$, the suddency domain is empty, and the second
		limit in display \eqref{eq:abelianZ} is vacuous, while the
		first continues to hold. Finiteness of $Z_P(\alpha-1)$ alone
		is therefore not enough.
	\end{remark}
	
	We now turn to the characterizations, beginning with the affine row.
	The fingerprint of the Gamma family determines it, by Bohr--Mollerup
	in probabilistic clothing.
	
	\begin{restatable}[Gamma characterization]{theorem}{thmgammachar}\label{thm:gammachar}
		Let $\mu$ be admissible with $I_\mu\supset(-k,\infty)$ and
		\begin{equation}\label{eq:gammahyp}
			\Leh_\mu(s)=\theta\,(s+k-1)\qquad\text{for all }s>1-k,
		\end{equation}
		where $\theta,k>0$. Then
		\begin{equation}\label{eq:gammaconc}
			\dd\mu(x)=c\,x^{k-1}e^{-x/\theta}\dd x
			\qquad\text{for some }c>0 :
		\end{equation}
		up to total mass, $\mu$ is the Gamma law. In particular
		$\Leh_\mu(s)=s/\lambda$ characterizes the exponential law of rate
		$\lambda$.
	\end{restatable}
	
	The Bohr--Mollerup argument above is instructive, but the rigidity
	theory proves more in one stroke: every dictionary row whose
	transform grows subgeometrically is characterized by its curve on a
	right half-line.
	
	\begin{corollary}[Blanket characterization]\label{cor:charall}
		Let $h$ be the transform of a row of
		Table~\ref{tab:dictionary} with $h(s+1)/h(s)\to1$ as
		$s\to\infty$, the case of the Dirac, exponential, Gamma,
		chi-square, chi, Weibull, generalized Gamma, uniform, Beta,
		Kumaraswamy and geometric rows. If $\mu$ is admissible with
		$(s_0,\infty)\subset S_\mu$ for some $s_0$ and
		$\Leh_\mu=h$ on $(s_0,\infty)$, then $\mu$ is the measure of
		that row, up to a positive scalar.
	\end{corollary}
	
	\begin{proof}
		Let $\nu$ be the row's measure, and enlarge $s_0$ so that
		$\Leh_\mu=\Leh_\nu=h$ on $(s_0,\infty)$. Both $Z_\mu$ and
		$Z_\nu$ are positive log-convex solutions of equation
		\eqref{eq:funceq} with $g=h$ there, and $h(s+1)/h(s)\to1$ is
		condition \eqref{eq:ratiocond}, so
		Theorem~\ref{thm:rigidity}(i) gives $Z_\mu=c\,Z_\nu$ on
		$(s_0,\infty)$ for some $c>0$. Lemma~\ref{lem:strip} then
		gives $\mu=c\nu$.
	\end{proof}
	
	The corollary makes the fingerprint reading literal for the entire
	subgeometric block. The rows it does not reach are exactly the ones
	that need more: the lognormal row, whose fiber is real, and the
	heavy-tailed block, whose suddency domains are bounded above.
	
	The log-affine fingerprint is subtler than the affine one, because
	it is the one shape that the log-periodic gauge of
	Lemma~\ref{lem:fiber} can survive; Example~\ref{ex:logaffinecompanion}
	exhibited the surviving gauge explicitly. Within multiplicatively
	infinitely divisible laws the fingerprint nevertheless characterizes.
	
	\begin{restatable}[Lognormal characterization]{theorem}{thmlognormalchar}\label{thm:lognormalchar}
		Let $\mu$ be multiplicatively infinitely divisible with
		$I_\mu=\R$, and suppose
		\begin{equation}\label{eq:lognormalhyp}
			\Lambda_\mu(s)=m+\sigma_0^2\Bigl(s-\tfrac12\Bigr),
			\qquad\sigma_0>0 .
		\end{equation}
		Then $\mu$ is lognormal with parameters $(m,\sigma_0^2)$, up to
		total mass.
	\end{restatable}
	
	\begin{remark}[Beyond the affine case]\label{rem:meijer}
		For rational forms whose poles sit to
		the \emph{right} of the suddency domain, such as Beta-prime, the
		domain is bounded above, the telescoping mechanism toward
		$+\infty$ is unavailable, and classification holds only modulo the
		log-periodic gauge of Lemma~\ref{lem:fiber}. The Pareto row is not
		an example of this difficulty, despite the position of its pole:
		its support is bounded away from the origin, so
		Corollary~\ref{cor:rigid}(iv) applies through the reflection
		identity \eqref{eq:reflection} and the law is rigid.
	\end{remark}
	
	For finitely supported measures, and hence for every empirical
	dataset, the inversion problem has a complete and effective answer.
	Finitely many values of the transform at unit-spaced suddency moments
	recover the data exactly, by Prony's method applied to the cocycle
	products.
	
	\begin{restatable}[Exact recovery]{theorem}{thmprony}
		\label{thm:prony}
		Let $\mu=\sum_{i=1}^{k}a_i\delta_{x_i}$ with distinct $x_i>0$ and
		$a_i>0$, and fix $s_0\in\R$. Then:
		\begin{enumerate}[label=\textup{(\roman*)},leftmargin=2.2em]
			\item the $2k-1$ values
			$\Leh_\mu(s_0+1),\dots,\Leh_\mu(s_0+2k-1)$ determine $\mu$ up
			to total mass, and no smaller number of unit-spaced values
			suffices in general;
			\item explicitly, the cocycle products
			\begin{equation}\label{eq:hj}
				\begin{gathered}
					h_j\eqdef\prod_{l=1}^{j}\Leh_\mu(s_0+l),\\
					j=0,1,\dots,2k-1,
				\end{gathered}
			\end{equation}
			form the moment sequence $h_j=\sum_iw_ix_i^{\,j}$ of the
			$k$-atomic probability measure
			$\tau=\sum_iw_i\delta_{x_i}$ with
			$w_i=a_ix_i^{s_0}/Z_\mu(s_0)$;
			\item the nodes $x_i$ are the roots of the degree-$k$
			polynomial whose coefficient vector spans the kernel of the
			$k\times(k+1)$ Hankel matrix
			$(h_{i+j})_{0\le i\le k-1,\;0\le j\le k}$, the weights $w_i$
			solve the resulting Vandermonde system, and
			$a_i\propto w_ix_i^{-s_0}$.
		\end{enumerate}
	\end{restatable}
	
	\begin{remark}[Moment coordinates]\label{rem:hankel}
		On the integers the transform is the ratio of consecutive moments,
		$\Leh_\mu(k)=m_k/m_{k-1}$ with $m_k=Z_\mu(k)$, and the cocycle
		identity \eqref{eq:momenttelescope} makes the integer Lehmer
		values a complete moment coordinate: log-convexity gives the
		Hankel minor inequalities $m_{k-1}m_{k+1}\ge m_k^2$, the
		recurrence coefficients of the orthogonal polynomials of $\mu$
		are rational in the Lehmer values, and convexity of
		$\Lambda_\mu$ gives the Tur\'an-type inequality
		$\Leh_\mu(s)^2\le\Leh_\mu(s-1)\Leh_\mu(s+1)$. The growth
		condition \eqref{eq:ratiocond} should be compared with the
		classical Stieltjes determinacy criteria of Cram\'er, Carleman,
		Krein and Lin, which constrain the growth of the ratios
		$m_k/m_{k-1}=\Leh_\mu(k)$; see~\cite{linstoyanov2017}. What
		Proposition~\ref{prop:infoform} adds is that the growth relevant
		to determinacy by the transform is a second difference of the
		free energy, equivalently the merging of consecutive escorts.
	\end{remark}
	
	\begin{remark}[Conditioning]\label{rem:conditioning}
		Exact recovery is a statement of exact arithmetic. The Hankel
		and Vandermonde matrices of the reconstruction are severely
		ill-conditioned: the condition number of a positive definite
		Hankel matrix grows at least geometrically in its dimension, by
		the bounds of Beckermann~\cite{beckermann2000}. In finite
		precision the reconstruction is therefore reliable only for
		small numbers of distinct atoms, and
		Theorem~\ref{thm:prony} should be read as a statement about
		the information carried by the curve rather than as a numerical
		algorithm.
	\end{remark}
	%=====================================================================
	\section{Statistical Theory}\label{sec:statistics}
	%=====================================================================
	
	Let $X_1,\dots,X_n$ be observations with common law $P$ on $\Rpos$,
	let $m_a=\E_P[X^a]$, and write
	\begin{equation}\label{eq:popemp}
		\begin{gathered}
			\Leh_P(s)=\frac{m_s}{m_{s-1}},\\
			\widehat\Leh_n(s)=\frac{\sum_{i=1}^nX_i^{\,s}}
			{\sum_{i=1}^nX_i^{\,s-1}}
		\end{gathered}
	\end{equation}
	for the population and empirical transforms; everything so far
	applies with $\mu=P$ or $\mu=\widehat P_n$. Two regimes organize
	the sampling theory developed in this section: a Donsker regime on
	suddency intervals where enough moments exist, and a critical regime
	at the tail index where the population transform diverges.
	
	The entrance to the Donsker regime is an exact algebraic
	linearization, through which every subsequent result flows:
	\begin{equation}\label{eq:linearize}
		\begin{gathered}
			\widehat\Leh_n(s)-\Leh_P(s)
			=\frac{1}{\widehat m_{s-1}}\cdot
			\frac1n\sum_{i=1}^{n}X_i^{\,s-1}\bigl(X_i-\Leh_P(s)\bigr),\\
			\widehat m_a=\frac1n\sum_iX_i^a,
		\end{gathered}
	\end{equation}
	whose summands have mean zero. The \emph{influence function} of the
	transform at suddency $s$ is therefore
	\begin{equation}\label{eq:if}
		\psi_s(x)=\frac{x^{\,s-1}\bigl(x-\Leh_P(s)\bigr)}{m_{s-1}} .
	\end{equation}
	
	\begin{restatable}[Central limit theorem]{theorem}{thmclt}\label{thm:clt}
		Fix $s$ with $2s,2s-2\in\mathring I_P$, which already forces
		$s\in\mathring S_P$ because $0\in I_P$ and $I_P$ is an interval.
		Then $\widehat\Leh_n(s)\to\Leh_P(s)$ almost surely, and
		\begin{equation}\label{eq:cltlimit}
			\begin{gathered}
				\sqrt n\,\bigl(\widehat\Leh_n(s)-\Leh_P(s)\bigr)
				\;\xrightarrow{\ d\ }\;
				\mathcal N\bigl(0,v(s)\bigr),\\
				v(s)=\E_P\bigl[\psi_s(X)^{2}\bigr] .
			\end{gathered}
		\end{equation}
		For the exponential law of mean $\theta$,
		\begin{equation}\label{eq:expvar}
			v(s)=\frac{\theta^{2}s^{2}\,\Gamma(2s-1)}{\Gamma(s)^{2}},
			\qquad s>\tfrac12 .
		\end{equation}
	\end{restatable}
	
	The influence function settles robustness quantitatively. Since
	$\psi_s$ grows like $x^{s}$ at infinity and like
	$-\Leh_P(s)x^{s-1}$ at the origin, the suddency moment is a
	continuous dial trading upper-tail against lower-tail sensitivity,
	with the exact efficiency cost supplied by $v(s)$; it is a trade
	rather than an escape, since $\psi_s$ is unbounded for every $s$
	and no choice of suddency moment makes the empirical transform
	B-robust. The finite-sample bias is equally explicit, and is an
	increment of the transform itself.
	
	\begin{restatable}[Leading bias]{proposition}{propbias}
		\label{prop:bias}
		Fix $s$ with $4s,4s-4\in\mathring I_P$. If
		$\supp P\subset[m,M]\subset\Rpos$, then, as $n\to\infty$,
		\begin{equation}\label{eq:bias}
			\E\bigl[\widehat\Leh_n(s)\bigr]-\Leh_P(s)
			=-\frac1n\,\frac{m_{2s-2}}{m_{s-1}^{2}}\,
			\Bigl(\Leh_P(2s-1)-\Leh_P(s)\Bigr)+O(n^{-2}).
		\end{equation}
		In general, if in addition
		$\sup_n\E\bigl[\widehat\Leh_n(s)^{p}\bigr]<\infty$ for
		some $p>2$, then identity \eqref{eq:bias} holds with remainder
		$O(n^{-2+2/p})$, which is $o(n^{-1})$.
	\end{restatable}
	
	Since $\Leh_P$ is nondecreasing and $2s-1>s$ exactly when $s>1$,
	the empirical transform is biased \emph{downward} for $s>1$ and
	\emph{upward} for $s<1$, with magnitude the transform's own
	increment over $[s,2s-1]$; the arithmetic node is exactly unbiased,
	as it must be, $\widehat\Leh_n(1)$ being the sample mean. The
	ratio-moment hypothesis in the unbounded case is not decorative:
	for $P$ Pareto of index $\tfrac12$ and $s=\tfrac1{10}$ both
	moment conditions hold, yet $\widehat\Leh_n(s)$ dominates
	$\min_iX_i$, whose expectation is infinite for $n\le2$, so the
	left side of formula \eqref{eq:bias} is $+\infty$ for small $n$
	and only the eventual behaviour can be asserted.
	
	\begin{corollary}[Bias-corrected estimator]\label{cor:biascorrect}
		In the bounded-support case of Proposition~\ref{prop:bias},
		the corrected estimator
		\begin{equation}\label{eq:biascorrect}
			\widetilde\Leh_n(s)\eqdef\widehat\Leh_n(s)
			+\frac1n\,\frac{\widehat m_{2s-2}}{\widehat m_{s-1}^{\,2}}\,
			\Bigl(\widehat\Leh_n(2s-1)-\widehat\Leh_n(s)\Bigr)
		\end{equation}
		satisfies
		$\E\bigl[\widetilde\Leh_n(s)\bigr]-\Leh_P(s)=O(n^{-2})$.
		In the unbounded case, if in addition $4s-2\in\mathring I_P$
		and the ratio-moment hypothesis of the proposition holds at the
		orders $s$ and $2s-1$, the bias of estimator
		\eqref{eq:biascorrect} is $o(n^{-1})$.
	\end{corollary}
	
	\begin{proof}
		The correction term carries an explicit factor $n^{-1}$, and the
		empirical quantities multiplying it converge to their population
		versions with $O(n^{-1})$ bias, by Proposition~\ref{prop:bias}
		and the same argument applied to the moments; in the bounded
		case all remainders are uniform. The expectation of the
		correction is therefore
		$\tfrac1n\,\tfrac{m_{2s-2}}{m_{s-1}^{2}}
		\bigl(\Leh_P(2s-1)-\Leh_P(s)\bigr)+O(n^{-2})$, which cancels
		the leading term of expansion \eqref{eq:bias}.
	\end{proof}
	
	The pointwise theory upgrades to a process-level one, at the cost of
	strengthening the moment and mixing hypotheses to what a maximal
	inequality for the increments actually requires.
	
	\begin{restatable}[Functional limits]{theorem}{thmfclt}
		\label{thm:fclt}
		Let $T\subset\mathring S_P$ be compact with
		$2T\cup(2T-2)\subset\mathring I_P$.
		\begin{enumerate}[label=\textup{(\roman*)},leftmargin=2.2em]
			\item In the independent case, in $C(T)$,
			\begin{equation}\label{eq:fcltiid}
				\sqrt n\bigl(\widehat\Leh_n-\Leh_P\bigr)
				\Longrightarrow\mathbb G,
			\end{equation}
			a centered Gaussian process with covariance kernel
			$\mathcal K(s,t)=\E_P[\psi_s\psi_t]$.
			\item If instead $(X_i)$ is strictly stationary and strongly
			mixing with coefficients satisfying, for some
			$\delta\in(0,2)$ and some $a>1+2/\delta$,
			\begin{equation}\label{eq:mixing}
				\alpha_{\mathrm{mix}}(k)=O\bigl(k^{-a}\bigr),
				\qquad\text{so that}\qquad
				\sum_{k\ge1}\alpha_{\mathrm{mix}}(k)^{\delta/(2+\delta)}
				<\infty,
			\end{equation}
			and if the envelope of the influence functions has a
			$(2+\delta)$-th moment, that is,
			\begin{equation}\label{eq:envelopemoment}
				(2+\delta)\,T\ \cup\
				\bigl((2+\delta)\,T-(2+\delta)\bigr)
				\subset\mathring I_P ,
			\end{equation}
			then limit \eqref{eq:fcltiid} holds with the long-run kernel
			\begin{equation}\label{eq:longrun}
				\mathcal K_{\mathrm{LR}}(s,t)
				=\sum_{k\in\Z}\Cov\bigl(\psi_s(X_0),\psi_t(X_k)\bigr)
			\end{equation}
			in place of $\mathcal K$; the series \eqref{eq:longrun}
			converges absolutely under condition \eqref{eq:mixing} by
			Davydov's covariance inequality~\cite{davydov1968}.
		\end{enumerate}
	\end{restatable}
	
	The restriction $\delta<2$ normalizes the bracketing argument in
	the proof and costs nothing when $a>2$, since the hypotheses then
	also hold at some $\delta'\in(0,2)$.
	
	The empirical curve is the ratio of the empirical moment generating
	function of $\log X$ at the orders $s$ and $s-1$, so
	Theorem~\ref{thm:fclt} belongs to the process theory of the
	empirical moment generating function begun by
	Cs\"org\H{o}~\cite{csorgo1982}, developed for the empirical
	Laplace transform by Cs\"org\H{o} and
	Teugels~\cite{csorgoteugels1990}, and used inferentially by
	Feuerverger~\cite{feuerverger1989}. The independent case of the
	theorem can also be deduced from that theory by the delta method
	applied to the ratio; the mixing case, with its long-run
	covariance, appears not to have been recorded in this form.
	
	Two consequences are immediate. Fitting $\widehat\Leh_n$ to a
	closed-form transform from Table~\ref{tab:dictionary} at finitely
	many suddency moments gives a $\sqrt n$-consistent, asymptotically
	normal minimum-distance estimator with sandwich covariance built
	from $\mathcal K$; affine and log-affine fits estimate the Gamma
	and lognormal parameters. And the residual of the fit is a
	goodness-of-fit statistic: on a grid $s_1<\dots<s_p$, $p\ge3$,
	with kernel matrix
	$\mathcal K_p=(\mathcal K(s_i,s_j))_{i,j\le p}$ nonsingular, the
	efficiently weighted fit has residual quadratic form asymptotically
	$\chi^2_{p-2}$ under the null.
	
	Two cautions attach to that statement, and both are instructive. The
	first is that the $\chi^2_{p-2}$ limit belongs to the efficient
	weighting alone; with an arbitrary weight matrix the limit is a
	weighted sum of independent $\chi^2_1$ variables, and since
	$\mathcal K_p$ is severely ill-conditioned on a fine grid, the
	approximation requires large samples. The second is a matter of what
	is being tested. Affinity and log-affinity \emph{characterize} the
	Gamma and lognormal families by
	Theorems~\ref{thm:gammachar} and~\ref{thm:lognormalchar}, but they do
	so on a half-line and, in the lognormal case, only within the
	multiplicatively infinitely divisible laws. A finite grid tests
	affinity on that grid, which is necessary for the Gamma family and
	sufficient only in the limit of a grid exhausting a half-line. The
	null hypothesis of the log-affine test, moreover, is not the
	lognormal law but its whole Heyde fiber: by Theorem~\ref{thm:heyde}
	the members of that fiber share the curve exactly, and no test based
	on the curve can separate them.
	
	The last limitation is worth recording formally; it is immediate
	from the definitions, and by the rigidity theory it is
	irreducible.
	
	\begin{corollary}[Fiber blindness]\label{cor:fiberblind}
		Any functional of the law $P$ that depends on $P$ only through
		$\Leh_P$ is constant on each fiber of the transform.
		Consequently every population target expressible through the
		curve, and every test whose null hypothesis is a property of
		the curve, is blind to the differences within the Heyde fiber
		of Theorem~\ref{thm:heyde}.
	\end{corollary}
	
	Beyond the Gaussian scale one wants tail bounds. Two are available,
	and they must be separated, because the exponential-family
	argument that gives the sharp asymptotic rate requires a hypothesis
	that heavy-tailed data do not satisfy, whereas a distribution-free
	bound is available for bounded data with no hypothesis at all.
	
	\begin{restatable}[Finite-sample bounds]{proposition}{prophoeffding}\label{prop:hoeffding}
		Suppose $\supp P\subset[m,M]\subset\Rpos$ with $m<M$, and fix
		$s\in\R$. For $y>0$ put $g_y(x)=x^{s}-y x^{s-1}$ and let
		\begin{equation}\label{eq:oscg}
			V_s(y)\eqdef\sup_{[m,M]}g_y-\inf_{[m,M]}g_y
			\;\le\;\bigl|M^{s}-m^{s}\bigr|
			+y\,\bigl|M^{s-1}-m^{s-1}\bigr| .
		\end{equation}
		Then for every $y>\Leh_P(s)$ and every $n\ge1$,
		\begin{equation}\label{eq:hoeffdingupper}
			\Prob\bigl(\widehat\Leh_n(s)\ge y\bigr)
			\;\le\;\exp\Bigl\{-\frac{2n\,m_{s-1}^{2}\,
				\bigl(y-\Leh_P(s)\bigr)^{2}}{V_s(y)^{2}}\Bigr\},
		\end{equation}
		and symmetrically, for every $y<\Leh_P(s)$,
		\begin{equation}\label{eq:hoeffdinglower}
			\Prob\bigl(\widehat\Leh_n(s)\le y\bigr)
			\;\le\;\exp\Bigl\{-\frac{2n\,m_{s-1}^{2}\,
				\bigl(\Leh_P(s)-y\bigr)^{2}}{V_s(y)^{2}}\Bigr\} .
		\end{equation}
	\end{restatable}

	Bounds \eqref{eq:hoeffdingupper} and \eqref{eq:hoeffdinglower}
	require no moment or smoothness hypothesis beyond boundedness, and
	they hold for every $n$. The exponents involve the population
	quantities $m_{s-1}$ and $\Leh_P(s)$, so the bounds are oracle
	bounds, and empirical-Bernstein substitutes can replace the oracle
	constants at the usual cost. When exponential moments of the relevant
	observable exist, the asymptotic rate is exactly identified.
	
	\begin{proposition}[Large deviations]\label{prop:ld}
		Fix $s\in\mathring S_P$, so that $m_s$ and $m_{s-1}$ are finite,
		fix $y>\Leh_P(s)$, put $g_y(x)=x^{s}-y x^{s-1}$, and assume the
		Cram\'er condition
		\begin{equation}\label{eq:cramercond}
			\E_P\bigl[e^{\lambda g_y(X)}\bigr]<\infty
			\qquad\text{for some }\lambda>0 .
		\end{equation}
		Then
		\begin{equation}\label{eq:ld}
			\begin{gathered}
				\lim_{n\to\infty}\frac1n\log
				\Prob\bigl(\widehat\Leh_n(s)\ge y\bigr)
				=-I_s(y),\\
				I_s(y)=-\inf_{\lambda\in\R}
				\log\E_P\bigl[e^{\lambda g_y(X)}\bigr],
			\end{gathered}
		\end{equation}
		a rate vanishing at $y=\Leh_P(s)$, nondecreasing on
		$[\Leh_P(s),\esssup_PX)$, and locally quadratic with the
		central-limit variance $v(s)$ of Theorem~\ref{thm:clt}.
	\end{proposition}
	
	The upper bound
	$\Prob(\widehat\Leh_n(s)\ge y)\le e^{-nI_s(y)}$ holds for every
	$n$ with no condition at all, by Chernoff's inequality; the
	Cram\'er condition \eqref{eq:cramercond} is needed only for the
	matching lower bound, and it is a genuine restriction, failing for
	every $\lambda>0$ already for $P$ exponential at $s=2$, where the
	rate degenerates to zero. Convexity of the rate in $y$ does not survive
	the contraction through $(a,b)\mapsto a/b$, nor could it: for
	bounded data the rate saturates at $-\log\Prob(X=\esssup_PX)$,
	so a nonconstant $I_s$ cannot be both convex and bounded. What
	survives is monotonicity and the local quadratic behaviour.
	
	\begin{proof}
		The event coincides with $\{\sum_ig_y(X_i)\ge0\}$, since the
		denominator $\sum_iX_i^{s-1}$ is positive; the summands are
		independent and identically distributed with common mean
		$m_s-y\,m_{s-1}=-m_{s-1}\bigl(y-\Leh_P(s)\bigr)<0$.
		Cram\'er's theorem~\cite{dembo1998}, applicable
		under condition \eqref{eq:cramercond}, identifies the rate as the
		Fenchel--Legendre transform of the cumulant generating function of
		$g_y(X)$ evaluated at $0$, which is the displayed expression.
	\end{proof}
	
	The name \emph{statistic-generating function} acquires here a
	second, decision-theoretic content: the curve is a lossless
	encoding of the sample. The sufficiency itself is classical and
	resides in the order statistic; what the proposition adds is that
	the curve is an exact re-coordinatization of it, with the passage
	to $2n-1$ unit-spaced values, a list whose length is fixed by the
	sample size alone, and the number of distinct values recovered as a
	Hankel rank. The passage is a re-coordinatization and not a
	reduction: the list is longer than the order statistic it
	re-encodes, so no tension arises with the classical
	Koopman--Pitman--Darmois restriction on fixed-dimension
	sufficiency.
	
	\begin{restatable}[Sufficiency]{proposition}{propsufficient}\label{prop:sufficient}
		Let $X_1,\dots,X_n$ be an i.i.d.\ sample of known size $n$ from
		any law on $\Rpos$, containing $k$ distinct values. Then:
		\begin{enumerate}[label=\textup{(\roman*)},leftmargin=2.2em]
			\item the $2n-1$ consecutive unit-spaced evaluations
			$\widehat\Leh_n(s_0+1),\dots,\widehat\Leh_n(s_0+2n-1)$
			determine the sample up to permutation, hence form a
			sufficient statistic whose dimension $2n-1$ is deterministic,
			fixed by the sample size and not by the data. In fact the
			first $2k-1$ of them already suffice, with $k$ recovered as
			the rank of the Hankel matrix of the cocycle products; but
			$k$ is a function of the data, so the truncated list does not
			have deterministic dimension;
			\item consequently the whole empirical curve
			$\widehat\Leh_n(\cdot)$ is a sufficient statistic, and is
			minimal sufficient for the full nonparametric model, being a
			measurable bijection of the vector of order statistics.
		\end{enumerate}
	\end{restatable}

	For exponential families the transform interacts with the natural
	parameter in the simplest conceivable way: tilting the law translates
	the suddency axis.
	
	\begin{restatable}[Translation covariance]{theorem}{thmtranslation}\label{thm:translation}
		Let $\{P_\eta\}$ be the exponential family generated by a base law
		$P_0$ with sufficient statistic $\log X$, that is,
		\begin{equation}\label{eq:expfam}
			\dd P_\eta=\frac{x^{\eta}\dd P_0}{Z_0(\eta)} .
		\end{equation}
		Then
		\begin{equation}\label{eq:translation}
			\Leh_{P_\eta}(s)=\Leh_{P_0}(s+\eta)
			\qquad\text{on }S_{P_0}-\eta :
		\end{equation}
		the natural parameter acts by translation of the suddency axis.
		Consequently:
		\begin{enumerate}[label=\textup{(\roman*)},leftmargin=2.2em]
			\item inference for $\eta$ is curve registration: the maximum
			likelihood estimator solves
			\begin{equation}\label{eq:mle}
				K_0'(\widehat\eta)=\frac1n\sum_{i=1}^n\log X_i
			\end{equation}
			and, in Lehmer terms, shifts the population curve onto the
			empirical one;
			\item the Fisher information for $\eta$ is
			$I(\eta)=\Var_{P_\eta}(\log X)$, so by
			Theorem~\ref{thm:jeffreys} the slope of the log-Lehmer curve
			is exactly the unit-averaged Fisher information of the model,
			\begin{equation}\label{eq:infoslope}
				\Lambda_{P_0}'(s)=\int_{s-1}^{s}I(u)\dd u :
			\end{equation}
			the steep parts of the curve are where the family is most
			informative;
			\item the escort closure recorded in
			identity \eqref{eq:escortclosed} is the statement that the classical
			families are unions of orbits of this translation action.
		\end{enumerate}
	\end{restatable}

	For a general exponential family with sufficient statistic $T$, the
	substitution $X=e^{T}$ reduces it to the log-linear case: the Lehmer
	transform of the exponentiated sufficient statistic is the canonical
	curve of the family, translated by the natural parameter.
	
	We turn now to the critical regime. For a regularly varying law the
	population curve diverges at the tail index while the empirical
	curve is finite everywhere; the reconciliation is a phase transition
	in the $(s,\log n)$ plane.
	
	\begin{restatable}[Phase diagram]{theorem}{thmphase}\label{thm:phase}
		Let $\Prob(X>x)=x^{-\alpha}\ell(x)$ with $\alpha>0$ and $\ell$
		slowly varying in the sense of~\cite{bgt1987}, and assume in
		addition that
		$\E[X^{-\delta}]<\infty$ for some $\delta>0$. Then for every
		$s\in(1-\delta,\infty)\setminus\{\alpha,\alpha+1\}$,
		\begin{equation}\label{eq:phasediagram}
			\frac{\log\widehat\Leh_n(s)}{\log n}
			\;\xrightarrow{\ \Prob\ }\;
			\rho_\alpha(s)=\min\Bigl(\frac1\alpha,\
			\max\Bigl(0,\ \frac s\alpha-1\Bigr)\Bigr):
		\end{equation}
		a piecewise linear limit, flat below $s=\alpha$, of slope
		$1/\alpha$ on $(\alpha,\alpha+1)$, and flat at height $1/\alpha$
		above $s=\alpha+1$, with knees exactly at the tail index and its
		unit shift.
	\end{restatable}
	
	The extra hypothesis $\E[X^{-\delta}]<\infty$ cannot be dropped when
	$s<1$: the proof compares $\sum_iX_i^{s}$ with $\sum_iX_i^{s-1}$,
	and for $s<1$ the second exponent is negative, so a law whose density
	blows up fast enough at the origin has
	$\E[X^{s-1}]=\infty$ and the denominator is no longer of order $n$.
	
	\begin{remark}[Uniformity]\label{rem:uniform}
		Convergence in limit \eqref{eq:phasediagram} is uniform on
		compact subsets of $(1-\delta,\infty)$, the knees included.
		Indeed $\log\widehat\Leh_n(\cdot)/\log n$ is nondecreasing, by
		Theorem~\ref{thm:mono} applied to the empirical measure, and
		the limit $\rho_\alpha$ is continuous with Lipschitz constant
		$1/\alpha$, so pointwise convergence in probability along a
		fine grid avoiding the two knees upgrades to uniform
		convergence on compacts by the monotone sandwich between
		consecutive grid points.
	\end{remark}
	
	The mechanism of the diagram is the competition between
	law-of-large-numbers and extreme-value behaviour in the power sums
	$\sum_ie^{s\log X_i}$, and it is not special to the present
	setting: it is the free-energy phase transition of Derrida's random
	energy model~\cite{derrida1981}, whose rigorous theory for sums of
	random exponentials, including fluctuation theorems at the critical
	temperature, is developed by Ben~Arous, Bogachev and
	Molchanov~\cite{benarous2005}. That theory is the natural source
	for a second-order analysis at the knees, which we leave open.
	
	The middle segment, whose slope is the reciprocal of the tail index,
	yields a consistent estimator.
	
	\begin{corollary}[Tail-index estimation]\label{cor:tail}
		Under the hypotheses of Theorem~\ref{thm:phase}, fix $s<t$ in
		$(\alpha,\alpha+1)$ and set
		\begin{equation}\label{eq:tailest}
			\widehat\alpha_{s,t}
			\eqdef\frac{(t-s)\log n}
			{\log\widehat\Leh_n(t)-\log\widehat\Leh_n(s)} .
		\end{equation}
		Then $\widehat\alpha_{s,t}\to\alpha$ in probability.
	\end{corollary}
	
	The requirement that the window lie inside $(\alpha,\alpha+1)$
	cannot be dispensed with: outside it Theorem~\ref{thm:phase} gives
	$\log\widehat\Leh_n(t)-\log\widehat\Leh_n(s)=o_{\Prob}(\log n)$
	and the estimator diverges. The corollary is therefore an oracle
	statement; no data-driven window selection is developed here.
	
	\begin{proof}
		By Theorem~\ref{thm:phase}, for $\alpha<s<t<\alpha+1$ the
		difference $\log\widehat\Leh_n(t)-\log\widehat\Leh_n(s)$ divided
		by $\log n$ converges in probability to
		\begin{equation}\label{eq:tailslope}
			\rho_\alpha(t)-\rho_\alpha(s)=\frac{t-s}{\alpha}>0,
		\end{equation}
		and the continuous mapping theorem applied to
		$x\mapsto(t-s)/x$ at the strictly positive limit gives the claim.
	\end{proof}
	
	In the pure Pareto case the rate of this estimator can be identified
	exactly, and it is logarithmic.
	
	\begin{restatable}[Rate of the estimator]{proposition}{proptailrate}
		\label{prop:tailrate}
		Let $P$ be the exact Pareto law $\Prob(X>x)=x^{-\alpha}$ for
		$x\ge1$, and fix $\alpha<s<t<\alpha+1$. Then
		\begin{equation}\label{eq:tailrate}
			\widehat\alpha_{s,t}-\alpha=O_{\Prob}\Bigl(\frac1{\log n}\Bigr).
		\end{equation}
	\end{restatable}
	
	\begin{remark}[Using the diagram]\label{rem:hill}
		The logarithmic rate \eqref{eq:tailrate} is slower than the
		$k^{-1/2}$ rate of Hill's estimator~\cite{hill1975} on $k$
		upper order statistics, and no reordering of the argument
		improves it, since the index enters the diagram only through the
		exponent of a power of $n$. Estimators reading the tail index
		from the growth of whole-sample power sums go back to
		Meerschaert and Scheffler~\cite{meerschaert1998}, whose
		$\log\sum_iX_i^{2}/\log n$ estimator is the section of the
		diagram at suddency order two, with the same rate; the
		moment-ratio estimators of~\cite{danielsson1996,dekkers1989}
		use ratios of moments on the order statistics beyond a
		threshold and achieve $\sqrt k$ rates. What the diagram adds is
		structural: the full curve, and the two knees at known unit
		separation, an internal consistency check and a diagnostic for
		departures from exact regular variation that neither line
		possesses. It is accordingly a diagnostic, not a competitor at
		estimation.
	\end{remark}
	
	The theory so far has taken the data as raw positive numbers. Applied
	instead to a probability vector, the
	transform becomes a concentration profile. For
	$\pp=(p_1,\dots,p_n)$ with self-referential measure
	$\mu=\sum_i\delta_{p_i}$, the counting measure carried by the
	probability values themselves, the curve
	\begin{equation}\label{eq:profile}
		\Leh_{\pp}(s)=\frac{\sum_ip_i^{\,s}}{\sum_ip_i^{\,s-1}}
	\end{equation}
	increases from $\min_ip_i$ to $\max_ip_i$, strictly unless $\pp$
	is uniform. It interpolates the standard indices: the mean
	probability $1/n$ at $s=1$, the Herfindahl--Simpson index
	$\sum_ip_i^2$ at $s=2$, and Berger--Parker dominance in the limit
	$s\to\infty$. Its relation to the entropy
	spectrum is exact.
	
	\begin{restatable}[R\'enyi and Hill]{proposition}{proprenyihill}\label{prop:renyihill}
		With $H_s$ the R\'enyi entropy of order $s$ and
		${}^{q}D=e^{H_q}$ the Hill diversity
		number~\cite{renyi1961,hillnumbers1973},
		\begin{equation}\label{eq:renyi}
			\Lambda_{\pp}(s)=(1-s)H_s(\pp)-(2-s)H_{s-1}(\pp),
		\end{equation}
		and for every integer $q\ge2$,
		\begin{equation}\label{eq:hillnumber}
			{}^{q}D=\Bigl(\prod_{j=2}^{q}\Leh_{\pp}(j)\Bigr)^{1/(1-q)} .
		\end{equation}
	\end{restatable}

	The escort measures are exactly the escort distributions of
	nonextensive thermostatistics~\cite{tsallis2009}, with the
	suddency moment as R\'enyi order shifted by one. The profile is a
	\emph{complete} invariant: by Corollary~\ref{cor:rigid} it
	determines $\pp$ up to relabeling, and its slope is an evenness
	gradient, vanishing identically if and only if $\pp$ is uniform.
	
	One comparison closes the section. Since
	$\Leh_\mu(s)=\E_{\mu_{s-1}}[X]$ increases from the mean at $s=1$
	to the essential supremum, the transform on $s\ge1$ is a smooth,
	law-invariant, positively homogeneous family of tail statistics,
	with formula \eqref{eq:if} as its exact influence function. Unlike
	the superquantile family of Rockafellar and
	Uryasev~\cite{rockafellar2002}, however, it is \emph{not} monotone
	with respect to first-order stochastic dominance, and two points
	already witness the failure: the data $(1,10)$ have contraharmonic
	mean $\tfrac{101}{11}$, while raising the smaller datum to $2$
	\emph{lowers} it to $\tfrac{104}{12}$. A pointwise increase of the
	data can decrease the statistic, so the transform is a
	dispersion-sensitive soft maximum rather than a risk measure in the
	axiomatic sense.
	
	%=====================================================================
	\section{Conclusion}\label{sec:conclusion}
	%=====================================================================
	
	This paper developed the statistical theory of the Lehmer transform,
	the statistic-generating curve of a positive dataset. An exact
	divergence calculus makes the slope of the curve a Jeffreys divergence
	and the arithmetic--geometric gap a relative entropy. A quantitative
	rigidity theory locates the frontier of identifiability at geometric
	growth of the curve, and the period-group dichotomy organizes the
	indeterminate side of that frontier. A dictionary of closed forms
	carries characterization theorems for the Gamma and lognormal
	families. Exact finite-sample inversion and the sufficiency theorem
	exhibit the empirical curve as a lossless re-coordinatization of the
	order statistic. And the sampling theory of the empirical curve runs
	from the influence function and the exact leading bias to limit
	theorems, finite-sample bounds and the heavy-tail phase diagram.
	
	One identity organized these results: the logarithm of the transform
	is the unit increment of a convex free energy. The unit difference
	telescopes exactly across integer spans, so that finitely many of its
	values reconstruct the partition function itself. It requires no
	differentiability, and so is defined for every admissible measure.
	And it is a ratio of two evaluations of one functional rather than a
	limit of such ratios, which is what makes deconvolution a division
	and estimation a ratio of two power sums.
	
	Two questions seem to us the most important left open. The first is
	the injectivity frontier. Subgeometric growth of the curve forces
	determinacy, resonant imaginary periods destroy it, and the two are
	mutually exclusive; what is missing is the converse, that every
	measure whose curve grows geometrically along some sequence carries a
	nontrivial fiber. A lattice measure sharing the lognormal curve shows
	that any proof must reach beyond the resonance mechanism. The second
	is the fluctuation theory at the knees of the phase diagram, where
	Gaussian and stable regimes exchange. A second-order theory there,
	with a data-driven choice of the suddency window, would turn the
	diagram from a pilot and a diagnostic into an efficient estimator.
	
	\bmhead{Data Availability}
	No new data were produced in this work.
	
	\bmhead{Conflict of Interest}
	The author declares no conflict of interest.
	
	\bmhead{Funding Declaration}
	This research has received no funding.
	
	\bmhead{Author Contributions}
	All aspects of this work, conceptualization, methodology, software,
	validation, formal analysis, investigation, writing, and
	visualization, were carried out by M.A.

	\clearpage
	
	\begin{appendices}
		
		%=====================================================================
		\section{Proofs}\label{app:proofs}
		%=====================================================================

		This appendix collects the proofs deferred from the body, in the
		order in which the results are used. Each result is restated
		before it is proved.
		
		We begin with the regularity of the partition function, on which
		every later argument rests.
		
		\lemlogconvex*
		
		\begin{proof}
			\emph{Assertion (i).} For $s_0,s_1\in I_\mu$ and
			$\theta\in(0,1)$ put $s_\theta=(1-\theta)s_0+\theta s_1$.
			H\"older's inequality with exponents $1/(1-\theta)$ and
			$1/\theta$, applied to the factorization
			$x^{s_\theta}=(x^{s_0})^{1-\theta}(x^{s_1})^{\theta}$, gives
			\begin{equation}\label{eq:holder}
				Z_\mu(s_\theta)\le
				Z_\mu(s_0)^{1-\theta}\,Z_\mu(s_1)^{\theta}<\infty .
			\end{equation}
			
			\emph{Assertion (ii).} Equality in bound \eqref{eq:holder}
			forces $x^{s_0-s_1}$ to be $\mu$-almost everywhere constant,
			that is, $\mu$ concentrated at a point; otherwise the
			inequality is strict and, once differentiation under the
			integral is justified,
			\begin{equation}\label{eq:Kpp}
				K_\mu''(u)=\Var_{\mu_u}(\log X)>0 .
			\end{equation}
			
			\emph{Assertion (iii).} Fix $[a,b]\subset\mathring I_\mu$ and
			choose $\varepsilon>0$ with
			$a-\varepsilon,b+\varepsilon\in I_\mu$. The first inequality in
			bound \eqref{eq:domination} is immediate, and the second holds
			because powers of $|\log x|$ are absorbed by
			$x^{\pm\varepsilon}$ at both ends of $\Rpos$.
			
			\emph{Assertion (iv).} Dominated convergence, applied through
			bound \eqref{eq:domination}, justifies complex differentiation
			under the integral to all orders, giving holomorphy on the
			strip; on the real axis the same domination gives
			$K_\mu\in C^\infty(\mathring I_\mu)$.
			
			\emph{Assertion (v).} The set $S_\mu$ is the intersection of two
			intervals, and $\Leh_\mu$ is a ratio of positive real-analytic
			functions.
		\end{proof}
		\propmaster*
		
		\begin{proof}
			The first equality in identity \eqref{eq:master} holds because
			definition \eqref{eq:defmellin} gives
			\begin{equation}\label{eq:mellinshift}
				(\Mellin\mu)(s+1)=\int_{\Rpos}x^{s}\dd\mu=Z_\mu(s).
			\end{equation}
			For the second, the substitution $t=\log x$ turns
			definition \eqref{eq:defZ} into definition \eqref{eq:defmgf},
			\begin{equation}\label{eq:logpush}
				Z_\mu(s)=\int_{\R}e^{st}\dd\bigl((\log)_*\mu\bigr)(t)=M(s).
			\end{equation}
			For the third, definition \eqref{eq:defescort} of the escort
			measure gives
			\begin{equation}\label{eq:escortmean}
				\E_{\mu_{s-1}}[X]
				=\frac{1}{Z_\mu(s-1)}\int_{\Rpos}x\cdot x^{s-1}\dd\mu(x)
				=\frac{Z_\mu(s)}{Z_\mu(s-1)}=\Leh_\mu(s).
			\end{equation}
			The fourth equality is the definition of $K_\mu$. Since
			$K_\mu\in C^\infty(\mathring I_\mu)$ by
			Lemma~\ref{lem:logconvex}, the fundamental theorem of calculus
			gives identity \eqref{eq:logL}, and differentiating under the
			integral sign gives
			\begin{equation}\label{eq:Kprime}
				K_\mu'(u)=\frac{Z_\mu'(u)}{Z_\mu(u)}
				=\frac{1}{Z_\mu(u)}\int_{\Rpos} x^{u}\log x\,\dd\mu
				=\E_{\mu_u}[\log X]
			\end{equation}
			and
			\begin{equation}\label{eq:Kdoubleprime}
				K_\mu''(u)=\E_{\mu_u}[\log^2X]-K_\mu'(u)^2
				=\Var_{\mu_u}(\log X),
			\end{equation}
			which are the two relations in identity \eqref{eq:Kderivs}.
		\end{proof}
		
		\lemstrip*
		
		\begin{proof}
			Both partition functions are holomorphic on the strip
			$\{\Re s\in U\}$ by Lemma~\ref{lem:logconvex} and agree on $U$,
			hence agree on the whole strip by the identity theorem. Fix
			$\sigma\in U$. The function
			\begin{equation}\label{eq:cfnormalized}
				\tau\longmapsto\frac{Z_\mu(\sigma+i\tau)}{Z_\mu(\sigma)}
				=\E_{\mu_\sigma}\bigl[e^{i\tau\log X}\bigr]
			\end{equation}
			is the characteristic function of $\log X$ under $\mu_\sigma$, and
			likewise for $\nu$. Equality of characteristic functions gives
			$\mu_\sigma=\nu_\sigma$, hence
			$x^{\sigma}\dd\mu=x^{\sigma}\dd\nu$, hence $\mu=\nu$.
		\end{proof}

		\propmodulus*
		
		\begin{proof}
			Conjugate symmetry follows from $\overline{x^{\bar s}}=x^{s}$ for
			$x>0$ together with positivity of $\mu$. Writing
			$x^{a+ib}=x^{a}e^{ib\log x}$ and normalizing by $Z_\mu(a)$ gives
			\begin{equation}\label{eq:phia}
				\frac{Z_\mu(a+ib)}{Z_\mu(a)}
				=\E_{\mu_a}\bigl[e^{ib\log X}\bigr]\eqdef\phi_a(b),
			\end{equation}
			and $|\phi_a(b)|\le1$ is the triangle inequality, which is
			assertion (i). \emph{Assertion (ii).} $|\phi_a(b)|=1$ holds precisely
			when $e^{ib\log X}$ is $\mu_a$-almost surely equal to a fixed
			unimodular constant $e^{i\varphi}$, that is, when
			\begin{equation}\label{eq:latticecoset}
				b\log X\in\varphi+2\pi\Z\quad\mu_a\text{-a.s.},
				\qquad\text{equivalently}\qquad
				\log X\in\frac{\varphi}{b}+\frac{2\pi}{b}\Z .
			\end{equation}
			Since $\mu_a$ and $\mu$ are mutually absolutely continuous, the
			same holds $\mu$-almost surely, which is condition \eqref{eq:lattice} with
			$c=e^{\varphi/b}$. \emph{Assertion (iii).} Suppose
			condition \eqref{eq:latticecoset} holds for $b_1$ and for $b_2$. If $t,t'$
			both lie in the support of $(\log)_*\mu$ then
			\begin{equation}\label{eq:incommensurable}
				t-t'\in\frac{2\pi}{b_1}\Z\cap\frac{2\pi}{b_2}\Z=\{0\},
			\end{equation}
			the intersection being trivial because $b_1/b_2$ is irrational.
			Hence the support is a single point.
		\end{proof}
		\propendpoints*
		
		\begin{proof}
			\emph{Assertion (i).} The definition \eqref{eq:defS} gives
			\begin{equation}\label{eq:supS}
				\sup S_\mu=\min(\tau,\tau+1)=\tau .
			\end{equation}
			As $s\uparrow\tau$,
			\begin{equation}\label{eq:rightlimits}
				\begin{gathered}
					Z_\mu(s)\longrightarrow\infty,\\
					Z_\mu(s-1)\longrightarrow Z_\mu(\tau-1)\in(0,\infty),
				\end{gathered}
			\end{equation}
			the second limit by continuity on $\mathring I_\mu$, so the ratio in
			definition \eqref{eq:defL} diverges.
			
			\emph{Assertion (ii).} The same definition gives
			\begin{equation}\label{eq:infS}
				\inf S_\mu=\max(\iota,\iota+1)=\iota+1 .
			\end{equation}
			As $s\downarrow\iota+1$ one has $s-1\downarrow\iota$, so
			\begin{equation}\label{eq:leftlimits}
				\begin{gathered}
					Z_\mu(s-1)\longrightarrow\infty,\\
					Z_\mu(s)\longrightarrow Z_\mu(\iota+1)\in(0,\infty),
				\end{gathered}
			\end{equation}
			and the ratio tends to zero.
		\end{proof}
		
		\thmmono*
		
		\begin{proof}
			Let $h$ be an observable with the requisite integrability and
			write the escort mean as
			\begin{equation}\label{eq:escortobs}
				\E_{\mu_u}[h]=\frac{1}{Z_\mu(u)}\int_{\Rpos}h\,x^{u}\dd\mu .
			\end{equation}
			Differentiating expression \eqref{eq:escortobs} under the integral
			sign, which the domination bound \eqref{eq:domination} justifies,
			gives the covariance rule
			\begin{equation}\label{eq:covrule}
				\frac{\dd}{\dd u}\E_{\mu_u}[h]
				=\frac{\int h\,x^{u}\log x\,\dd\mu}{Z_\mu(u)}
				-\frac{\int h\,x^{u}\dd\mu}{Z_\mu(u)}
				\cdot\frac{Z_\mu'(u)}{Z_\mu(u)}
				=\Cov_{\mu_u}(h,\log X).
			\end{equation}
			Taking $h=X$ at $u=s-1$ and using
			$\Leh_\mu(s)=\E_{\mu_{s-1}}[X]$ from identity \eqref{eq:master}
			yields relation \eqref{eq:covid}. Nonnegativity holds because
			$X=e^{\log X}$ is a strictly increasing function of $\log X$, so
			the two variables are comonotone and their covariance is
			nonnegative by the Chebyshev association inequality; it vanishes
			if and only if $\log X$ is escort-almost surely constant, that is,
			if and only if $\mu$ is a point mass. For relation
			\eqref{eq:logderiv}, differentiate
			$\Lambda_\mu(s)=K_\mu(s)-K_\mu(s-1)$ and use
			$K_\mu''(u)=\Var_{\mu_u}(\log X)\ge0$ from identity \eqref{eq:Kderivs}.
		\end{proof}

		\propgruss*
		
		\begin{proof}
			\emph{Assertion (i).} Recall Popoviciu's inequality: a variable $Y$
			confined to an interval of length $\ell$ satisfies
			\begin{equation}\label{eq:popoviciu}
				\Var Y\le\E\Bigl(Y-\tfrac{\max+\min}{2}\Bigr)^2
				\le\Bigl(\frac{\ell}{2}\Bigr)^{2}.
			\end{equation}
			Under the escort $\mu_{s-1}$ the variables $X$ and $\log X$ are
			confined to intervals of lengths $M-m$ and $\log(M/m)$
			respectively, so the Cauchy--Schwarz inequality together with
			inequality \eqref{eq:popoviciu} gives
			\begin{equation}\label{eq:grussproof}
				\Leh_\mu'(s)=\Cov_{\mu_{s-1}}(X,\log X)
				\le\sqrt{\Var_{\mu_{s-1}}(X)\,\Var_{\mu_{s-1}}(\log X)}
				\le\tfrac14(M-m)\log\tfrac Mm ,
			\end{equation}
			and $\Var_{\mu_u}(\log X)\le\tfrac14\log^2(M/m)$, integrated over
			the unit window in relation \eqref{eq:logderiv}, gives the bound on
			$\Lambda_\mu'$. Nonnegativity is Theorem~\ref{thm:mono}. For
			assertion (ii), differentiation under the integral sign gives
			$K_\mu^{(k)}(u)=\kappa_k(u)$, so
			$\Lambda_\mu(s)=K_\mu(s)-K_\mu(s-1)$ yields identity
			\eqref{eq:Lamk}. Since $\Leh_\mu=e^{\Lambda_\mu}$, Fa\`a di
			Bruno's formula in its exponential form,
			\begin{equation}\label{eq:faadibruno}
				\frac{\dd^n}{\dd s^n}e^{\Lambda_\mu}
				=e^{\Lambda_\mu}\,
				B_n\bigl(\Lambda_\mu',\dots,\Lambda_\mu^{(n)}\bigr),
			\end{equation}
			gives identity \eqref{eq:Belln}; the low cases use
			$B_2(a_1,a_2)=a_2+a_1^2$ and $B_3=a_3+3a_1a_2+a_1^3$.
		\end{proof}
		\prophomogrefl*
		
		\begin{proof}
			\emph{Assertion (i).} A change of variables gives
			\begin{equation}\label{eq:dilation}
				\int_{\Rpos}(\lambda x)^{s}\dd\mu=\lambda^{s}Z_\mu(s),
			\end{equation}
			and in the ratio in definition \eqref{eq:defL} the dilation factors contribute
			$\lambda^{s}/\lambda^{s-1}=\lambda$. \emph{Assertion (ii).} The
			change of variables $y=1/x$ gives
			\begin{equation}\label{eq:inversionZ}
				Z_{\check\mu}(s)=Z_\mu(-s),
			\end{equation}
			hence
			\begin{equation}\label{eq:inversionL}
				\begin{gathered}
					\Leh_{\check\mu}(s)=\frac{Z_\mu(-s)}{Z_\mu(1-s)},\\
					\Leh_\mu(1-s)=\frac{Z_\mu(1-s)}{Z_\mu(-s)},
				\end{gathered}
			\end{equation}
			whose product is $1$. When $\check\mu=\mu$ this reads
			$\Lambda_\mu(s)+\Lambda_\mu(1-s)=0$, so $\Lambda_\mu$ is odd about
			$s=\tfrac12$; oddness forces $\Lambda_\mu(\tfrac12)=0$ and the
			vanishing of all even derivatives there.
		\end{proof}

		\thmstatgen*
		
		\begin{proof}
			\emph{Assertion (i).} Strict monotonicity and real-analyticity
			are Theorem~\ref{thm:mono} and Lemma~\ref{lem:logconvex}. For
			the endpoints, by identity \eqref{eq:logL} it suffices to show
			that $K_\mu'(u)=\E_{\mu_u}[\log X]\to\log M$ as $u\to\infty$.
			Fix $\varepsilon>0$ and put $A=\{x>M-\varepsilon\}$, so that
			$\mu(A)>0$. Then
			\begin{equation}\label{eq:statgenratio}
				\frac{\int_{A^c}x^{u}\dd\mu}{\int_Ax^{u}\dd\mu}
				\le\frac{\mu(\Rpos)\,(M-\varepsilon)^{u}}
				{\mu(\{x>M-\varepsilon/2\})\,(M-\varepsilon/2)^{u}}
				\xrightarrow[u\to\infty]{}0,
			\end{equation}
			so $\mu_u(A)\to1$ and consequently
			\begin{equation}\label{eq:statgenliminf}
				\liminf_{u\to\infty}\E_{\mu_u}[\log X]
				\ge\log(M-\varepsilon).
			\end{equation}
			Since also $\E_{\mu_u}[\log X]\le\log M$ and $\varepsilon$ was
			arbitrary, the limit is $\log M$. The case $s\to-\infty$
			follows from the reflection identity \eqref{eq:reflection}. A
			strictly increasing continuous map with limits $m$ and $M$ is a
			bijection onto $(m,M)$.
			
			\emph{Assertion (ii).} Normalize $M=1$. With an atom of mass
			$w_M>0$ at $1$ and $\mu((M',1))=0$ one has
			$Z_\mu(s)=w_M+O((M')^{s})$, whence
			$\Leh_\mu(s)=1-O((M')^{s-1})=1-O((M')^{s})$, the fixed
			factor $1/M'$ being absorbed into the implied constant; this
			is the rate \eqref{eq:georate}.
			
			\emph{Assertion (iii).} For two points with ratio
			$r=\min/\max\in(0,1]$,
			\begin{equation}\label{eq:twopointexact}
				M-\Leh_\mu(s)=M\,\frac{r^{\,s-1}(1-r)}{1+r^{\,s-1}}
				\le M\max_{r\in(0,1]}r^{\,s-1}(1-r)
				=\frac{M}{s}\Bigl(1-\frac1s\Bigr)^{s-1},
			\end{equation}
			the maximum being attained at $r=(s-1)/s$; and
			$(1-1/s)^{s-1}=e^{-1}(1+O(s^{-1}))$.
			
			\emph{Assertion (iv).} Since $\Leh_\mu'>0$ and $\Leh_\mu$ is
			real-analytic, the analytic inverse function theorem gives a
			real-analytic inverse whose Taylor coefficients about
			$t_0=\Leh_\mu(s_0)$ are the classical Lagrange--B\"urmann
			coefficients
			\begin{equation}\label{eq:burmann}
				\bigl[(t-t_0)^k\bigr]\ \Leh_\mu^{-1}(t)
				=\frac1{k!}\lim_{s\to s_0}\frac{\dd^{k-1}}{\dd s^{k-1}}
				\Bigl(\frac{s-s_0}{\Leh_\mu(s)-t_0}\Bigr)^{k},
			\end{equation}
			and at $s_0=1$ the first coefficient is
			\begin{equation}\label{eq:firstburmann}
				\frac{1}{\Leh_\mu'(1)}
				=\frac{1}{\Cov_{\mu_0}(X,\log X)}
			\end{equation}
			by identity \eqref{eq:slopeatone}, the covariance being taken
			under the normalized measure $\mu_0=\mu/Z_\mu(0)$.
		\end{proof}
		
		We turn to the divergence calculus. The two statements below are
		both computations with the escort densities; the third is the
		Donsker--Varadhan variational formula in the form the transform
		takes.
		\lembregman*
		
		\begin{proof}
			The escort densities give
			\begin{equation}\label{eq:escortratio}
				\frac{\dd\mu_a}{\dd\mu_b}(x)
				=x^{\,a-b}\,\frac{Z_\mu(b)}{Z_\mu(a)} ,
			\end{equation}
			so integrating the logarithm of this ratio against $\mu_a$ gives
			\begin{equation}\label{eq:bregmanproof}
				\KL(\mu_a\|\mu_b)
				=(a-b)\,\E_{\mu_a}[\log X]+K_\mu(b)-K_\mu(a)
				=K_\mu(b)-K_\mu(a)-(b-a)K_\mu'(a),
			\end{equation}
			where we used $\E_{\mu_a}[\log X]=K_\mu'(a)$ from
			identity \eqref{eq:Kderivs}; integrability of $\log X$ under every interior
			escort follows from the domination bound \eqref{eq:domination}.
		\end{proof}
		
		\thmsandwich*
		
		\begin{proof}
			Apply identity \eqref{eq:bregman} with $(a,b)=(s-1,s)$. Since
			$K_\mu(s)-K_\mu(s-1)=\Lambda_\mu(s)$, this gives
			\begin{equation}\label{eq:remainder1}
				\KL(\mu_{s-1}\|\mu_s)=\Lambda_\mu(s)-K_\mu'(s-1),
			\end{equation}
			which is identity \eqref{eq:sandwichlower} after substituting
			$K_\mu'(s-1)=\E_{\mu_{s-1}}[\log X]$. Applying identity
			\eqref{eq:bregman} with $(a,b)=(s,s-1)$ gives, in the same way,
			\begin{equation}\label{eq:remainder2}
				\KL(\mu_{s}\|\mu_{s-1})=K_\mu'(s)-\Lambda_\mu(s),
			\end{equation}
			which is identity \eqref{eq:sandwichupper}. Nonnegativity of the
			relative entropy yields the sandwich \eqref{eq:sandwich}, with
			equality if and only if $\mu_{s-1}=\mu_s$, that is, if and only if
			$X$ is $\mu$-almost surely constant.
		\end{proof}
		
		\coramgm*
		
		\begin{proof}
			We argue directly, so that no interiority hypothesis on the
			moment domain is needed. The size-biased law $\mu_1$ has density
			\begin{equation}\label{eq:sizebiased}
				\frac{\dd\mu_1}{\dd\mu}(x)=\frac{x}{A},
			\end{equation}
			which is positive $\mu$-almost everywhere, so the two laws are
			mutually absolutely continuous and
			\begin{equation}\label{eq:amgmproof}
				\KL(\mu\,\|\,\mu_1)
				=\int_{\Rpos}\log\frac{A}{x}\dd\mu(x)
				=\log A-\E_\mu[\log X]
				=\log\frac AG ,
			\end{equation}
			the integral being finite because $A<\infty$ and
			$\E_\mu|\log X|<\infty$. This is identity \eqref{eq:amgm}. The
			general statement \eqref{eq:amgmgeneral} is identity
			\eqref{eq:sandwichlower} rearranged, and holds for
			$s\in\mathring S_\mu$.
		\end{proof}
		
		\lemmidpoint*
		
		\begin{proof}
			By identity \eqref{eq:logL},
			\begin{equation}\label{eq:midint}
				\Lambda_\mu\bigl(s+\tfrac12\bigr)
				=\int_{s-1/2}^{s+1/2}K_\mu'(u)\dd u .
			\end{equation}
			The midpoint quadrature rule for a $C^2$ integrand $g$ on an
			interval of unit length has error $\tfrac1{24}g''(\xi)$; applying
			it with $g=K_\mu'$ gives identity \eqref{eq:midpoint}. Identity
			\eqref{eq:thirdcentral} follows by differentiating
			$K_\mu''(u)=\Var_{\mu_u}(\log X)$ under the integral sign.
		\end{proof}
		
		\thmjeffreys*
		
		\begin{proof}
			Add the two exact remainders \eqref{eq:remainder1} and
			identity \eqref{eq:remainder2}. The Lehmer terms cancel, leaving
			\begin{equation}\label{eq:jeffreysproof}
				J(\mu_{s-1},\mu_s)
				=K_\mu'(s)-K_\mu'(s-1)
				=\int_{s-1}^{s}K_\mu''(u)\dd u
				=\Lambda_\mu'(s),
			\end{equation}
			the last step being relation \eqref{eq:logderiv}. Finally
			$K_\mu''(u)=\Var_{\mu_u}(\log X)$ is the Fisher information of the
			escort family at natural parameter $u$, since the score of the
			family is $\log x-K_\mu'(u)$, whose variance is $K_\mu''(u)$.
		\end{proof}

		\proprenyidiv*
		
		\begin{proof}
			\emph{Assertion (i).} With
			$c=\lambda a+(1-\lambda)b\in\mathring I_\mu$,
			\begin{equation}\label{eq:renyiproof}
				\int_{\Rpos}\Bigl(\frac{\dd\mu_a}{\dd\mu}\Bigr)^{\lambda}
				\Bigl(\frac{\dd\mu_b}{\dd\mu}\Bigr)^{1-\lambda}\dd\mu
				=\frac{Z_\mu(c)}{Z_\mu(a)^{\lambda}Z_\mu(b)^{1-\lambda}}
				=\exp\bigl[K_\mu(c)-\lambda K_\mu(a)
				-(1-\lambda)K_\mu(b)\bigr],
			\end{equation}
			and $D_\lambda=\frac{1}{\lambda-1}\log(\cdot)$ of the left-hand
			side of identity \eqref{eq:renyiproof}, which is identity
			\eqref{eq:renyidiv}.
			
			\emph{Assertion (ii).} The Chernoff information is
			$\sup_{\lambda\in[0,1]}(1-\lambda)D_\lambda(\mu_s\|\mu_{s-1})$,
			which by identity \eqref{eq:renyidiv} with $(a,b)=(s,s-1)$ and
			$c=s-1+\lambda$ is the displayed expression
			\eqref{eq:chernoff}.
			
			\emph{Assertion (iii).} Let $q\ll\mu_{s-1}$ with
			$\E_q|\log X|<\infty$ and put $r=\dd q/\dd\mu_{s-1}$. Then
			\begin{equation}\label{eq:dvproof}
				\E_q[\log X]-\KL(q\|\mu_{s-1})
				=\int(\log x-\log r)\,r\dd\mu_{s-1}
				=\Lambda_\mu(s)-\int r\log\frac{r}{x/\Leh_\mu(s)}
				\dd\mu_{s-1}.
			\end{equation}
			Since $\dd\mu_s/\dd\mu_{s-1}=x/\Leh_\mu(s)$ by identity
			\eqref{eq:master}, the last integral in identity \eqref{eq:dvproof} is
			$\KL(q\|\mu_s)\ge0$, vanishing if and only if $q=\mu_s$; and if
			$q\not\ll\mu_{s-1}$ the bracket is $-\infty$.
		\end{proof}
		
		The multiplicative theory comes next. The L\'evy--Khintchine
		representation transfers to the transform by subtracting the
		cumulant function at two orders one unit apart.
		\thmcharacter*
		
		\begin{proof}
			The integrand $(x,y)\mapsto(xy)^{s}$ is nonnegative and the
			character is multiplicative, so Tonelli's theorem gives
			\begin{equation}\label{eq:tonelli}
				Z_{\mu\boxtimes\nu}(s)
				=\int_{\Rpos}\!\!\int_{\Rpos}(xy)^{s}\dd\mu(x)\dd\nu(y)
				=Z_\mu(s)\,Z_\nu(s),
			\end{equation}
			with both sides finite or infinite together. Taking the ratio of
			identity \eqref{eq:tonelli} at $s$ and at $s-1$ gives identity
			\eqref{eq:homo}.
		\end{proof}
		
		\thmpowers*
		
		\begin{proof}
			Since $(X^{r})^{a}=X^{ra}$ we get $Z_{X^r}(a)=Z_X(ra)$, hence
			\begin{equation}\label{eq:powerZ}
				\Leh_{X^r}(s)=\frac{Z_X(rs)}{Z_X(rs-r)} .
			\end{equation}
			Taking logarithms of identity \eqref{eq:powerZ} gives the first equality in
			identity \eqref{eq:powerlog}, and the substitution $u=rv$ gives the
			integral form. For integer $r$, the cocycle identity
			\eqref{eq:cocycle} applied over the window of length $r$ ending at
			$rs$ factors the same ratio into unit steps, which is identity
			\eqref{eq:power}.
		\end{proof}

		\thmlevy*
		
		\begin{proof}
			On the interval of finiteness the cumulant function of $\log X$
			has the L\'evy--Khintchine form
			\begin{equation}\label{eq:lkform}
				K_\mu(u)=bu+\tfrac12\sigma^2u^2
				+\int_\R\bigl(e^{uy}-1-uy\mathbf 1_{|y|\le1}\bigr)
				\Pi(\dd y).
			\end{equation}
			Subtracting representation \eqref{eq:lkform} at $u=s-1$ from its value at
			$u=s$, the Gaussian part contributes
			\begin{equation}\label{eq:lkgauss}
				\tfrac12\sigma^2\bigl(s^2-(s-1)^2\bigr)
				=\sigma^2\Bigl(s-\tfrac12\Bigr),
			\end{equation}
			the drift contributes $b$, and the jump part contributes
			\begin{equation}\label{eq:lkjump}
				\int_\R\bigl[e^{sy}-e^{(s-1)y}
				-y\mathbf 1_{|y|\le1}\bigr]\Pi(\dd y),
				\qquad
				e^{sy}-e^{(s-1)y}=e^{(s-1)y}\bigl(e^y-1\bigr),
			\end{equation}
			which together give identity \eqref{eq:levyL}; the integrand is
			$O(y^2)$ at the origin and integrable at infinity by the
			exponential-moment hypothesis. Differentiating under the
			integral sign, with domination as in
			bound \eqref{eq:domination}, gives identity \eqref{eq:levyderiv},
			whose integrand is nonnegative because $y$ and $e^y-1$ share
			the same sign. Finally, the $k$th multiplicative root
			$X^{[k]}$ is by definition the variable whose logarithm is the
			$k$th convolution root of $\log X$, which exists and is unique
			in law precisely because $\log X$ is infinitely divisible; its
			cumulant function is $K_\mu/k$, whose unit increment is
			$\Lambda_X/k$, so that
			\begin{equation}\label{eq:multroot}
				\Leh_{X^{[k]}}=\Leh_X^{1/k} .
			\end{equation}
		\end{proof}
		
		\thmrenorm*
		
		\begin{proof}
			\emph{Assertion (i).} Theorem~\ref{thm:homo} gives
			$Z_{\prod_iX_i}(a)=Z_X(a)^k$, and Theorem~\ref{thm:power} with
			$r=1/\sqrt k$ gives
			\begin{equation}\label{eq:renormproof}
				Z_{\mathcal R_kX}(s)=Z_{\prod_iX_i}\bigl(s/\sqrt k\bigr)
				=Z_X\bigl(s/\sqrt k\bigr)^{k},
			\end{equation}
			which is identity \eqref{eq:renormK} after taking logarithms. For
			assertion (ii), the fixed-point condition reads
			\begin{equation}\label{eq:fixedpointcond}
				k\,K_X\bigl(u/\sqrt k\bigr)=K_X(u)
			\end{equation}
			for all $u$ in a neighbourhood of the origin, for one fixed
			integer $k\ge2$; no more than a single $k$ is needed.
			Define $g(u)=K_X(u)/u^{2}$ for $u\neq0$; condition
			\eqref{eq:fixedpointcond} becomes
			\begin{equation}\label{eq:gfixed}
				g\bigl(u/\sqrt k\bigr)=g(u) .
			\end{equation}
			Since $K_X(0)=0$ and $K_X'(0)=\E[\log X]=0$, Taylor's theorem gives
			\begin{equation}\label{eq:gtaylor}
				\lim_{u\to0}g(u)=\frac{K_X''(0)}{2}=\frac{\sigma^2}{2} .
			\end{equation}
			Iterating relation \eqref{eq:gfixed} gives $g(u)=g(u/k^{n/2})$
			for every $n\ge1$, and letting $n\to\infty$ with limit
			\eqref{eq:gtaylor} therefore yields $g\equiv\sigma^2/2$, that
			is,
			\begin{equation}\label{eq:Kquadratic}
				K_X(u)=\frac{\sigma^2u^2}{2}
			\end{equation}
			on the neighbourhood. By Lemma~\ref{lem:strip} the law of $X$ is
			lognormal with parameters $(0,\sigma^2)$, and then
			\begin{equation}\label{eq:logaffineproof}
				\Lambda_X(s)=K_X(s)-K_X(s-1)
				=\frac{\sigma^2}{2}\bigl(s^2-(s-1)^2\bigr)
				=\sigma^2\Bigl(s-\tfrac12\Bigr),
			\end{equation}
			which is identity \eqref{eq:logaffine}. Conversely, a lognormal law
			with parameters $(0,\sigma^2)$ satisfies identity
			\eqref{eq:Kquadratic}, and identity \eqref{eq:renormK} then gives
			\begin{equation}\label{eq:lnfixed}
				K_{\mathcal R_kX}(s)=k\,\frac{\sigma^2s^2}{2k}
				=\frac{\sigma^2s^2}{2}=K_X(s)
			\end{equation}
			for every $k$, so a lognormal law is a fixed point of every
			$\mathcal R_k$; in particular a fixed point for one $k\ge2$ is
			a fixed point for all. That log-affinity of $\Leh_X$ does
			\emph{not} conversely force lognormality is a consequence of a
			gauge freedom: multiplying $Z_X$ by a positive $1$-periodic
			real-analytic factor $G$ with
			\begin{equation}\label{eq:gaugenormalized}
				\begin{gathered}
					G(0)=1,\\
					G'(0)=G''(0)=0
				\end{gathered}
			\end{equation}
			leaves $\Leh_X$, the total mass, $\E[\log X]$ and
			$\Var(\log X)$ all unchanged, while the law itself changes.
			Example~\ref{ex:logaffinecompanion} exhibits such a factor
			explicitly and verifies that the perturbed density is
			nonnegative.
			
			\emph{Assertion (iii).} Taylor's theorem
			with the same two vanishing derivatives gives
			\begin{equation}\label{eq:renormlimit}
				K_{\mathcal R_nX}(s)=n\,K_X\bigl(s/\sqrt n\bigr)
				=n\Bigl[\frac{\sigma^2s^2}{2n}+O\bigl(n^{-3/2}\bigr)\Bigr]
				\longrightarrow\frac{\sigma^2s^2}{2},
			\end{equation}
			locally uniformly, and subtracting the same expression at $s-1$
			gives the limit \eqref{eq:multclt}.
		\end{proof}
		
		We come to the rigidity theory. The first result exhibits the
		lognormal fiber by an explicit Gaussian integral; the two that
		follow establish the structure of the period group and the general
		resonance construction of which the lognormal fiber is one
		instance.
		\propdeconv*
		
		\begin{proof}
			Identity \eqref{eq:deconv} is identity \eqref{eq:homo} rearranged.
			For the stability statement, division of the two representations
			in equation \eqref{eq:relerr} gives
			\begin{equation}\label{eq:stabproof}
				\varepsilon_X
				=\frac{1+\varepsilon_Y}{1+\varepsilon_N}-1
				=\frac{\varepsilon_Y-\varepsilon_N}{1+\varepsilon_N},
			\end{equation}
			whence, using $|\varepsilon_Y-\varepsilon_N|\le2\varepsilon$ and
			$|1+\varepsilon_N|\ge1-\varepsilon>\tfrac12$,
			\begin{equation}\label{eq:stabbound}
				|\varepsilon_X|
				\le\frac{2\varepsilon}{1-\varepsilon} .
			\end{equation}
			Equality holds in bound \eqref{eq:stabbound} at
			$\varepsilon_Y=\varepsilon$ and $\varepsilon_N=-\varepsilon$.
		\end{proof}
		
		\lemfiber*
		
		\begin{proof}
			Suppose first that $Z_{\mu'}=e^{\varpi}Z_\mu$ on $S\cup(S-1)$
			with $\varpi(s)=\varpi(s-1)$. For $s\in S$ both $s$ and $s-1$ lie
			in that union, and the gauge cancels in the ratio at consecutive
			orders, so $\Leh_{\mu'}=\Leh_\mu$. Conversely, suppose the
			transforms agree on $S$ and set
			\begin{equation}\label{eq:pidef}
				\varpi\eqdef\log Z_{\mu'}-\log Z_\mu .
			\end{equation}
			Then $\varpi(s)-\varpi(s-1)=0$ on $S$, and $\varpi$ is
			real-analytic on $S$, which is open, by
			Lemma~\ref{lem:logconvex}. Since $S$ is an open
			interval of length greater than one, the relation
			$\varpi(s)=\varpi(s-1)$ extends $\varpi$ to a $1$-periodic real-analytic
			function, and equation \eqref{eq:gaugeeq} follows.
		\end{proof}
		
		\thmrigidity*
		
		\begin{proof}
			Let $\varphi_i=\log Z_i$, which is convex, and let
			$p=\varphi_1-\varphi_2$, which is $1$-periodic by
			Lemma~\ref{lem:fiber}. Convexity traps the one-sided derivatives
			between consecutive chord slopes: for every $s>s_0+1$ and
			$i=1,2$,
			\begin{equation}\label{eq:chords}
				\log g(s)=\varphi_i(s)-\varphi_i(s-1)
				\le\varphi_i'(s^{\pm})
				\le\varphi_i(s+1)-\varphi_i(s)=\log g(s+1).
			\end{equation}
			Here $g$ is automatically nondecreasing: for convex $\varphi_i$
			the unit increment $s\mapsto\varphi_i(s)-\varphi_i(s-1)$ is
			nondecreasing, being an average of the nondecreasing derivative
			over a sliding unit window, and that increment is $\log g$. Hence
			the limit superior in definition \eqref{eq:etadef} is nonnegative.
			At every common point of differentiability, that is, at all but
			countably many $s$, the chain \eqref{eq:chords} gives
			\begin{equation}\label{eq:pprimebound}
				|p'(s)|\le\log\frac{g(s+1)}{g(s)} .
			\end{equation}
			Fix such an $s$. By periodicity $p'(s)=p'(s+n)$ for every integer
			$n\ge0$, so applying the bound \eqref{eq:pprimebound} at $s+n$ and
			letting $n\to\infty$ gives
			\begin{equation}\label{eq:pprimeeta}
				|p'(s)|\le\liminf_{n\to\infty}\log\frac{g(s+n+1)}{g(s+n)}
				\le\eta .
			\end{equation}
			\emph{Assertion (i).} If $\eta=0$ then $p'=0$ off a countable set;
			since $p$ is a difference of convex functions it is locally
			Lipschitz, hence constant.
			
			\emph{Assertion (ii).} Suppose $\eta<\infty$ and fix a unit
			interval $J=[a,a+1]\subset(s_0+1,\infty)$, on which $p$ is
			defined. Being a difference of convex functions, $p$ is locally
			Lipschitz on $(s_0+1,\infty)$, hence absolutely continuous, so its
			total variation over $J$ is $\int_J|p'|$. Let
			$s_{\min}$ and $s_{\max}$ be points of $J$ at which the
			continuous periodic function $p$ attains its extrema. The two arcs
			of the period circle joining them each carry total variation at
			least $\osc(p)=\max p-\min p$, so
			\begin{equation}\label{eq:oscbound}
				2\,\osc(p)\;\le\;\int_J|p'(u)|\dd u\;\le\;\eta,
				\qquad\text{that is}\qquad
				\osc(p)\le\frac{\eta}{2} .
			\end{equation}
			Centring $p$ by the choice
			\begin{equation}\label{eq:centring}
				c=\exp\Bigl(\tfrac12\bigl(\max p+\min p\bigr)\Bigr)
			\end{equation}
			then gives
			\begin{equation}\label{eq:centredbound}
				\Bigl|\log\frac{Z_1(s)}{c\,Z_2(s)}\Bigr|
				\le\frac{\osc(p)}{2}\le\frac{\eta}{4},
			\end{equation}
			which is the two-sided bound \eqref{eq:quantgauge}.
			
			\emph{Consequences.} Apply
			the theorem with $g=\Leh_\mu=\Leh_\nu$, both partition functions
			being log-convex solutions of equation \eqref{eq:funceq} by
			Lemma~\ref{lem:logconvex}; in case (i) one concludes
			$Z_\nu=cZ_\mu$, and Lemma~\ref{lem:strip} identifies $\nu=c\mu$.
		\end{proof}
		
		The constant in the two-sided bound \eqref{eq:quantgauge} is valid but
		not optimal, and it is worth recording how far it is from optimal,
		since the gap measures how much of the residual gauge the convexity
		argument fails to see.
		
		\begin{remark}\label{rem:notsharp}
			Take
			\begin{equation}\label{eq:extremalpair}
				\begin{gathered}
					\varphi_2'(s)=\lfloor s\rfloor,\\
					\varphi_1'(s)=\Bigl\lfloor s+\tfrac12\Bigr\rfloor-\tfrac12 ,
				\end{gathered}
			\end{equation}
			and let $\varphi_1,\varphi_2$ be their primitives. Both are
			convex, and both satisfy the same difference equation, since
			\begin{equation}\label{eq:extremalL}
				\begin{gathered}
					\varphi_i(s)-\varphi_i(s-1)=L(s)\eqdef s-1\\
					(i=1,2),
				\end{gathered}
			\end{equation}
			so that $Z_i=e^{\varphi_i}$ are positive log-convex solutions of
			equation \eqref{eq:funceq} with $g=e^{L}$. Here
			\begin{equation}\label{eq:extremaleta}
				L(s+1)-L(s)\equiv1,
				\qquad\text{so that}\qquad
				\eta=1,
			\end{equation}
			while
			$p'=\varphi_1'-\varphi_2'$ is the square wave taking the values
			$\mp\tfrac12$ on the two halves of each period, whence
			\begin{equation}\label{eq:extremalosc}
				\begin{gathered}
					\osc(p)=\tfrac14=\frac{\eta}{4},\\
					\Bigl|\log\frac{Z_1}{cZ_2}\Bigr|\le\frac{\eta}{8} .
				\end{gathered}
			\end{equation}
			The pair \eqref{eq:extremalpair} therefore realizes exactly half of
			the oscillation that the bound \eqref{eq:oscbound} permits. The
			obstruction to attaining the full bound is visible in the
			construction: a jump of $p'$ by an amount $J$ forces one of the two
			convex functions to jump by $J$ within the window
			$[L(s),L(s+1)]$, so $J\le\eta$ and the square wave cannot have
			amplitude $\eta$. Determining the optimal constant is an open
			problem; the pair \eqref{eq:extremalpair} shows that it lies
			between $\eta/8$ and $\eta/4$.
		\end{remark}

		\propinfoform*
		
		\begin{proof}
			Condition \eqref{eq:ratiocond} says
			$\Lambda_\mu(s+1)-\Lambda_\mu(s)\to0$. Writing
			\begin{equation}\label{eq:lamdiff}
				\Lambda_\mu(s)=K_\mu(s)-K_\mu(s-1)
			\end{equation}
			gives
			\begin{equation}\label{eq:diffidentity}
				\Lambda_\mu(s+1)-\Lambda_\mu(s)
				=\bigl[K_\mu(s+1)-K_\mu(s)\bigr]
				-\bigl[K_\mu(s)-K_\mu(s-1)\bigr],
			\end{equation}
			which is the second difference in limit \eqref{eq:seconddiff}, so (i)
			and (ii) are equivalent. By the fundamental theorem of calculus
			and Theorem~\ref{thm:jeffreys},
			\begin{equation}\label{eq:cesaro}
				\Lambda_\mu(s+1)-\Lambda_\mu(s)
				=\int_{s}^{s+1}\Lambda_\mu'(u)\dd u
				=\int_{s}^{s+1}J\bigl(\mu_{u-1},\mu_u\bigr)\dd u ,
			\end{equation}
			and the integrand is nonnegative, so (i) and (iii) are
			equivalent.
		\end{proof}

		\thmheyde*
		
		\begin{proof}
			Positivity of the density in definition \eqref{eq:heydedef} is clear for
			$|\varepsilon|\le1$. Let $\phi_{m,\sigma^2}$ denote the Gaussian
			density of $\log X$. Substituting $t=m+\sigma u$ and completing
			the square,
			\begin{equation}\label{eq:heydeint}
				\int_\R e^{st}\phi_{m,\sigma^2}(t)
				\sin\Bigl(\frac{2\pi(t-m)}{\sigma^{2}}\Bigr)\dd t
				=e^{ms}\,\Im\,
				\exp\Bigl\{\tfrac12\bigl(\sigma s
				+i\tfrac{2\pi}{\sigma}\bigr)^{2}\Bigr\} .
			\end{equation}
			Expanding the exponent,
			\begin{equation}\label{eq:heydeexpand}
				\tfrac12\Bigl(\sigma s+i\frac{2\pi}{\sigma}\Bigr)^2
				=\frac{\sigma^2s^2}{2}-\frac{2\pi^2}{\sigma^2}+2\pi is ,
			\end{equation}
			so that the right-hand side of identity \eqref{eq:heydeint} equals
			\begin{equation}\label{eq:heydevalue}
				e^{ms+\sigma^2s^2/2}\,e^{-2\pi^2/\sigma^2}\sin(2\pi s).
			\end{equation}
			Adding the unmodulated part gives identity \eqref{eq:heydeZ}, and
			evaluating at $s=0$ confirms unit mass. The factor multiplying
			$e^{ms+\sigma^2s^2/2}$ in identity \eqref{eq:heydeZ} is a strictly positive
			$1$-periodic gauge, since
			$\varepsilon e^{-2\pi^2/\sigma^2}<1$, so it cancels in the ratio
			at consecutive orders and $\Leh_{\mu^{(\varepsilon)}}=\Leh_\mu$.
			Distinctness across $\varepsilon$ is visible at $s=\tfrac14$.
		\end{proof}
		
		\propPistructure*
		
		\begin{proof}
			That $\Theta_\mu$ is a subgroup is immediate from the definition. For
			closedness, let $\omega_n\in\Theta_\mu$ with $\omega_n\to\omega$, and
			let $s$ be a point of the strip at which neither $s$ nor $s+i\omega$
			is a pole of $\Leh_\mu$. Continuity of $\Leh_\mu$ at $s+i\omega$
			gives
			\begin{equation}\label{eq:closedproof}
				\Leh_\mu(s+i\omega)=\lim_{n\to\infty}\Leh_\mu(s+i\omega_n)
				=\Leh_\mu(s),
			\end{equation}
			so the meromorphic functions $\Leh_\mu(\cdot+i\omega)$ and
			$\Leh_\mu$ agree off a discrete set and hence identically, that
			is, $\omega\in\Theta_\mu$. Suppose now $\Theta_\mu=\R$ and fix a real
			$a\in\mathring S_\mu$. Then $\Leh_\mu(a+i\omega)=\Leh_\mu(a)$ for
			every $\omega\in\R$, so the meromorphic function
			$\Leh_\mu-\Leh_\mu(a)$ vanishes on the entire vertical line
			$a+i\R$, a set with accumulation points in the strip; by the
			identity theorem $\Leh_\mu$ is constant. This contradicts
			Theorem~\ref{thm:mono}, which gives $\Leh_\mu'>0$ on
			$\mathring S_\mu\cap\R$ for measures not concentrated at a point.
			The classification of closed proper subgroups of $\R$ completes
			the proof.
		\end{proof}
		
		\thmresonant*
		
		\begin{proof}
			\emph{Assertion (i).} Put
			\begin{equation}\label{eq:comegadef}
				c_\omega(s)\eqdef\frac{Z_\mu(s+i\omega)}{Z_\mu(s)} .
			\end{equation}
			The condition $\omega\in\Theta_\mu$ says
			\begin{equation}\label{eq:comegaperiodic}
				\frac{Z_\mu(s+i\omega)}{Z_\mu(s-1+i\omega)}
				=\frac{Z_\mu(s)}{Z_\mu(s-1)}
				\iff
				c_\omega(s)=c_\omega(s-1),
			\end{equation}
			so $c_\omega$ is $1$-periodic. For real $s$,
			\begin{equation}\label{eq:modulatedZ}
				\int_{\Rpos}x^{s}\cos(\omega\log x+\phi)\dd\mu
				=\Re\Bigl[e^{i\phi}\int_{\Rpos}x^{s+i\omega}\dd\mu\Bigr]
				=\Re\bigl[e^{i\phi}Z_\mu(s+i\omega)\bigr],
			\end{equation}
			whence
			\begin{equation}\label{eq:modulatedZ2}
				Z_{\mu_\varepsilon}(s)
				=Z_\mu(s)\Bigl(1+\varepsilon
				\Re\bigl[e^{i\phi}c_\omega(s)\bigr]\Bigr).
			\end{equation}
			Because $\mu$ is not carried by a geometric progression of
			ratio $e^{2\pi/|\omega|}$, Proposition~\ref{prop:modulus}(ii)
			gives the strict bound $|c_\omega(s)|<1$ for every real
			$s\in\mathring I_\mu$, so the factor in identity \eqref{eq:modulatedZ2}
			is strictly positive; it is $1$-periodic because $c_\omega$
			is. Lemma~\ref{lem:fiber} therefore gives
			$\Leh_{\mu_\varepsilon}=\Leh_\mu$. Since
			$\mu_\varepsilon\le2\mu$, each $\mu_\varepsilon$ is
			admissible with $I_{\mu_\varepsilon}\supseteq I_\mu$.
			
			Nothing in the argument so far used the value of $\phi$: the
			factor in identity \eqref{eq:modulatedZ2} is $1$-periodic and,
			by the strict bound $|c_\omega|<1$, positive for every real
			$\phi$. What the choice of phase must secure is only that the
			modulation be non-degenerate. Suppose both
			$\cos(\omega\log X)$ and $\cos(\omega\log X+\pi/2)
			=-\sin(\omega\log X)$ were $\mu$-almost surely constant. Then
			$e^{i\omega\log X}$ would be $\mu$-almost surely constant,
			which by the computation in identity \eqref{eq:latticecoset} means that
			$\mu$ is carried by a geometric progression of ratio
			$e^{2\pi/|\omega|}$, contrary to hypothesis. Hence at least one
			$\phi\in\{0,\pi/2\}$ makes $\cos(\omega\log X+\phi)$
			non-constant, which is all that is asserted; any other $\phi$
			with the same property serves equally well, and part (ii) uses
			one such. For such a $\phi$ the modulating cosine is in
			particular not $\mu$-almost surely zero, and since
			\begin{equation}\label{eq:distinctness}
				\frac{\dd\mu_\varepsilon}{\dd\mu}
				-\frac{\dd\mu_{\varepsilon'}}{\dd\mu}
				=(\varepsilon-\varepsilon')\cos(\omega\log x+\phi),
			\end{equation}
			the measures are pairwise distinct.
			
			\emph{Assertion (ii).} For the lognormal,
			\begin{equation}\label{eq:lognormalperiod}
				\Leh_\mu(s+i\omega)
				=e^{m+\sigma^2(s+i\omega-1/2)}
				=\Leh_\mu(s)\,e^{i\sigma^2\omega},
			\end{equation}
			so $\omega\in\Theta_\mu$ if and only if
			$\sigma^2\omega\in2\pi\Z$, giving
			$\Theta_\mu=\tfrac{2\pi}{\sigma^2}\Z$. At the fundamental period
			$\omega=2\pi/\sigma^2$ with the phase
			$\phi=-2\pi m/\sigma^2-\pi/2$, the modulation
			definition \eqref{eq:modulation} is exactly the one in
			definition \eqref{eq:heydedef}, since
			$\cos(\theta-\pi/2)=\sin\theta$.
			
			\emph{Assertion (iii).} Write $F$ for the meromorphic
			continuation of $\Leh_\mu$, by hypothesis a nonconstant
			product of a rational function $R$ and Gamma factors
			$\Gamma(a_i+s/p_i)^{\varepsilon_i}$ with
			$\varepsilon_i\in\{\pm1\}$, and let
			$\omega\in\Theta_\mu\setminus\{0\}$, so that
			$F(s+i\omega)\equiv F(s)$ by the identity theorem.
			
			First, $F$ has no zeros and no poles. The Gamma function has
			no zeros and its poles lie at the nonpositive integers, so
			the zeros and poles of every factor
			$\Gamma(a_i+s/p_i)^{\pm1}$ lie on the real axis, and the
			zeros and poles of $F$ off the real axis, which can come
			only from $R$, are finite in number. If $F$ had a zero or a
			pole at some point $s_0$, the periodicity would propagate
			it, with its order, along the infinite orbit
			$\{s_0+ik\omega:k\in\Z\}$, every point of which except
			$s_0$ lies off the real axis: a contradiction.
			
			Next, $F=e^{\alpha+\beta s}$ for constants
			$\alpha,\beta$. Writing $R=p/q$ with polynomials $p$ and
			$q$, the functions
			\begin{equation}\label{eq:PQentire}
				\begin{gathered}
					P=p\prod_{i:\,\varepsilon_i=-1}
					\frac1{\Gamma(a_i+s/p_i)},\\
					Q=q\prod_{i:\,\varepsilon_i=+1}
					\frac1{\Gamma(a_i+s/p_i)}
				\end{gathered}
			\end{equation}
			are entire of order one, being polynomials times reciprocal
			Gamma factors, and $F=P/Q$. Since $F$ is entire and
			zero-free, the zeros of $P$ and of $Q$ coincide with
			multiplicity, so the canonical products in their Hadamard
			factorizations are identical and cancel in the
			quotient~\cite{boas1954}, leaving $F=e^{\alpha+\beta s}$.
			On the real interval $\mathring S_\mu$ the function
			$F=\Leh_\mu$ is positive and nondecreasing, so $\alpha$
			and $\beta$ are real with $\beta\ge0$, and $\beta>0$
			since $F$ is nonconstant. Thus
			$\Leh_\mu(s)=c\,b^{\,s}$ with $c=e^{\alpha}>0$ and
			$b=e^{\beta}>1$, and
			$\Leh_\mu(s+i\omega)\equiv\Leh_\mu(s)$ holds exactly
			when $\omega\log b\in2\pi\Z$, so that
			$\Theta_\mu=\tfrac{2\pi}{\log b}\,\Z$. If $\Leh_\mu$
			is not of this exponential form, no nonzero period can
			exist, and $\Theta_\mu=\{0\}$.
			
			Both branches of the dichotomy occur. A nonconstant rational
			transform has a zero or a pole in $\C$, so its period group
			is trivial; the log-affine transform of the lognormal is the
			exponential case, with
			$\Theta_\mu=\tfrac{2\pi}{\sigma^2}\Z$ in agreement
			with assertion (b); and exponentials do arise from Gamma
			products, as the Legendre duplication formula
			$\Gamma(2s)/\bigl(\Gamma(s)\Gamma(s+\tfrac12)\bigr)
			=4^{s}/(2\sqrt\pi)$ shows.
		\end{proof}
		
		The characterization results follow. We first verify the
		dictionary row by row, then prove the Abelian tail law and the
		two characterization theorems, and close with the exact recovery
		theorem for finitely supported measures.
		\proptwosources*
		
		\begin{proof}[Proof of Proposition~\ref{prop:twosources}]
			The two cases are complementary by definition, and the modulus
			statements are Proposition~\ref{prop:modulus}(ii) with
			$b=\omega_0$ and $2\pi/|b|=\log R$. For the converse in case (i),
			suppose $\supp\mu\subset\{cR^{k}\}$. Then for $x=cR^k$,
			\begin{equation}\label{eq:latticechar}
				x^{i\omega_0}=e^{i\omega_0\log c}\,e^{i\omega_0k\log R}
				=e^{i\omega_0\log c}\,e^{2\pi ik}
				=e^{i\omega_0\log c},
			\end{equation}
			a constant unimodular factor, so
			$Z_\mu(s+i\omega_0)=e^{i\omega_0\log c}Z_\mu(s)$; the constant
			cancels in the ratio at consecutive orders and hence
			$\omega_0\in\Theta_\mu$. The same computation shows
			$\cos(\omega_0\log x+\phi)$ is constant on the progression, so the
			modulation definition \eqref{eq:modulation} multiplies $\mu$ by a constant.
			Case (ii) is Theorem~\ref{thm:resonant}(i).
		\end{proof}
		
		\proplatticecompanion*
		
		\begin{proof}
			Completing the square in the exponent gives
			\begin{equation}\label{eq:latticeZ}
				Z_\mu(s)=\sum_{k\in\Z}e^{-\beta k^{2}/2+\beta ks}
				=e^{\beta s^{2}/2}\sum_{k\in\Z}e^{-\beta(k-s)^{2}/2},
			\end{equation}
			and the theta series on the right of identity \eqref{eq:latticeZ}
			is $1$-periodic in $s$, being a sum over the shifted lattice
			$\Z-s$. The gauge therefore cancels in the ratio at consecutive
			orders, which gives the curve \eqref{eq:latticeLeh}; identity
			\eqref{eq:heydeZ} at $\varepsilon=0$, $m=0$ and
			$\sigma^2=\beta$ shows this is the lognormal curve. The
			measure is carried by the geometric progression
			$\{e^{\beta k}\}$, so it is of lattice type by
			Proposition~\ref{prop:twosources}(i).
		\end{proof}
		
		\corfrontier*
		
		\begin{proof}
			Suppose condition \eqref{eq:ratiocond} holds. Choose
			$\omega\in\Theta_\mu\setminus\{0\}$ and a phase $\phi$ as in
			Theorem~\ref{thm:resonant}(i), and let
			$\{\mu_\varepsilon\}_{\varepsilon\in[-1,1]}$ be the resulting
			family. Each $\mu_\varepsilon$ satisfies
			$\mu_\varepsilon\le2\mu$, so it is a finite measure; let
			$\widetilde\mu_\varepsilon
			=\mu_\varepsilon/\mu_\varepsilon(\Rpos)$ be its normalization,
			which is a probability measure with
			$\Leh_{\widetilde\mu_\varepsilon}=\Leh_\mu$ on $(s_0,\infty)$.
			Both $Z_{\widetilde\mu_\varepsilon}$ and $Z_\mu$ are positive
			log-convex solutions of equation \eqref{eq:funceq} with
			$g=\Leh_\mu$ there, so Theorem~\ref{thm:rigidity}(i) gives
			$Z_{\widetilde\mu_\varepsilon}=cZ_\mu=Z_{c\mu}$ on $(s_0,\infty)$
			for some $c>0$. Lemma~\ref{lem:strip} then gives
			$\widetilde\mu_\varepsilon=c\mu$, and comparing total masses, both
			being $1$, forces $c=1$ and hence
			$\widetilde\mu_\varepsilon=\mu$. Therefore
			\begin{equation}\label{eq:frontiercontr}
				1+\varepsilon\cos(\omega\log x+\phi)
				=\mu_\varepsilon(\Rpos)\qquad\mu\text{-a.s.}
			\end{equation}
			for every $\varepsilon\in[-1,1]$. Taking any $\varepsilon\neq0$ in
			identity \eqref{eq:frontiercontr} shows that $\cos(\omega\log X+\phi)$ is
			$\mu$-almost surely constant, contradicting the choice of $\phi$.
		\end{proof}

		\propdictionary*
		
		\begin{proof}
			Each entry is $m_s/m_{s-1}$ for the stated moment function
			$m_a=\E[X^a]$, and each period group follows from the criteria
			of Theorem~\ref{thm:resonant}(ii),(iii) together with
			Proposition~\ref{prop:twosources}. For the log-logistic and
			half-Cauchy rows, Euler's reflection formula
			$\Gamma(z)\Gamma(1-z)=\pi/\sin(\pi z)$ exhibits the
			transform as a product of Gamma factors and a rational
			function, so criterion (iii) applies to them as well; and none
			of the transforms in the table other than the Dirac and
			lognormal rows is of the exponential form $c\,b^{\,s}$,
			each having zeros or poles, so criterion (iii) makes their
			period groups trivial.
			
			In the exponential block the moment functions are exponentials
			of quadratics or products of Gamma factors. For the lognormal
			law with parameters $(m,\sigma^2)$,
			\begin{equation}\label{eq:dictLN}
				\begin{gathered}
					m_a=e^{am+a^2\sigma^2/2},\\
					\Lambda_\mu(s)=m+\sigma^2\Bigl(s-\tfrac12\Bigr),
				\end{gathered}
			\end{equation}
			the increment of the quadratic exponent. For
			$\mathrm{Gamma}(k,\theta)$ the recurrence
			$\Gamma(z)=(z-1)\Gamma(z-1)$ collapses the ratio,
			\begin{equation}\label{eq:dictGamma}
				\begin{gathered}
					m_a=\frac{\theta^a\Gamma(k+a)}{\Gamma(k)},\\
					\Leh_\mu(s)=\theta\,\frac{\Gamma(k+s)}{\Gamma(k+s-1)}
					=\theta(s+k-1);
				\end{gathered}
			\end{equation}
			the exponential law is the case $k=1$, and the chi-square law
			with $k$ degrees of freedom is $\mathrm{Gamma}(\tfrac k2,2)$,
			giving $2(s+\tfrac k2-1)=2s+k-2$. For the chi, Weibull and
			generalized Gamma laws the moment functions are, writing $u$ for
			the moment order to keep it distinct from the scale parameter
			$a$ of the generalized Gamma law,
			\begin{equation}\label{eq:dictChi}
				\begin{gathered}
					m_u=\frac{2^{u/2}\,\Gamma\bigl(\tfrac{k+u}2\bigr)}
					{\Gamma\bigl(\tfrac k2\bigr)},\\
					m_u=\lambda^u\,\Gamma\Bigl(1+\frac uk\Bigr),\\
					m_u=\frac{a^{u}\,\Gamma\bigl(\tfrac{d+u}p\bigr)}
					{\Gamma\bigl(\tfrac dp\bigr)} ,
				\end{gathered}
			\end{equation}
			and in each case the ratio at consecutive orders is the entry of
			Table~\ref{tab:dictionary}.
			
			In the compact block the poles sit to the left of the domain.
			For $\mathrm{Uniform}(0,b)$,
			\begin{equation}\label{eq:dictUnif}
				\begin{gathered}
					m_a=\frac{b^a}{a+1},\\
					\Leh_\mu(s)=\frac{b\,s}{s+1} ,
				\end{gathered}
			\end{equation}
			while for $\mathrm{Beta}(a_0,b_0)$ the ratio telescopes through
			the two Gamma recurrences,
			\begin{equation}\label{eq:dictBeta}
				\begin{gathered}
					m_a=\frac{B(a_0+a,b_0)}{B(a_0,b_0)},\\
					\Leh_\mu(s)=\frac{s+a_0-1}{s+a_0+b_0-1} ,
				\end{gathered}
			\end{equation}
			and for $\mathrm{Kumaraswamy}(a_0,b_0)$ one has
			$m_a=b_0\,B(1+a/a_0,\,b_0)$.
			
			In the heavy block the divergence sits at the right endpoint.
			For $\mathrm{Pareto}(\alpha)$ on $[x_m,\infty)$,
			\begin{equation}\label{eq:dictPareto}
				\begin{gathered}
					m_a=\frac{\alpha x_m^a}{\alpha-a},\\
					\Leh_\mu(s)=x_m\,\frac{\alpha-s+1}{\alpha-s} ,
				\end{gathered}
			\end{equation}
			and for $\mathrm{Beta\text{-}prime}(a_0,b_0)$, whose density is
			proportional to $x^{a_0-1}(1+x)^{-a_0-b_0}$,
			\begin{equation}\label{eq:dictBetaPrime}
				\begin{gathered}
					m_a=\frac{B(a_0+a,\,b_0-a)}{B(a_0,b_0)},\\
					\Leh_\mu(s)=\frac{s+a_0-1}{b_0-s} ;
				\end{gathered}
			\end{equation}
			the $F$ law is the Beta-prime scaled by $d_2/d_1$ with
			$(a_0,b_0)=(d_1/2,d_2/2)$. The inverse-Gamma law telescopes in
			the same way as the Gamma law, with the argument reflected,
			\begin{equation}\label{eq:invgamma}
				\begin{gathered}
					m_a=\frac{\theta^a\Gamma(k-a)}{\Gamma(k)},\\
					\Leh_\mu(s)=\frac{\theta\,\Gamma(k-s)}{\Gamma(k-s+1)}
					=\frac{\theta}{k-s} ,
				\end{gathered}
			\end{equation}
			and the Fr\'echet law has $m_a=\Gamma(1-a/\alpha)$. Two rows
			become trigonometric through Euler's reflection formula. For
			$\mathrm{Log\text{-}logistic}(\beta)$,
			\begin{equation}\label{eq:loglogistic}
				m_a=\Gamma\Bigl(1+\frac a\beta\Bigr)
				\Gamma\Bigl(1-\frac a\beta\Bigr)
				=\frac{\pi a/\beta}{\sin(\pi a/\beta)},
			\end{equation}
			whence
			\begin{equation}\label{eq:loglogisticL}
				\Leh_\mu(s)=\frac{s\,\sin(\pi(s-1)/\beta)}
				{(s-1)\,\sin(\pi s/\beta)},
			\end{equation}
			with a removable singularity at $s=1$; and for the half-Cauchy
			law, where $m_a=\sec(\pi a/2)$ for $|a|<1$,
			\begin{equation}\label{eq:halfcauchyL}
				\Leh_\mu(s)=\frac{\cos(\pi(s-1)/2)}{\cos(\pi s/2)}
				=\tan\Bigl(\frac{\pi s}{2}\Bigr).
			\end{equation}
			The one-sided stable law of index $\alpha$, with Laplace
			transform $e^{-\lambda^\alpha}$, has the classical fractional
			moments
			\begin{equation}\label{eq:dictStable}
				\begin{gathered}
					m_a=\frac{\Gamma(1-a/\alpha)}{\Gamma(1-a)}
					\quad(a<\alpha),\\
					\Leh_\mu(s)=\frac{(1-s)\,\Gamma(1-\tfrac s\alpha)}
					{\Gamma(1+\tfrac{1-s}\alpha)} .
				\end{gathered}
			\end{equation}
			
			The discrete block is where the period column becomes
			nontrivial. For geometric data $\lambda_n=q^{\,n}$ with
			$0<q<1$,
			\begin{equation}\label{eq:dictGeom}
				\begin{gathered}
					Z_\mu(s)=\frac{1}{1-q^{s}}\quad(s>0),\\
					\Leh_\mu(s)=\frac{1-q^{\,s-1}}{1-q^{\,s}} ,
				\end{gathered}
			\end{equation}
			and $\Leh_\mu(s+i\omega)=\Leh_\mu(s)$ holds if and only if
			$q^{i\omega}=1$, so that
			$\Theta_\mu=\tfrac{2\pi}{\log(1/q)}\Z$, the lattice case of
			Proposition~\ref{prop:twosources}(i). For the zeta or Zipf law
			$\Prob(X=n)=n^{-\varrho}/\zeta(\varrho)$ one has
			\begin{equation}\label{eq:dictZipf}
				\begin{gathered}
					Z_\mu(a)=\frac{\zeta(\varrho-a)}{\zeta(\varrho)},\\
					\Leh_\mu(s)=\frac{\zeta(\varrho-s)}{\zeta(\varrho-s+1)} .
				\end{gathered}
			\end{equation}
			
			It remains to see that the Zipf row has trivial period group,
			which the criteria of Theorem~\ref{thm:resonant}(iii) do not
			cover because its transform is a ratio of zeta values rather
			than of Gamma factors. Argue directly. After the translation
			$s\mapsto\varrho-s$, a period $\omega\in\Theta_\mu$ makes
			the exponential of $\log\zeta(s)-\log\zeta(s+1)$ invariant
			under $s\mapsto s+i\omega$ on $\Re s>1$, so the difference
			itself is invariant up to an additive constant in $2\pi i\Z$;
			the constant is independent of $s$ by continuity, and letting
			$\Re s\to\infty$, where $\log\zeta\to0$ along both lines,
			forces it to vanish. The difference is therefore invariant.
			Substituting the
			Euler-product expansion
			$\log\zeta(s)=\sum_p\sum_{k\ge1}p^{-ks}/k$, absolutely
			convergent there, gives
			\begin{equation}\label{eq:zetaperiod}
				\sum_{p}\sum_{k\ge1}
				\frac{\bigl(p^{-ik\omega}-1\bigr)
					\bigl(p^{-ks}-p^{-k(s+1)}\bigr)}{k}=0
			\end{equation}
			identically in $\Re s>1$. Uniqueness of Dirichlet series
			forces $n^{i\omega}=1$ for every prime power $n$, in
			particular for $n=2$ and $n=3$, so that $\omega\log2$ and
			$\omega\log3$ both lie in $2\pi\Z$; since $\log2/\log3$
			is irrational, $\omega=0$.
		\end{proof}
		
		\propabelian*
		
		\begin{proof}
			Finiteness of $Z_P(s)$ for $s<\alpha$ and divergence for
			$s>\alpha$ follow from the tail hypothesis \eqref{eq:regvar} by
			comparison, so $\sup I_P=\alpha$. For $0<s<\alpha$, integration
			by parts gives
			\begin{equation}\label{eq:abelianparts}
				Z_P(s)=s\int_0^\infty x^{s-1}\,\Prob(X>x)\dd x .
			\end{equation}
			Split the integral in identity \eqref{eq:abelianparts} at $1$; the piece
			over $(0,1)$ stays bounded as $s\uparrow\alpha$. Write the tail
			as
			\begin{equation}\label{eq:tailsplit}
				\begin{gathered}
					\Prob(X>x)=c\,x^{-\alpha}\bigl(1+\varepsilon_1(x)\bigr),\\
					\varepsilon_1(x)\longrightarrow0 .
				\end{gathered}
			\end{equation}
			The main term of the decomposition \eqref{eq:tailsplit} gives
			\begin{equation}\label{eq:abelianmain}
				(\alpha-s)\,s\int_1^\infty x^{s-1-\alpha}c\,\dd x
				=sc\longrightarrow\alpha c ,
			\end{equation}
			while for any $\delta>0$, choosing $A$ with
			$|\varepsilon_1|<\delta$ on $(A,\infty)$,
			\begin{equation}\label{eq:abelianerr}
				(\alpha-s)\,s\int_1^\infty x^{s-1-\alpha}c\,|\varepsilon_1(x)|
				\dd x
				\le (\alpha-s)\,sc\!\int_1^A x^{s-1-\alpha}|\varepsilon_1|\dd x
				+\delta\,sc ,
			\end{equation}
			whose first term vanishes as $s\uparrow\alpha$ and whose second
			is arbitrarily small. Hence $(\alpha-s)Z_P(s)\to\alpha c$.
			Dividing by $Z_P(s-1)$ gives the second limit in limit
			\eqref{eq:abelianZ}, once one knows
			$Z_P(s-1)\to Z_P(\alpha-1)\in(0,\infty)$ as $s\uparrow\alpha$.
			This is where the interiority hypothesis
			$\alpha-1\in\mathring I_P$ enters. The partition function is
			finite and continuous on $\mathring I_P$ by
			Lemma~\ref{lem:logconvex}, and $s-1\to\alpha-1$ inside
			$\mathring I_P$, so
			$Z_P(s-1)\to Z_P(\alpha-1)\in(0,\infty)$. Interiority
			cannot be traded for mere finiteness of $Z_P(\alpha-1)$: on
			$\{x<1\}$ the integrands $x^{s-1}$ \emph{decrease} to
			$x^{\alpha-1}$ as $s\uparrow\alpha$, so dominated
			convergence there needs an integrable majorant $x^{s_1-1}$
			with $s_1<\alpha$, which is interiority itself, and
			Remark~\ref{rem:abelianinterior} shows the convergence can
			fail without it. If $Z_P$ continues
			meromorphically across $\Re s=\alpha$, the first limit
			identifies a simple pole of $Z_P$ at $\alpha$ with residue
			$-\alpha c$, and the stated residue of $\Leh_P$ follows.
		\end{proof}
		
		\thmgammachar*
		
		\begin{proof}
			By homogeneity we may take $\theta=1$. Define
			\begin{equation}\label{eq:fdef}
				f(x)\eqdef Z_\mu(x-k),\qquad x>0 .
			\end{equation}
			The hypothesis \eqref{eq:gammahyp} reads
			$Z_\mu(s)=(s+k-1)Z_\mu(s-1)$, which under the substitution
			$s=x-k$ becomes the functional equation
			\begin{equation}\label{eq:gammafunc}
				f(x+1)=x\,f(x),
			\end{equation}
			and $f$ is log-convex by Lemma~\ref{lem:logconvex} with
			$f(1)=Z_\mu(1-k)\in(0,\infty)$. The Bohr--Mollerup
			theorem~\cite{bohr1922}, in the modern exposition
			of~\cite{artin1964}, applies. Its sandwich argument, which is
			the mechanism of Theorem~\ref{thm:rigidity} in a special case,
			runs as follows: for
			$x\in(0,1]$ and $n\ge2$, log-convexity applied to the triples
			$(n-1,n,n+x)$ and $(n,n+x,n+1)$ squeezes
			\begin{equation}\label{eq:bmsandwich}
				(n-1)^{x}f(n)\le f(n+x)\le n^{x}f(n),
			\end{equation}
			while the functional equation \eqref{eq:gammafunc} evaluates
			\begin{equation}\label{eq:bmeval}
				\begin{gathered}
					f(n)=(n-1)!\,f(1),\\
					f(n+x)=x(x+1)\cdots(x+n-1)f(x).
				\end{gathered}
			\end{equation}
			Letting $n\to\infty$ in bound \eqref{eq:bmsandwich} forces
			$f(x)=f(1)\Gamma(x)$ on $(0,1]$, and the functional equation
			propagates the identity to all $x>0$. Hence
			\begin{equation}\label{eq:gammaZ}
				Z_\mu(s)=f(1)\,\Gamma(s+k)\qquad\text{on }(-k,\infty),
			\end{equation}
			which is the partition function of $c\,x^{k-1}e^{-x}\dd x$;
			Lemma~\ref{lem:strip} identifies $\mu$, and undoing the dilation
			restores $\theta$.
		\end{proof}
		
		\thmlognormalchar*
		
		\begin{proof}
			Let $(b,\sigma^2,\Pi)$ be the triplet of $\log X$. By
			Theorem~\ref{thm:levy},
			\begin{equation}\label{eq:lncharderiv}
				\begin{gathered}
					\Lambda_\mu'(s)=\sigma^2+g(s),\\
					g(s)=\int_\R e^{(s-1)y}w(y)\Pi(\dd y),\\
					w(y)=y\bigl(e^y-1\bigr)\ge0,
				\end{gathered}
			\end{equation}
			and the hypothesis \eqref{eq:lognormalhyp} says
			$\Lambda_\mu'\equiv\sigma_0^2$, so $g$ is constant on $\R$.
			Differentiation under the integral sign in the definition of
			$g$ is legitimate because $I_\mu=\R$ makes $g$ finite on all of
			$\R$, so that the bilateral Laplace transform of the positive
			measure $w\dd\Pi$ is finite, hence analytic, on the whole line.
			If $w\dd\Pi\neq0$ then differentiating gives
			\begin{equation}\label{eq:lncharzero}
				g'(s)=\int_\R y\,e^{(s-1)y}w(y)\Pi(\dd y)\equiv0,
			\end{equation}
			which says that the bilateral Laplace transforms of the
			positive measures $(y\vee0)w\dd\Pi$ and
			$((-y)\vee0)w\dd\Pi$ coincide on $\R$. By
			Lemma~\ref{lem:strip}, applied after the substitution
			$x=e^{y}$, these measures are equal; having disjoint supports,
			both vanish. Since $w>0$ off the origin, $\Pi=0$. Hence
			$K_\mu$ is quadratic, $\sigma^2=\sigma_0^2$, and $b=m$ by
			matching; Lemma~\ref{lem:strip} identifies the lognormal.
		\end{proof}
		
		\thmprony*
		
		\begin{proof}
			\emph{Assertion (ii).} By the cocycle identity
			\eqref{eq:cocycle},
			\begin{equation}\label{eq:pronyh}
				\begin{gathered}
					h_j=\frac{Z_\mu(s_0+j)}{Z_\mu(s_0)}
					=\sum_{i=1}^kw_ix_i^{\,j},\\
					w_i=\frac{a_ix_i^{s_0}}{Z_\mu(s_0)}>0,\\
					\sum_iw_i=1,
				\end{gathered}
			\end{equation}
			so the numbers $(h_0,\dots,h_{2k-1})$ are the first $2k$
			moments of the $k$-atomic probability measure
			$\tau=\sum_iw_i\delta_{x_i}$.
			
			\emph{Assertion (i), uniqueness.} If $\tau'$ is another measure
			with at most $k$ atoms and the same $2k$ moments, then the
			signed measure $\tau-\tau'$ has at most $2k$ atoms and
			annihilates all polynomials of degree at most $2k-1$; the
			Vandermonde matrix at the distinct atoms is invertible, so
			$\tau=\tau'$.
			
			\emph{Assertion (i), sharpness.} The $2k-2$ values
			$\Leh_\mu(s_0+1),\dots,\Leh_\mu(s_0+2k-2)$ determine only
			$(h_0,\dots,h_{2k-2})$, that is, $2k-1$ moments, and $2k-1$
			moments do not determine a $k$-atomic measure. It is enough to
			exhibit the failure at $k=2$, where three moments
			$h_0=1$, $h_1$, $h_2$ are prescribed with $h_2>h_1^{2}$. For
			every node $x_1$ in the interval
			$\bigl(0,\,h_1-\sqrt{h_2-h_1^{2}}\,\bigr)$ the system
			\begin{equation}\label{eq:pronysharp}
				\begin{gathered}
					w_1+w_2=1,\\
					w_1x_1+w_2x_2=h_1,\\
					w_1x_1^{2}+w_2x_2^{2}=h_2
				\end{gathered}
			\end{equation}
			has the unique solution
			\begin{equation}\label{eq:pronysharpsol}
				\begin{gathered}
					x_2=\frac{h_2-h_1x_1}{h_1-x_1},\\
					w_2=\frac{h_1-x_1}{x_2-x_1},\\
					w_1=1-w_2,
				\end{gathered}
			\end{equation}
			with $0<x_1<h_1<x_2$ and $w_1,w_2\in(0,1)$. The
			formulas \eqref{eq:pronysharpsol} therefore produce a genuine
			one-parameter family of distinct two-atom probability measures
			sharing $h_0,h_1,h_2$, and pulling them back through
			$a_i\propto w_ix_i^{-s_0}$ gives distinct measures $\mu$ with
			the same $2k-2$ Lehmer values. The general case is the same
			count: the truncated Hamburger--Stieltjes problem with $2k-1$
			prescribed moments, interior to the moment cone, has a
			one-parameter family of $k$-atomic representing measures.
			
			\emph{Assertion (iii).} Let $H=(h_{i+j})_{0\le i\le k-1,\,
				0\le j\le k}$, a $k\times(k+1)$ matrix, and let
			$(c_0,\dots,c_k)$ span its kernel, which is nontrivial by
			dimension count and one-dimensional because $\tau$ has exactly
			$k$ atoms. The relation $Hc=0$ says that
			$\sum_{j=0}^kc_jh_{i+j}=0$ for $0\le i\le k-1$, that is,
			\begin{equation}\label{eq:pronykernel}
				\begin{gathered}
					\sum_{l=1}^{k}w_lx_l^{\,i}\,Q(x_l)=0,\\
					0\le i\le k-1,\\
					Q(x)\eqdef\sum_{j=0}^kc_jx^{j} .
				\end{gathered}
			\end{equation}
			Since the Vandermonde matrix $(x_l^{\,i})$ is invertible and
			$w_l>0$, relation \eqref{eq:pronykernel} forces
			$Q(x_l)=0$ for every $l$: the nodes are the roots of $Q$. The
			weights then solve the Vandermonde system
			$\sum_lw_lx_l^{\,j}=h_j$, and $a_i\propto w_ix_i^{-s_0}$ by
			identity \eqref{eq:pronyh}, the total mass being the one free constant.
		\end{proof}
		
		We turn to the sampling theory. The pointwise limit theorem comes
		first, since the bias expansion and the functional limits both
		build on the same exact linearization.
		
		\thmclt*
		
		\begin{proof}
			By the strong law, $\widehat m_s\to m_s$ and
			$\widehat m_{s-1}\to m_{s-1}>0$ almost surely, so consistency
			follows from continuity of the ratio. The linearization
			identity \eqref{eq:linearize} is exact, and its summands
			$X_i^{\,s-1}(X_i-\Leh_P(s))$ are independent with mean
			\begin{equation}\label{eq:cltmeanzero}
				m_s-\Leh_P(s)\,m_{s-1}=0
			\end{equation}
			and finite variance by the moment hypothesis. The classical
			central limit theorem, together with
			$\widehat m_{s-1}\to m_{s-1}$ and Slutsky's lemma, gives the limit
			\eqref{eq:cltlimit}. For the exponential law of mean $\theta$,
			\begin{equation}\label{eq:expmoments}
				m_a=\theta^{a}\Gamma(1+a),
				\qquad\text{whence}\qquad
				\Leh_P(s)=\frac{m_s}{m_{s-1}}=\theta s ,
			\end{equation}
			and, expanding the square,
			\begin{equation}\label{eq:expvarproof}
				\E\bigl[X^{2s-2}(X-\theta s)^{2}\bigr]
				=\theta^{2s}\bigl[\Gamma(2s+1)-2s\Gamma(2s)+s^2\Gamma(2s-1)\bigr]
				=\theta^{2s}s^{2}\Gamma(2s-1),
			\end{equation}
			using the recurrences
			\begin{equation}\label{eq:gammarecs}
				\begin{gathered}
					\Gamma(2s+1)=2s(2s-1)\Gamma(2s-1),\\
					\Gamma(2s)=(2s-1)\Gamma(2s-1).
				\end{gathered}
			\end{equation}
			Dividing
			identity \eqref{eq:expvarproof} by $m_{s-1}^{2}=\theta^{2s-2}\Gamma(s)^{2}$
			gives formula \eqref{eq:expvar}. Note that the moment hypothesis
			$2s-2\in\mathring I_P=(-1,\infty)$ reads $s>\tfrac12$, which is
			exactly the range in which formula \eqref{eq:expvar} is finite:
			the empirical transform ceases to be $\sqrt n$-consistent
			precisely where its asymptotic variance blows up. This is a
			statement about the estimator and not about the parameter,
			which for the exponential family is $\Leh_P(s)=\theta s$ and is
			$\sqrt n$-estimable at every $s$ from the sample mean.
		\end{proof}
		
		\propbias*
		
		\begin{proof}
			Write $b=m_{s-1}$ and $L=\Leh_P(s)$, and let
			\begin{equation}\label{eq:UVdef}
				\begin{gathered}
					U=\widehat m_s-L\,\widehat m_{s-1},\\
					V=\widehat m_{s-1}-b ,
				\end{gathered}
			\end{equation}
			so that $\E U=\E V=0$ and
			$\widehat\Leh_n(s)-L=U/\widehat m_{s-1}$. Applying the
			algebraic identity
			$1/\widehat m_{s-1}=1/b-V/(b\,\widehat m_{s-1})$ three
			times gives the exact decomposition
			\begin{equation}\label{eq:biasdecomp}
				\widehat\Leh_n(s)-L
				=\frac Ub-\frac{UV}{b^{2}}+\frac{UV^{2}}{b^{3}}
				-\bigl(\widehat\Leh_n(s)-L\bigr)\frac{V^{3}}{b^{3}},
			\end{equation}
			the last term using $U/\widehat m_{s-1}
			=\widehat\Leh_n(s)-L$ once more. Take expectations. The
			first term vanishes. With $u=X^{s}-LX^{s-1}$ and
			$v=X^{s-1}-b$, both centred, independence gives
			\begin{equation}\label{eq:biasUV}
				\E[UV]=\frac{\E[uv]}n
				=\frac{m_{2s-1}-L\,m_{2s-2}}n
				=\frac{m_{2s-2}}n\bigl(\Leh_P(2s-1)-L\bigr),
			\end{equation}
			which is the leading term of identity \eqref{eq:bias}, and
			\begin{equation}\label{eq:biasUVV}
				\E[UV^{2}]=\frac{\E[uv^{2}]}{n^{2}}=O(n^{-2}),
			\end{equation}
			because in the triple sum defining $\E[UV^{2}]$ every term
			containing a lone centred factor has expectation zero; the
			moment $\E|uv^{2}|$ is finite because every exponent
			involved lies in the convex hull of $\{0,4s-4,4s\}$, an
			interval contained in $I_P$.
			
			It remains to bound
			$\E[(\widehat\Leh_n(s)-L)V^{3}]$. Let
			$E=\{\widehat m_{s-1}\ge b/2\}$. On $E$ one has
			$|\widehat\Leh_n(s)-L|\le2|U|/b$, so H\"older's
			inequality gives
			\begin{equation}\label{eq:biasgoodevent}
				\bigl|\E\bigl[(\widehat\Leh_n(s)-L)V^{3}
				\mathbf 1_{E}\bigr]\bigr|
				\le\frac2b\,\|U\|_{4}\,\|V\|_{4}^{3}=O(n^{-2}),
			\end{equation}
			since $\E U^{4}=O(n^{-2})$ and $\E V^{4}=O(n^{-2})$ by the
			Marcinkiewicz--Zygmund inequality under the fourth-moment
			hypothesis $4s,4s-4\in\mathring I_P$. On $E^{c}$ the mean
			$\widehat m_{s-1}$ lies in $(0,b/2)$, so $|V|\le b$ there,
			while the fourth-moment Chebyshev bound gives
			$\Prob(E^{c})\le16\,\E V^{4}/b^{4}=O(n^{-2})$. For
			bounded support one has
			$\widehat\Leh_n(s)\in[m,M]$ for every sample, so the
			contribution of $E^{c}$ is $O(\Prob(E^{c}))=O(n^{-2})$ and
			the first assertion follows. In general, H\"older's
			inequality with the moment hypothesis on the ratio gives
			\begin{equation}\label{eq:biasbadevent}
				\bigl|\E\bigl[(\widehat\Leh_n(s)-L)V^{3}
				\mathbf 1_{E^{c}}\bigr]\bigr|
				\le b^{3}\Bigl(\sup_n\|\widehat\Leh_n(s)\|_{p}
				+L\Bigr)\Prob(E^{c})^{1-1/p}
				=O\bigl(n^{-2+2/p}\bigr),
			\end{equation}
			which dominates the other remainders and gives the second
			assertion.
		\end{proof}
		
		\thmfclt*
		
		\begin{proof}
			Consider the class
			\begin{equation}\label{eq:donskerclass}
				\mathcal F=\bigl\{\tilde f_s(x)=x^{s-1}(x-\Leh_P(s)):
				s\in T\bigr\}
			\end{equation}
			of the linearization identity \eqref{eq:linearize}. The parameter map
			$s\mapsto\tilde f_s(x)$ is differentiable with
			\begin{equation}\label{eq:envelope}
				\begin{gathered}
					\bigl|\partial_s\tilde f_s(x)\bigr|\le
					C_T\bigl(1+|\log x|\bigr)
					\bigl(x^{\overline u}+x^{\underline u}\bigr)\eqdef F(x),\\
					\overline u=\max T,\quad\underline u=\min T-1,
				\end{gathered}
			\end{equation}
			and $\E F(X)^2<\infty$ because
			$2T\cup(2T-2)\subset\mathring I_P$ absorbs the logarithm by the
			domination bound \eqref{eq:domination}. A class Lipschitz in a
			one-dimensional parameter with square-integrable envelope is
			$P$-Donsker~\cite{vaart1996}. Hence
			$n^{-1/2}\sum_i\tilde f_\cdot(X_i)$ converges in
			$\ell^\infty(T)$ to a centered Gaussian process with covariance
			$\E[\tilde f_s\tilde f_t]$ and continuous paths; the uniform
			strong law over $\{x^{s-1}\}$ together with
			$\inf_Tm_{s-1}>0$ lets Slutsky's lemma divide by
			$\widehat m_{s-1}$, giving assertion (i).
			
			In the mixing case, the same linearization gives the exact
			factorization
			\begin{equation}\label{eq:mixingfactor}
				\sqrt n\bigl(\widehat\Leh_n(s)-\Leh_P(s)\bigr)
				=\frac{m_{s-1}}{\widehat m_{s-1}}\cdot
				\frac{1}{\sqrt n}\sum_{i=1}^{n}\psi_s(X_i),
			\end{equation}
			in which the summands form a stationary sequence and the
			first factor is a random scalar. The uniform ergodic theorem
			over the class $\{x^{s-1}:s\in T\}$, together with
			$\inf_Tm_{s-1}>0$, sends the first factor to $1$ uniformly
			on $T$, so by Slutsky's lemma it suffices to prove the
			functional limit for the second factor, and we do so by
			verifying the hypotheses of the bracketing functional
			central limit theorem of Andrews and
			Pollard~\cite{andrewspollard1994} for strongly mixing
			sequences, with their exponents $Q=2$ and $\gamma=\delta$;
			the invariance principle of Doukhan, Massart and
			Rio~\cite{doukhan1995} covers the absolutely regular case
			under weaker entropy conditions. Three hypotheses are to be
			checked. First, the mixing series: with $Q=2$ and
			$\gamma=\delta$ the required condition is
			\begin{equation}\label{eq:apcond}
				\sum_{k\ge1}k^{Q-2}\,
				\alpha_{\mathrm{mix}}(k)^{\gamma/(Q+\gamma)}
				=\sum_{k\ge1}
				\alpha_{\mathrm{mix}}(k)^{\delta/(2+\delta)}<\infty,
			\end{equation}
			which is condition \eqref{eq:mixing} and which the
			polynomial rate $a>1+2/\delta$ supplies. Second, the moment
			condition: each $\tilde f_s$ is dominated by the envelope
			$F$ of bound \eqref{eq:envelope}, and
			$\E F(X)^{2+\delta}<\infty$ because condition
			\eqref{eq:envelopemoment} absorbs the logarithm through the
			domination bound \eqref{eq:domination}; this is the
			required $(Q+\gamma)$-th moment of the envelope. Third, the
			bracketing integral: the class is Lipschitz in the
			one-dimensional parameter $s$ with Lipschitz envelope $F$,
			so its bracketing numbers satisfy
			$N_{[\,]}(\varepsilon)\le C/\varepsilon$, and the
			required integral
			$\int_0^1\varepsilon^{-\gamma/(2+\gamma)}
			N_{[\,]}(\varepsilon)^{1/Q}\dd\varepsilon$ converges
			because $\delta/(2+\delta)+\tfrac12<1$ exactly when
			$\delta<2$, which is the standing restriction. The theorem
			of Andrews and Pollard then yields stochastic
			equicontinuity and the functional limit of the second
			factor, with finite-dimensional laws given by the central
			limit theorem for strongly mixing
			sequences~\cite{ibragimov1962} and covariance the long-run
			kernel \eqref{eq:longrun}, whose absolute convergence
			follows from Davydov's covariance
			inequality~\cite{davydov1968} under
			condition \eqref{eq:mixing}.
		\end{proof}
		\prophoeffding*
		
		\begin{proof}
			The event $\{\widehat\Leh_n(s)\ge y\}$ coincides with
			$\{\sum_ig_y(X_i)\ge0\}$, since the denominator
			$\sum_iX_i^{s-1}$ is positive. The summands are independent,
			bounded with range at most $V_s(y)$, and have common mean
			\begin{equation}\label{eq:gmean}
				\E\bigl[g_y(X)\bigr]=m_s-y\,m_{s-1}
				=-m_{s-1}\bigl(y-\Leh_P(s)\bigr)<0 .
			\end{equation}
			Hoeffding's inequality applied to the deviation
			$\sum_i\bigl(g_y(X_i)-\E g_y(X)\bigr)\ge
			n\,m_{s-1}(y-\Leh_P(s))$ gives the bound
			\eqref{eq:hoeffdingupper}; the bound
			\eqref{eq:hoeffdinglower} is the same argument applied to
			$-g_y$. The estimate for $V_s(y)$ in definition \eqref{eq:oscg} follows
			because the oscillation of a difference is at most the sum of the
			oscillations, and $x\mapsto x^{a}$ is monotone on $[m,M]$.
		\end{proof}
		
		\propsufficient*
		
		\begin{proof}
			\emph{Assertion (i).} It is Theorem~\ref{thm:prony} applied to the
			empirical measure $\widehat\mu_n=\sum_i\delta_{X_i}$, whose total
			mass $n$ is known and fixes the one free constant, together with
			$k\le n$. For assertion (ii),
			$\widehat\mu_n$ has bounded support, so Corollary~\ref{cor:rigid}
			applies: the curve determines $\widehat\mu_n$ up to a positive
			factor, and the known total mass fixes the factor. The curve is
			therefore a measurable bijection of the multiset of observations,
			equivalently of the vector of order statistics, and any bijection
			of a sufficient statistic is sufficient. The order statistic is
			sufficient for every i.i.d.\ model by the factorization criterion,
			and minimal sufficient for the full nonparametric model, whence so
			is the curve.
		\end{proof}
		
		\thmtranslation*
		
		\begin{proof}
			For every order $a$,
			\begin{equation}\label{eq:translZ}
				Z_{P_\eta}(a)=\frac{Z_0(a+\eta)}{Z_0(\eta)},
			\end{equation}
			and in the ratio at consecutive orders the normalization cancels,
			giving identity \eqref{eq:translation}. \emph{Assertion (i).} The
			log-likelihood is
			\begin{equation}\label{eq:loglik}
				\ell_n(\eta)=n\Bigl[\eta\,\overline{\log X}-K_0(\eta)\Bigr]
				+\text{terms free of }\eta,
			\end{equation}
			strictly concave by strict convexity of $K_0$
			(Lemma~\ref{lem:logconvex}); setting the score to zero gives
			equation \eqref{eq:mle}. \emph{Assertion (ii).} The information of a
			one-parameter exponential family in its natural parameter is
			$I(\eta)=K_0''(\eta)=\Var_{P_\eta}(\log X)$, and identity
			\eqref{eq:infoslope} is identity \eqref{eq:jeffreys}. Assertion
			(iii) is the escort-closure computation of
			Section~\ref{sec:characterization}.
		\end{proof}

		\thmphase*
		
		\begin{proof}
			Write $S_n(u)=\sum_iX_i^u$ and $M_n=\max_iX_i$. We record three
			facts.
			
			\emph{(F1)} For $-\delta<u<\alpha$ one has $\E X^u<\infty$, so
			$S_n(u)/n\to m_u$ almost surely and
			$\log S_n(u)=\log n+O_{\Prob}(1)$. This is where the
			hypothesis $\E[X^{-\delta}]<\infty$ enters, and it is needed
			exactly when $u<0$, that is, when $s<1$.
			
			\emph{(F2)} $\log M_n/\log n\to1/\alpha$ in probability: with
			$1-F(x)=x^{-\alpha+o(1)}$,
			\begin{equation}\label{eq:maxbounds}
				\begin{gathered}
					\Prob\bigl(M_n\le n^{(1+\varepsilon)/\alpha}\bigr)
					\ge\bigl(1-n^{-1-\varepsilon+o(1)}\bigr)^n\to1,\\
					\Prob\bigl(M_n\le n^{(1-\varepsilon)/\alpha}\bigr)
					\le e^{-n^{\varepsilon+o(1)}}\to0 .
				\end{gathered}
			\end{equation}
			
			\emph{(F3)} For $u>\alpha$ the power sum is squeezed between
			the largest observation and a truncation at any convergent
			order: for every $\alpha'\in(0,\alpha)$,
			\begin{equation}\label{eq:snsqueeze}
				M_n^{\,u}\;\le\;S_n(u)\;\le\;
				M_n^{\,u-\alpha'}\,S_n(\alpha') ,
			\end{equation}
			so by (F1) and (F2) applied to the bounds \eqref{eq:snsqueeze}
			\begin{equation}\label{eq:squeeze}
				\frac u\alpha\le\liminf\frac{\log S_n(u)}{\log n}
				\le\limsup\frac{\log S_n(u)}{\log n}
				\le\frac{u-\alpha'}\alpha+1
				\quad\text{in probability};
			\end{equation}
			letting $\alpha'\uparrow\alpha$ squeezes
			$\log S_n(u)/\log n\to u/\alpha$.
			
			Now apply (F1) and (F3) to $u=s$ and to $u=s-1$, and read off
			the three regimes in turn. Below the tail index, that is for
			$s<\alpha$, both exponents lie in the convergent range and both
			logarithms are $\log n+O_{\Prob}(1)$, so
			\begin{equation}\label{eq:phase1}
				\log S_n(s)-\log S_n(s-1)=O_{\Prob}(1)
				=o_{\Prob}(\log n)
			\end{equation}
			and the limit is flat at height zero. In the middle window
			$\alpha<s<\alpha+1$ the numerator has passed the index while
			the denominator has not, so
			\begin{equation}\label{eq:phase2}
				\begin{gathered}
					\log S_n(s)=\frac s\alpha\log n+o_{\Prob}(\log n),\\
					\log S_n(s-1)=\log n+o_{\Prob}(\log n),
				\end{gathered}
			\end{equation}
			with difference $(\tfrac s\alpha-1)\log n$ to leading order,
			which is the segment of slope $1/\alpha$. Above the unit shift,
			for $s>\alpha+1$, both exponents have passed the index and
			\begin{equation}\label{eq:phase3}
				\log S_n(s)=\frac s\alpha\log n+o_{\Prob}(\log n),
				\quad
				\log S_n(s-1)=\frac{s-1}\alpha\log n
				+o_{\Prob}(\log n),
			\end{equation}
			with difference $\tfrac1\alpha\log n$ to leading order, which
			is the upper plateau. These three cases are exactly the limit
			\eqref{eq:phasediagram}.
		\end{proof}
		
		\proptailrate*
		
		\begin{proof}
			For $u<\alpha$ the exact Pareto law has
			$\E X^u=\alpha/(\alpha-u)<\infty$, so by the strong law
			\begin{equation}\label{eq:rate1}
				\log S_n(u)=\log n+O_{\Prob}(1).
			\end{equation}
			For $u>\alpha$ the variable $X^u$ satisfies
			$\Prob(X^u>y)=y^{-\alpha/u}$ for $y\ge1$, a Pareto law of index
			$\alpha/u\in(0,1)$; hence $n^{-u/\alpha}S_n(u)$ converges in
			distribution to a strictly positive stable law of index
			$\alpha/u$~\cite{resnick2007}, and since the limit is almost
			surely positive and finite,
			\begin{equation}\label{eq:rate2}
				\log S_n(u)=\frac u\alpha\log n+O_{\Prob}(1).
			\end{equation}
			For $\alpha<s<t<\alpha+1$ both $s-1$ and $t-1$ are less than
			$\alpha$, so combining identity \eqref{eq:rate1} and identity \eqref{eq:rate2},
			\begin{equation}\label{eq:rate3}
				\log\widehat\Leh_n(u)
				=\Bigl(\frac u\alpha-1\Bigr)\log n+O_{\Prob}(1),
				\qquad u\in\{s,t\},
			\end{equation}
			whence
			\begin{equation}\label{eq:rate4}
				\begin{gathered}
					D_n\eqdef\log\widehat\Leh_n(t)-\log\widehat\Leh_n(s)
					=\frac{t-s}{\alpha}\log n\,\bigl(1+\xi_n\bigr),\\
					\xi_n=O_{\Prob}\Bigl(\frac1{\log n}\Bigr).
				\end{gathered}
			\end{equation}
			Therefore
			$\widehat\alpha_{s,t}=(t-s)\log n/D_n=\alpha/(1+\xi_n)$ and
			\begin{equation}\label{eq:rate5}
				\widehat\alpha_{s,t}-\alpha
				=-\alpha\xi_n+O_{\Prob}(\xi_n^2)
				=O_{\Prob}\Bigl(\frac1{\log n}\Bigr).
			\end{equation}
		\end{proof}
		\proprenyihill*
		
		\begin{proof}
			By definition
			\begin{equation}\label{eq:renyidef}
				H_s(\pp)=\frac{1}{1-s}\log\sum_ip_i^{s},
			\end{equation}
			so
			\begin{equation}\label{eq:renyihillproof}
				\begin{gathered}
					(1-s)H_s(\pp)=\log Z_{\pp}(s),\\
					(2-s)H_{s-1}(\pp)=\log Z_{\pp}(s-1),
				\end{gathered}
			\end{equation}
			and subtracting gives identity \eqref{eq:renyi}. For identity
			\eqref{eq:hillnumber}, the cocycle identity \eqref{eq:cocycle} at
			$s=q$ and $k=q-1$, together with $Z_{\pp}(1)=\sum_ip_i=1$, gives
			\begin{equation}\label{eq:hillproof}
				\begin{gathered}
					Z_{\pp}(q)=\prod_{j=2}^{q}\Leh_{\pp}(j),\\
					{}^{q}D=\Bigl(\sum_ip_i^q\Bigr)^{1/(1-q)}
					=Z_{\pp}(q)^{1/(1-q)} .
				\end{gathered}
			\end{equation}
		\end{proof}
		
	\end{appendices}
	
\end{document}